\documentclass[a4paper,reqno,11pt]{amsart}
\usepackage{amssymb, amsfonts}
\usepackage{fancyhdr, mathrsfs, graphicx, epsfig, amsmath, amssymb, amsthm}
\usepackage{float}
\usepackage{enumerate}
\usepackage{amscd}
\usepackage{tikz, multicol}
\usetikzlibrary{chains}

\theoremstyle{plain}
 \newtheorem{theorem}{Theorem}[section]
 \newtheorem{proposition}[theorem]{Proposition}
 \newtheorem{lemma}[theorem]{Lemma}
 \newtheorem{corollary}[theorem]{Corollary}
 \newtheorem{proposed problem}[theorem]{Proposed problem}
 \newtheorem{open problem}[theorem]{Open problem}
 \newtheorem{proposed project}[theorem]{Proposed project}
\theoremstyle{definition}
 \newtheorem{example}[theorem]{Example}
 \newtheorem{definition}[theorem]{Definition}
\theoremstyle{remark}
 \newtheorem{remark}[theorem]{Remark}
 
 \numberwithin{equation}{section}

\title{Geometric structures on the complement of a toric mirror arrangement}
\author{Dali Shen}
\email{dali\_shen@hotmail.com}

\begin{document}

\setcounter{page}{1}
\pagenumbering{arabic}

\maketitle

\begin{abstract}
We study geometric structures on the complement of a toric mirror arrangement associated with a root system.
Inspired by those special hypergeometric functions
found by Heckman-Opdam, as well as the work of Couwenberg-Heckman-Looijenga on geometric structures
on projective arrangement complements,
we consider a family of connections on a total space, namely, a $\mathbb{C}^{\times}$-bundle on the complement of a
toric mirror arrangement (=finite union of hypertori, determined by a root system).
We prove that these connections are torsion free and
flat, and hence define a family of affine structures on the total space, which is equivalent to a family of projective structures
on the toric
arrangement complement. We then determine a parameter region for which the
projective structure admits a locally complex hyperbolic metric. In the end, we find a finite subset of this region for which
the orbifold in question can be biholomorphically mapped onto a Heegner divisor complement of a ball quotient.
\end{abstract}

\tableofcontents

\section{Introduction}

This paper deals with the geometric structures on the complement of a toric mirror arrangement associated with a
root system, mainly growing up from the PhD work of the author \cite{Shen-2015}.
It could be viewed as a natural
toric analogue of the theory of geometric structures on projective arrangement complements studied by Couwenberg, Heckman and
Looijenga in 2005 \cite{Couwenberg-Heckman-Looijenga}. In the 80's of last century, following other people who contributed to the
theory of hypergeometric functions, like Picard, Terada and so on, Deligne and
Mostow studied the monodromy problem of the Lauricella hypergeometric functions $F_{D}$ and gave a complete treatment
on the subject \cite{Deligne-Mostow}\cite{Mostow}, which provides ball quotient structures on $\mathbb{P}^{n}$ minus a
hyperplane configuration of type $A_{n+1}$. Almost in the same year, Barthel, Hirzebruch and H\"{o}fer
investigated Galois
coverings of $\mathbb{P}^2$ which ramify in a configuration of lines and found that $\mathbb{P}^2$ could be uniformized to a
complex ball for certain cases \cite{Barthel-Hirzebruch-Hofer}.
Then some 20 years later, Couwenberg, Heckman and Looijenga developed it to
a more general setting by means of the Dunkl connection, which deals with the geometric structures on projective
arrangement complements. Meanwhile also in the 80's and 90's of last century, Heckman and Opdam introduced and studied a kind of
hypergeometric functions associated with root systems in a series of papers \cite{Heckman-Opdam-1,Heckman-Opdam-2}
\cite{Opdam-1,Opdam-2},
which is actually a multivariable analogue of the classical Euler-Gauss hypergeometric functions. Inspired by these
special hypergeometric functions, we adopt the point of view in \cite{Couwenberg-Heckman-Looijenga} to study the
geometric structures on arrangement complements for the toric situation and we believe this case provides some
interesting examples of ball quotients.

We first in Section $\ref{sec:affine-structures}$ construct a projective structure on a toric arrangement complement.
The basic idea is that we can write a projective
structure on a complex manifold $M$ in terms of an affine structure on $M\times\mathbb{C}^{\times}$. It is well-known that
an affine structure
on a complex manifold is given by a torsion free and flat connection on its (co)tangent bundle, and vice versa. So constructing
a projective structure on $M$ is equivalent to producing a torsion free and flat connection on $M\times\mathbb{C}^{\times}$.
We start with an adjoint torus $H:=\mathrm{Hom}(Q,\mathbb{C}^{\times})$ given by a root lattice
$Q:=\mathbb{Z}R$ where $R$ is a reduced irreducible root system.
Denote the Lie algebra of $H$ by $\mathfrak{h}$ and the Weyl group of $R$ by $W$.
We are also given a toric mirror arrangement associated with a
root system $R$, that is, a finite collection of hypertori each of which is defined by $H_{\alpha}:=\{h\in H \mid e^{\alpha}(h)=1\}$
where $e^{\alpha}$ is a character of $H$. We write $H^{\circ}$ for the complement of the union of these hypertori. Let $\kappa$ be
a $W$-invariant multiplicity parameter for $R$ defined by $\kappa:=(k_{\alpha})_{\alpha\in R}\in \mathbb{C}^{R}$.
Inspired by the
special hypergeometric system constructed by Heckman and Opdam, we consider for $u,v\in\mathfrak{h}$, such a second order
differential operator on $\mathcal{O}_{H^{\circ}}$:
\[ D^{\kappa}_{u,v}:=\partial_{u}\partial_{v}+\frac{1}{2}\sum_{\alpha>0}k_{\alpha}\alpha(u)\alpha(v)
   \frac{e^{\alpha}+1}{e^{\alpha}-1}\partial_{\alpha^{\vee}}+\partial_{b^{\kappa}(u,v)}+a^{\kappa}(u,v) \]
where $\partial_{u}$ denotes the associated translation invariant vector field on $H$ for any $u\in\mathfrak{h}$ and
\[ a^{\kappa}:\mathfrak{h}\times\mathfrak{h}\rightarrow\mathbb{C},
  \quad b^{\kappa}:\mathfrak{h}\times\mathfrak{h}\rightarrow\mathfrak{h} \]
are a $W$-invariant bilinear form and a $W$-equivariant bilinear map respectively.
We want this system to define a projective structure on $H^{\circ}$. That means
for each multiplicity parameter $\kappa$ and each $W$-equivariant bilinear map
$b^{\kappa}$, there exists a $W$-invariant bilinear form $a^{\kappa}$ such that
the system of differential equations $D^{\kappa}_{u,v}f=0$ for all $u,v \in \mathfrak{h}$
is integrable. In order to see the integrability of the system, we treat it from a different
point of view, i.e., the one from the work of Couwenberg-Heckman-Looijenga. Now we associate to these data connections
$\nabla^{\kappa}=\nabla^{0}+\Omega^{\kappa}$ and $\tilde{\nabla}^{\kappa}=\tilde{\nabla}^{0}+\tilde{\Omega}^{\kappa}$
on the cotangent bundles of $H^{\circ}$ and $H^{\circ}\times\mathbb{C}^{\times}$ with
$\Omega^{\kappa}\in \mathrm{Hom}(\Omega_{H^{\circ}},\Omega_{H^{\circ}}\otimes\Omega_{H^{\circ}})$
given by
\begin{equation*}
\Omega^{\kappa}: \zeta\in\Omega_{H^{\circ}}\mapsto\frac{1}{2}\sum_{\alpha>0}
k_{\alpha}\frac{e^{\alpha}+1}{e^{\alpha}-1}\zeta(\partial_{\alpha^{\vee}})d\alpha\otimes d\alpha + (B^{\kappa})^{*}(\zeta)
\end{equation*}
and $\tilde{\Omega}^{\kappa}\in \mathrm{Hom}(\Omega_{H^{\circ}\times\mathbb{C}^{\times}},
\Omega_{H^{\circ}\times\mathbb{C}^{\times}}\otimes\Omega_{H^{\circ}\times\mathbb{C}^{\times}})$
given by
\begin{align*}
\tilde{\Omega}^{\kappa}:
\left\{
\begin{aligned}
\zeta\in\Omega_{H^{\circ}} \mapsto &\frac{1}{2}\sum_{\alpha>0}
k_{\alpha}\frac{e^{\alpha}+1}{e^{\alpha}-1}\zeta(\partial_{\alpha^{\vee}})d\alpha\otimes d\alpha+(B^{\kappa})^{*}(\zeta)\\
            &-\zeta\otimes\frac{dt}{t}-\frac{dt}{t}\otimes\zeta, \\
\frac{dt}{t}\in\Omega_{\mathbb{C}^{\times}} \mapsto &A^{\kappa}-\frac{dt}{t}\otimes\frac{dt}{t}.
\end{aligned}
\right.
\end{align*}
Here $\nabla^{0}$ and $\tilde{\nabla}^{0}$ denote the (flat) translation invariant connections on $H$ and
$H\times\mathbb{C}^{\times}$ respectively, $t$ is the coordinate for $\mathbb{C}^{\times}$, and $A^{\kappa}$ and
$B^{\kappa}$ denote the translation invariant
tensor fields on $H$ or $H\times\mathbb{C}^{\times}$ defined by
$a^{\kappa}$ and $b^{\kappa}$ respectively. We can show that the system defined by
$D^{\kappa}_{u,v}f=0$ for all $u,v\in\mathfrak{h}$ is integrable if and only if
the connection $\tilde{\nabla}^{\kappa}$ given above defines an affine structure, i.e., the connection $\tilde{\nabla}^{\kappa}$ is
torsion free and flat. The torsion freeness of $\tilde{\nabla}^{\kappa}$ comes directly from the torsion freeness of
$\nabla^{\kappa}$ while the flatness of $\tilde{\nabla}^{\kappa}$ needs more effort. In order to check the flatness of
$\tilde{\nabla}^{\kappa}$, we need to invoke a flatness criterion set up by Looijenga \cite{Looijenga-1999}, or by
Kohno \cite{Kohno-1990} at an earlier time. This criterion
requires us to compactify $H^{\circ}\times\mathbb{C}^{\times}$ and compute the residues of $\tilde{\Omega}^{\kappa}$ along those
added mirrors and boundary divisors. Then by applying the criterion to our situation, we can obtain the
conditions for $\tilde{\nabla}^{\kappa}$ being flat. According to these conditions, we can find an appropriate bilinear
form $a^{\kappa}$ so that the connection $\tilde{\nabla}^{\kappa}$ is indeed flat and hence a
$W$-invariant projective structure is constructed on $H^{\circ}$ in terms of $\nabla^{\kappa}$.

We next in Section $\ref{sec:hyperbolic-structures}$ show that the toric arrangement complement $H^{\circ}$
admits a hyperbolic structure when $\kappa$ lies in some certain region
so that its image under the projective evaluation map lands in a complex ball. The basic idea is that we first
identify the monodromy
representation of the system with the reflection representation of the extended affine Artin group,
and thus define a Hermitian form $h$ on the image of
the evaluation map for each $\kappa$ resorting to the reflection representation,
then we can find its hyperbolic region by computing its determinant and show that its dual Hermitian form $h^{*}$
is greater than $0$ (equivalently $h<0$) so that the desired result follows. We first compute the
eigenvalues of the residue endomorphisms of $\tilde{\nabla}^{\kappa}$ along mirrors and boundary divisors
respectively, and a surprising fact is that there
are at most two eigenvalues for each residue endomorphism no matter whether along a mirror
or a boundary divisor. This actually tells us what the local behavior of the evaluation map
looks like for the affine structure
around those divisors. Then we construct the reflection representation of the so-called affine Artin group $\mathrm{Art}(M)$
where $M$ is the affine Coxeter matrix associated with the affine root system $\tilde{R}$ of $R$, and the
extended affine Artin group $\mathrm{Art}'(M)$ ($:=\mathrm{Art}(M)\rtimes (P^{\vee}/Q^{\vee})$) can also be identified with the
fundamental group of the orbifold $W\backslash H^{\circ}$ by Brieskorn's theorem,
hence we can identify the reflection representation with
the monodromy representation of the system accordingly. We further define a Hermitian form $h$ on the corresponding target
space $A$ from the point of view of the reflection representation so that
we can obtain the hyperbolic region for the system by investigating
its determinant. For our situation we can
write out the evaluation map around those subregular points in terms of local coordinates with those
local exponents. Here by subregular points we mean those points lying in one and only one mirror or boundary divisor.
Prepared by these, finally we can prove the dual Hermitian form $h^{*}$ is greater than zero when $\kappa$
lies in the hyperbolic region so that the $\Gamma$-covering of $W\backslash H^{\circ}$ admits a complex ball structure, where
$\Gamma$ stands for the projective monodromy group.

Finally in the last section, since we have already had the local exponents along those reflection mirrors
and boundary divisors on hand,
we can invoke the so-called
Schwarz conditions from \cite{Couwenberg-Heckman-Looijenga} to find all the ball quotients arised in this setting.
This is listed in Table $3$.
On the other hand, a much more ambitious goal is to give each such ball quotient a
modular interpretation, although we are still far away from this for now. But there are already some work on this,
like the Deligne-Mostow theory for type $A_{n}$ and
other two groups, i.e., Allcock, Carlson and Toledo \cite{Allcock-Carlson-Toledo-2002} and Kondo
\cite{Kondo-2000} for type $E_{6}$ and $E_{7}$ respectively. Unfortunately we have to say we barely have any clue for
the other types for the moment, but we explain the
modular interpretation for type $A_{n}$ over here in order to shed some light on this direction.

\vspace{2mm}
\textbf{Acknowledgements.} I would like to thank my PhD supervisors: to Eduard Looijenga for his patient guidance
during my PhD time, including but not only on this work; to Gert Heckman for taking me to walk around in this
beautiful subject.

\section{Projective structures}\label{sec:affine-structures}

In this section we construct a projective structure on a toric arrangement complement $H^{\circ}$. This is equivalent to
constructing an affine structure on $H^{\circ}\times\mathbb{C}^{\times}$, i.e., producing a torsion free flat connection on
$H^{\circ}\times\mathbb{C}^{\times}$. In Section $\ref{subsec:projective-structures}$, we provide a general idea on how to construct such
a desired connection \index{Connection} on $M\times\mathbb{C}^{\times}$ out of a given connection on a complex manifold $M$.
In Section
$\ref{subsec:connections-on-tori}$, following the idea of the preceding section, we do construct such a connection for
$H^{\circ}\times\mathbb{C}^{\times}$,
which is inspired by the work of Heckman and Opdam on special hypergeometric system associated with
\index{Hypergeometric system!associated with a root system}
a root system. In Section $\ref{subsec:flatness-conditions}$, we show the constructed connection on $H^{\circ}\times\mathbb{C}^{\times}$
can be flat as long as we choose an appropriate bilinear form $a^{\kappa}$ for it.

\subsection{Affine and projective structures}\label{subsec:projective-structures}

Let $M$ be a complex manifold of dimension $n$.

\begin{definition}
A $\emph{projective}$ $\emph{structure}$ \index{Projective structure} on $M$ is given by
an atlas of holomorphic charts for which
the transition maps are projective-linear and which is maximal for that property.
Likewise, an $\emph{affine structure}$ \index{Affine structure} on $M$ is
given by an atlas of holomorphic
charts for which the transition maps are affine-linear and which is maximal for that
property.
\end{definition}

So a projective structure on $M$ is
locally modelled on the pair $(\mathbb{P}^{n}, \mathrm{Aut}(\mathbb{P}^{n}))$ of projective
space and projective group and an affine structure is locally modelled on the
pair $(\mathbb{A}^{n}, \mathrm{Aut}(\mathbb{A}^{n}))$ of affine space and affine group.

We recall from \cite{Couwenberg-Heckman-Looijenga}
that an affine structure defines a subsheaf of rank $n+1$ in
the structure sheaf
$\mathcal{O}_{M}$ containing constants, the sheaf of locally affine-linear functions.
The differentials of these make up
a local system on the sheaf $\Omega_{M}$ of differentials on $M$,
and such a local system is given by a holomorphic connection on $\Omega_{M}$,
$\nabla: \Omega_{M}\rightarrow\Omega_{M}\otimes\Omega_{M}$ with extension
$\nabla: \Omega_{M}^{k}\otimes\Omega_{M}\rightarrow\Omega_{M}^{k+1}\otimes\Omega_{M}$
by the Leibniz rule
$\nabla(\omega\otimes\zeta)=d(\omega)\otimes\zeta+(-1)^{k}\omega\wedge\nabla(\zeta)$ for $k\in\mathbb{N}$.
This connection is flat and torsion free. For any connection $\nabla$ on $\Omega_{M}$ its square $\nabla^2:\Omega^{k}_{M}\otimes\Omega_{M}\rightarrow\Omega^{k+2}_{M}\otimes\Omega_{M}$
is a morphism of $\mathcal{O}_{M}$-modules, given by wedging with a section $\mathrm{R}$ of
$\mathrm{End}_{\mathcal{O}_{M}}(\Omega_{M},\Omega^{2}_{M}\otimes\Omega_{M})$, called the curvature of $\nabla$,
and $\nabla$ is flat if and only if $\mathrm{R}=0$. The connection $\nabla$ on $\Omega_{M}$ is also torsion free,
which means that the composite of $\wedge:\Omega_{M}\otimes\Omega_{M}\rightarrow\Omega^{2}_{M}$ with $\nabla:\Omega_{M}\rightarrow\Omega_{M}\otimes\Omega_{M}$
is equal to the exterior derivative $d:\Omega_{M}\rightarrow\Omega^{2}_{M}$.
Indeed flat differentials in $\Omega_{M}$ are then closed,
and provide by the Poincar\'{e} lemma a subsheaf of $\mathcal{O}_{M}$ of rank $n+1$ containing constants.
We refer to Deligne's lecture notes for an excellent exposition of the language of connections and more
\cite{Deligne-1970}. Conversely, a torsion free flat connection on the cotangent \index{Connection!flat}
\index{Connection!torsion free}
bundle of $M$ defines an affine structure on $M$.

A projective structure can also be described in terms of a connection, at least locally.
Let us first observe that
such a structure on $M$ defines locally a tautological $\mathbb{C}^{\times}$-bundle $\pi: L\rightarrow M$
whose total space has an affine structure and for which scalar multiplication respects the affine structure.
This local $\mathbb{C}^{\times}$-bundle is unique up to scalar multiplication and needs not be globally defined.

We write a projective structure on $M$ in terms of an affine structure on $M\times\mathbb{C}^{\times}$ in the
following proposition.

\begin{proposition} \label{prop:affine-projective}
Let $M$ be a complex manifold endowed with a holomorphic connection $\nabla:\Omega_{M}\rightarrow\Omega_{M}\otimes\Omega_{M}$ on its cotangent bundle.
Suppose the connection $\nabla$ is torsion free and a connection $\tilde{\nabla}$ on $M\times\mathbb{C}^{\times}$ is given by
\begin{align*}
\tilde{\nabla}(\zeta)&=\nabla(\zeta)-\zeta\otimes\frac{dt}{t}-\frac{dt}{t}\otimes\zeta,\\
\tilde{\nabla}(\frac{dt}{t})&=A-\frac{dt}{t}\otimes\frac{dt}{t},
\end{align*}
with $\zeta\in\Omega_{M}$ and $t$ the coordinate on $\mathbb{C}^{\times}$,
where $A$ is a symmetric section of $\Omega_{M}\otimes\Omega_{M}$. Then this connection $\tilde{\nabla}$ defines
an affine structure on $M\times\mathbb{C}^{\times}$ if and only if
the curvature of $\nabla$, viewed as a $\mathcal{O}_{M}$-homomorphism
$\nabla\nabla$: $\Omega_{M}\rightarrow\Omega_{M}^{2}\otimes\Omega_{M}$, is given by wedging from the right
with the symmetric section $-A$ of $\Omega_{M}\otimes\Omega_{M}: \zeta\mapsto -\zeta\wedge A$.

\end{proposition}

\begin{proof}
Put $L:=M\times\mathbb{C}^{\times}$ and denote by $\pi: L\rightarrow M$ and $t: L\rightarrow\mathbb{C}^{\times}$
the projections. Then we have
$$\Omega_{L}\cong\pi^{\ast}\Omega_{M}\oplus\mathcal{O}_{L}\frac{dt}{t}$$
and regard the natural map $\Omega_{M}\rightarrow\pi_{\ast}\pi^{\ast}\Omega_{M}$
as an inclusion. We have to check the flatness and torsion freeness of the
connection $\tilde{\nabla}:\Omega_{L}\rightarrow\Omega_{L}\otimes\Omega_{L}$ defined above.

We first check the flatness. Observe that for $\omega,\zeta\in\Omega_{M}$ we have
\begin{align*}
\tilde{\nabla}(\omega\otimes\zeta)&=d\omega\otimes\zeta-
\omega\wedge(\nabla(\zeta)-\zeta\otimes\frac{dt}{t}-\frac{dt}{t}\otimes\zeta)\\
&=\nabla(\omega\otimes\zeta)+(\omega\wedge\zeta)\otimes\frac{dt}{t}-\frac{dt}{t}\wedge(\omega\otimes\zeta)
\end{align*}
which in turn implies
\begin{align*}
\tilde{\nabla}(\nabla(\zeta))&=
\nabla^2(\zeta)+(\wedge\nabla)(\zeta)\otimes\frac{dt}{t}-\frac{dt}{t}\wedge\nabla(\zeta)\\
&=\nabla^2(\zeta)+d\zeta\otimes\frac{dt}{t}-\frac{dt}{t}\wedge\nabla(\zeta)
\end{align*}
since $\nabla$ is torsion free by assumption. Hence we get for $\zeta\in \Omega_{M}$
\begin{align*}
\tilde{\nabla}^2(\zeta)=&\tilde{\nabla}(\nabla(\zeta)-\zeta\otimes\frac{dt}{t}-\frac{dt}{t}\otimes\zeta)\\
=&\nabla^2(\zeta)+d\zeta\otimes\frac{dt}{t}-\frac{dt}{t}\wedge\nabla(\zeta)-d\zeta\otimes\frac{dt}{t}+
   \zeta\wedge(A-\frac{dt}{t}\otimes\frac{dt}{t})\\
   &+\frac{dt}{t}\wedge(\nabla(\zeta)-\zeta\otimes\frac{dt}{t}-\frac{dt}{t}\otimes\zeta)\\
=&\nabla^2(\zeta)+\zeta\wedge A.
\end{align*}
So $\tilde{\nabla}^2(\zeta)=0$ if and only if $\nabla^2(\zeta)=-\zeta\wedge A$.

Observe that the above definition of $\tilde{\nabla}(dt/t)$ is equivalent to
$\tilde{\nabla}(dt)=tA$, and so we get
\begin{align*}
\tilde{\nabla}^2(dt)&=\tilde{\nabla}(tA)=dt\wedge A+t\tilde{\nabla}(A)\\
&=dt\wedge A+t(\nabla(A)+(\wedge A)\otimes\frac{dt}{t}-\frac{dt}{t}\wedge A)\\
&=t\nabla(A)=0
\end{align*}
because $A$ is symmetric, and using the Bianchi identity $\nabla(A)=0$.

The verification that $\tilde{\nabla}$ is torsion free is easy. Indeed for $\zeta\in\Omega_{M}$
\begin{align*}
&\wedge\tilde{\nabla}(\zeta)=\wedge\nabla(\zeta)-\zeta\wedge\frac{dt}{t}-\frac{dt}{t}\wedge\zeta=d\zeta\\
&\wedge \tilde{\nabla}(dt)=t(\wedge A)=0
\end{align*}
since $\nabla$ is torsion free and $A$ is symmetric.
This completes the proof of the proposition.

\end{proof}

\begin{lemma} \label{lem:local-affine-functions}
As assumed in the above proposition, the local
affine functions on $M\times\mathbb{C}^{\times}$ are of the form $c+tf$, with $f\in\mathcal{O}_{M}$ satisfying $\nabla(df)+fA=0$
and $c\in\mathbb{C}$ a constant.
\end{lemma}

\begin{proof}
For $\varphi\in\mathcal{O}_{L}$ of the form $\sum f_{k}t^{k}$ with $f_{k}\in\mathcal{O}_{M}$ we get
\begin{align*}
\tilde{\nabla}(d\varphi)=&\sum\tilde{\nabla}(t^{k}df_{k}+kf_{k}t^{k}\frac{dt}{t})\\
=&\sum t^{k}(k\frac{dt}{t}\otimes df_{k}+\nabla(df_{k})-df_{k}\otimes\frac{dt}{t}-\frac{dt}{t}\otimes df_{k})\\
    &+\sum t^{k}(kdf_{k}\otimes\frac{dt}{t}+k(k-1)f_{k}\frac{dt}{t}\otimes \frac{dt}{t}+kf_{k}A)\\
=&\sum t^{k}(\nabla(df_{k})+kf_{k}A)+\sum k(k-1)t^{k}f_{k}\frac{dt}{t}\otimes \frac{dt}{t}\\
   &+\sum (k-1)t^{k}(df_{k}\otimes\frac{dt}{t}+\frac{dt}{t}\otimes df_{k})=0
\end{align*}
if and only if $f_{k}=0$ for $k\neq0,1$ and $f_{0},f_{1}\in\mathcal{O}_M$ are solutions of
\begin{eqnarray*}
df_{0}=0\;,\;\nabla(df_{1})+f_{1}A=0\;.
\end{eqnarray*}
\end{proof}

Given a projective structure on $M$ the pair $(\nabla, A)$ of a torsion free connection $\nabla$ on
$\Omega_{M}$ whose curvature is given by $\zeta\mapsto -\zeta\wedge A$ with $A$ a symmetric section of
$\Omega_{M}\otimes\Omega_{M}$ is not unique, because the way defining $\tilde{\nabla}$ produces not just the tautological
line bundle, but also a trivialization $t$. Let us see how this changes if we choose another
local trivialization $t^{\prime}$. Write $t^{\prime}=te^{g}$, with $g\in\mathcal{O}_{M}$.
From $\frac{dt^{\prime}}{t^{\prime}}=\frac{dt}{t}+dg$, we see that
$$\tilde{\nabla}(\zeta)=\nabla^{\prime}(\zeta)-\zeta\otimes\frac{dt^{\prime}}{t^{\prime}}
-\frac{dt^{\prime}}{t^{\prime}}\otimes\zeta$$ with
$$\nabla^{\prime}(\zeta):=\nabla(\zeta)+dg\otimes\zeta+\zeta\otimes dg.$$
Furthermore,
\begin{align*}
\tilde{\nabla}(\frac{dt^{\prime}}{t^{\prime}})&=\tilde{\nabla}(\frac{dt}{t}+dg)\\
&=A-\frac{dt}{t}\otimes\frac{dt}{t}+\nabla(dg)-dg\otimes\frac{dt}{t}-\frac{dt}{t}\otimes dg\\
&=A+\nabla(dg)+dg\otimes dg-\frac{dt^{\prime}}{t^{\prime}}\otimes\frac{dt^{\prime}}{t^{\prime}}\\
&=A^{\prime}-\frac{dt^{\prime}}{t^{\prime}}\otimes\frac{dt^{\prime}}{t^{\prime}}
\end{align*}
with
$$A^{\prime}:=A+\nabla(dg)+dg\otimes dg.$$

It is worthwhile to write out the content of the above lemma in local coordinates
$z=(z^1,\cdots,z^n)$ on $M$. Let $\nabla^0:\Omega_M\rightarrow\Omega_M\otimes\Omega_M$
be the connection defined by $\nabla^0(dz^k)=0$ for all $k$.

\begin{corollary} \label{cor:local-coordinates}
In these local coordinates let $\nabla=\nabla^0+\Omega:\Omega_M\rightarrow\Omega_M\otimes\Omega_M$
be a connection on $\Omega_M$, so $\Omega:\Omega_M\rightarrow\Omega_M\otimes\Omega_M$
is a morphism of $\mathcal{O}_M$-modules and $\Omega(dz^k)=\sum\Gamma^k_{ij}dz^i\otimes dz^j$
with $\Gamma^k_{ij}$ the connection coefficients of $\nabla$.
Let $A$ be a quadratic differential on $M$, so $A$ is a symmetric section
of $\Omega_M\otimes\Omega_M$ given in the local coordinates as
$A=\sum A_{ij}dz^i\otimes dz^j$ with $A_{ij}=A_{ji}$ for all $1\leq i,j\leq n$.
Then the linear system of second order differential equations $\nabla(df)+fA=0$
for $f\in\mathcal{O}_M$ takes in these local coordinates the explicit form
\[ (\partial_i\partial_j+\sum \Gamma^k_{ij}\partial_k+A_{ij})f=0 \]
for all $1\leq i,j\leq n$. It has local solution space of dimension at most $n+1$ with
equality if and only if the connection $\nabla$ is torsion free and its curvature $\mathrm{R}$
is given by $\Omega_M\ni\zeta\mapsto-\zeta\wedge A\in\Omega^2_M\otimes\Omega_M$.
In these local coordinates $\nabla$ is torsion free if and only if $\Gamma^k_{ij}=\Gamma^k_{ji}$
for all $1\leq i,j,k\leq n$, and $\mathrm{R}(\zeta)=-\zeta\wedge A$ for all $\zeta\in\Omega_M$ if and only if
$2\mathrm{R}^k_{lij}=\delta^k_jA_{il}-\delta^k_iA_{jl}$ for all $1\leq i,j,k,l\leq n$ with
$\delta$ the Kronecker symbol and
\[ \mathrm{R}^k_{lij}=(\partial_i\Gamma^k_{lj}-\partial_j\Gamma^k_{li})+
\sum(\Gamma^k_{mi}\Gamma^m_{lj}-\Gamma^k_{mj}\Gamma^m_{li}) \]
the coefficients of the curvature matrix $\mathrm{R}^k_l=\sum \mathrm{R}^k_{lij}dz^i\wedge dz^j$ of the
curvature $\mathrm{R}$ in the basis $dz^l$.
\end{corollary}

\begin{proof}
In these local coordinates we have $df=\sum(\partial_jf)dz^j$ for $f\in\mathcal{O}_M$ and so
$\nabla^{0}(df)=\sum\partial_i\partial_j(f)dz^i\otimes dz^j$ and hence $\nabla(df)+fA=0$ amounts to
\[ \sum(\partial_i\partial_jf+\sum \Gamma^k_{ij}\partial_kf+A_{ij}f)dz^i\otimes dz^j=0 \]
which yields the above linear system of second differential equations. The connection $\nabla$
is torsion free if and only if $\wedge\nabla=d$ which amounts to $\sum\Gamma^k_{ij}dz^i\wedge dz^j=0$
or equivalently $\Gamma^k_{ij}=\Gamma^k_{ji}$ for all $1\leq i,j,k\leq n$.
The curvature $\mathrm{R}$ of $\nabla$ sends $dz^k$ to the element $\sum \mathrm{R}^k_{lij}(dz^i\wedge dz^j)\otimes dz^l$
and the condition that $\mathrm{R}=\nabla^2:\Omega_M\rightarrow\Omega^2_M\otimes\Omega_M$ is just equal to
$\zeta\mapsto-\zeta\wedge A$ for $\zeta\in\Omega_M$ amounts to
\[ \mathrm{R}(dz^k)=\sum \mathrm{R}^k_{lij}(dz^i\wedge dz^j)\otimes dz^l=\sum A_{il}(dz^i\wedge dz^k)\otimes dz^l=-dz^k\wedge A \]
for all $1\leq k\leq n$ and so
\[ \sum \mathrm{R}^k_{lij}(dz^i\wedge dz^j)=\sum A_{il}(dz^i\wedge dz^k) \]
for all $1\leq k,l\leq n$. Contraction with the vector field $\partial_m$ yields
\begin{align*}
\sum 2\mathrm{R}^k_{lmj}dz^j&=\sum \mathrm{R}^k_{lij}(\delta^i_mdz^j-\delta^j_mdz^i)\\
&=\sum A_{il}(\delta^i_mdz^k-\delta^k_mdz^i)\\
&=A_{ml}dz^k-\sum\delta^k_mA_{il}dz^i\\
&=\sum(\delta^k_jA_{ml}-\delta^k_mA_{jl})dz^j
\end{align*}
for all $1\leq k,l,m\leq n$. Hence the condition for the relation $\mathrm{R}(\zeta)=-\zeta\wedge A$ becomes
\[ 2\mathrm{R}^k_{lij}=\delta^k_jA_{il}-\delta^k_iA_{jl} \]
for all $1\leq i,j,k,l\leq n$.
\end{proof}

\subsection{Differential operators and connections on tori}\label{subsec:connections-on-tori}

We consider the situation discussed above in the special case where the underlying complex manifold
is an algebraic torus.

Let $\mathfrak{a}$ be a real vector space of dimension $n$ endowed with an inner product $(\cdot,\cdot)$ (making it a
Euclidean vector space). The inner product identifies $\mathfrak{a}$ with its dual $\mathfrak{a}^{\ast}$, so that
the latter also is endowed with an inner product, by abuse of notation still denoted by $(\cdot,\cdot)$. Suppose also given  a reduced irreducible finite root system $R\subset \mathfrak{a}^{\ast}$. Then the
corresponding orthogonal reflection for each $\alpha\in R$
\[ s_{\alpha}(\beta)=\beta-\frac{2(\beta,\alpha)}{(\alpha,\alpha)}\alpha, \quad
\beta\in \mathfrak{a}^{\ast}\]
preserves the set $R$ and the crystallographic condition
\[ \frac{2(\beta,\alpha)}{(\alpha,\alpha)}\in \mathbb{Z} \]
holds for all $\alpha,\beta\in R$.
Let $Q=\mathbb{Z}R$ denote
 the root lattice \index{Lattice!root} in $\mathfrak{a}^{\ast}$ and denote the corresponding
 dual root system in $\mathfrak{a}$ by $R^{\vee}$.
We then have the coweight lattice \index{Lattice!coweight}
$P^{\vee}=\mathrm{Hom}(Q,\mathbb{Z})$
of $R^{\vee}$ in $\mathfrak{a}$. Hence
\[ H=\mathrm{Hom}(Q,\mathbb{C}^{\times}) \]
is a so-called adjoint algebraic torus with
(rational) character lattice $Q$. \index{Lattice!character}

We denote the Lie algebra of $H$ by $\mathfrak{h}$, so $\mathfrak{h}=\mathbb{C}\otimes\mathfrak{a}$
and $H=\mathfrak{h}/2\pi\sqrt{-1}P^{\vee}$ as a complex torus. For $v\in\mathfrak{h}$ we denote by
$\partial_{v}$ the associated translation invariant vector field on $H$. Likewise, if we are given
$\phi\in\mathfrak{h}^{\ast}$, then we denote by $d\phi$ the associated translation invariant
differential on $H$. In case $\phi$ determines a character of $H$ (meaning $\phi\in Q$), we denote that character by
$e^{\phi}$. If $\exp:\mathfrak{h}\rightarrow H=\mathfrak{h}/2\pi\sqrt{-1}P^{\vee}$ is the exponential map with
the inverse $\log:H\rightarrow \mathfrak{h}$, then we have $e^{\phi}(h)=e^{\phi(\log h)}$ for all
$h\in H$. We also have $d\phi=(e^{\phi})^{\ast}(\frac{dt}{t})$ with $t$ the coordinate on $\mathbb{C}^{\times}$.
We denote the (flat)
translation invariant connections on $H$ and $H\times\mathbb{C}^{\times}$ by
$\nabla^{0}$ and $\tilde{\nabla}^{0}$ respectively
(so that $\partial_{v}=\nabla_{\partial_{v}}^{0}$).

Each $\alpha$ in $R$ determines a character $e^{\alpha}$,
then $R$ generates
the character lattice $Q$ and each element of $R$ is primitive in $Q$.
So the set $R_{+}:=R/\pm$ of antipodal pairs in $R$ indexes in one-one manner
the kernels of these characters. The kernel $H_{\alpha}=\{ h\in H\mid e^{\alpha}(h)=1\}$, also called the
\emph{mirror} associated with the root $\alpha$,
has its Lie algebra $\mathfrak{h}_{\alpha}$ which is the zero set of $\alpha$.
We call the finite collection of these hypertori $H_{\alpha}$'s a $\emph{toric mirror arrangement}$ associated with a
\index{Toric arrangement}
root system $R$, sometimes also called a \emph{toric arrangement} in this paper if no confusion would arise.
We write $H^{\circ}$ for the complement \index{Toric arrangement complement}
of the union of these hypertori as follows:
\[ H^{\circ}:=H-\cup_{\alpha\in R_{+}}H_{\alpha}. \]

Let $K$ be the space of multiplicity parameters for $R$ defined as the space of $W$-invariant functions
\[  \kappa=(k_{\alpha})_{\alpha\in R}\in \mathbb{C}^{R} \]
where $W$ is the Weyl group \index{Weyl group}
generated by all reflections $s_{\alpha}$. We shall sometimes write $k_{i}$ instead of $k_{\alpha_{i}}$
if $\alpha_{1},\cdots,\alpha_{n}$ is a basis of simple roots in $R_{+}$.
It is clear that $K$ is isomorphic to $\mathbb{C}^{r}$ as a $\mathbb{C}$-vector space if $r$ is the number
of $W$-orbits in $R$ (i.e., $r=1$ or $2$). Hence for convenience, we sometimes
also write $k$ for $k_{1}$ and $k'$ for $k_{n}$ if
$\alpha_{n}\notin W\alpha_{1}$ when no confusion can arise. But note that $k'$ has a different meaning for
type $A_{n}$, which can be seen from Remark $\ref{rem:k'-for-An}$.

We also have
given for each $\alpha\in R$ a nonzero $\emph{coroot}$ \index{Coroot} $\alpha^{\vee}\in\mathfrak{h}$
such that
$(-\alpha)^{\vee}=-\alpha^{\vee}$ and $\alpha^{\vee}(\alpha)=2$. Let
\[ a^{\kappa}: \mathfrak{h}\times\mathfrak{h}\rightarrow\mathbb{C}, \quad
b^{\kappa}: \mathfrak{h}\times\mathfrak{h}\rightarrow\mathfrak{h}  \]
be a symmetric bilinear form and a symmetric bilinear map respectively,
which are invariant and equivariant under the $W$ action respectively.
We notice that $a^{\kappa}$ is just a multiple of the given inner product
$(\cdot,\cdot)$ by the Schur's lemma since $R$ is irreducible.

\begin{lemma}\label{lem:bmap}
If $R$ is irreducible then $b^{\kappa}$ vanishes unless $R$ is of type $A_{n}$ for
$n\geq 2$ in which case there exists a $k'\in\mathbb{C}$ such that
\[ b^{\kappa}(u,v)=\frac{1}{2}k'\sum_{\alpha >0}\alpha(u)\alpha(v)\alpha'  \quad
\text{for any} \; u,v\in \mathfrak{h} \]
with $\alpha'=\varepsilon_{i}+\varepsilon_{j}-\frac{2}{n+1}\sum_{l}\varepsilon_{l}$
if we take the construction of $\alpha$ from Bourbaki \cite{Bourbaki}:
$\alpha=\varepsilon_{i}-\varepsilon_{j}$ for $1\leq i<j \leq n+1$.
\end{lemma}

\begin{proof}
We write $b$ for $b^{\kappa}$ if no confusion arises. Obviously we can
identify $b$ with an element of
$\mathrm{Hom}(\mathfrak{a},\mathrm{Sym}^{2}(\mathfrak{a}^{\ast}))^{W}$.
First fix a positive definite generator $g$ of
$\mathrm{Sym}^{2}(\mathfrak{a}^{\ast})^{W}$ and then choose a line $L\subset\mathfrak{a}$
such that its $g$-orthogonal complement $H\subset\mathfrak{a}$ is a hyperplane
for which $R_{H}:=R\cap H$ spans $H$ and is an irreducible root system. We
decompose $\mathrm{Sym}^{2}(\mathfrak{a}^{\ast})=\mathrm{Sym}^{2}(H^{*})\oplus(L^{*}
\otimes H^{*})\oplus(H^{*}\otimes L^{*})\oplus(L^{*})^{\otimes 2}$. If we consider the
$W(R_{H})$-invariant part, the
middle two summands immediately become trivial since $(H^{*})^{W(R_{H})}=0$. Then
we have
$\mathrm{Sym}^{2}(\mathfrak{a}^{*})^{W(R_{H})}=\mathbb{R}g_{L}\oplus\mathbb{R}g$
where $g=g_{H}+g_{L}$ and $g_{H}$ resp. $g_{L}$ is the restriction of $g$ on $H$ resp. $L$.

Let $f\in \mathrm{Hom}(\mathfrak{a},\mathrm{Sym}^{2}(\mathfrak{a}^{\ast}))^{W}$, then
$f(v)=\mu g_{L}+\lambda g$ for some $v\in L$ since $L$ belongs to the $W(R_{H})$-invariant
part. Assume that there exists a $w$ such that $w(v)=-v$. Since $w$ preserves both
$g_{L}$ and $g$, we must have $\mu=\lambda=0$ by the linearity of $f$ and thus
$f(v)=0$. Since the $W$-orbit of $v$ spans $V$, it follows that $f=0$.

This assumption is certainly satisfied when $-1\in W$. Let's consider the remaining
cases: $A_{n\geq 2}$, $D_{\mathrm{odd}}$ and $E_{6}$. For $E_{6}$, we take $v$ to be a root,
then $R_{H}$ is of type $A_{5}$. For $D_{\mathrm{odd}}$, we take $v$ perpendicular to a
subsystem of type $D_{n-1}$, then there is a $w$ whose restriction to $H$ is a
reflection (in terms of the construction in Bourbaki: $v=\varepsilon_{1}$ and
$w=s_{\varepsilon_{1}-\varepsilon_{2}}s_{\varepsilon_{1}+\varepsilon_{2}})$.

When $R$ is of type $A_{n\geq 2}$, we use the construction in Bourbaki again:
$\mathfrak{a}$ is the hyperplane in $\mathbb{R}^{n+1}$ defined by
$\sum_{i=1}^{n+1}x_{i}=0$. Put $\bar{x}_{i}:=x_{i}|\mathfrak{a}$ so that
$\sum_{i}\bar{x}_{i}=0$. Let $\bar{\varepsilon}_{i}\in\mathfrak{a}$ be the
orthogonal projection of $\varepsilon_{i}\in\mathbb{R}^{n+1}$ in $\mathfrak{a}$.
The orthogonal complement of $\varepsilon_{i}$ in $\mathfrak{a}$ is spanned by
a subsystem of type $A_{n-1}$. Note that all the $\bar{\varepsilon}_{i}$'s make
up a $W$-orbit with sum zero. So if we write $f(\bar{\varepsilon}_{i})=\mu\bar{x}_{i}^{2}
+\lambda g$, sum them up, we get $0=\sum_{i=1}^{n+1}f(\bar{\varepsilon}_{i})=
\mu\sum_{i=0}^{n+1}\bar{x}_{i}^{2}+(n+1)\lambda g$. Hence we have
$f(\bar{\varepsilon}_{i})=\mu(\bar{x}_{i}^{2}-\frac{1}{n+1}\sum_{i=1}^{n+1}\bar{x}_{i}^{2})$.
This indeed defines an element of
$\mathrm{Hom}(\mathfrak{a},\mathrm{Sym}^{2}(\mathfrak{a}^{*}))^{W}$ and
we thus have
$\mathrm{dim}(\mathrm{Hom}(\mathfrak{a},\mathrm{Sym}^{2}(\mathfrak{a}^{*}))^{W})=1$.

Let $b_{0}(u,v)=\sum_{\alpha >0}\alpha(u)\alpha(v)\alpha'$. Since $w(\alpha')=w(\alpha)'$
we have $wb_{0}(u,v)=b_{0}(wu,wv)$ for all $u,v\in\mathfrak{h}$ and $w\in
W(A_{n})=\mathfrak{S}_{n+1}$. Then we see that $b_{0}$ is a generator of
$\mathrm{Hom}(\mathrm{Sym}^{2}\mathfrak{h},\mathfrak{h})^{W}$.
\end{proof}

\begin{remark}\label{rem:k'-for-An}
In fact, for type $A_{n}$, another generator is obtained by taking $v\in\mathfrak{a}
\mapsto\partial_{v}\sigma_{3}|\mathfrak{a}$ where
$\bar{\sigma}_{3}:=\sigma_{3}|\mathfrak{a}$ is an element of
$(\mathrm{Sym}^{3}(\mathfrak{a}^{*}))^{W}$. This point will become more clear
when we discuss the toric Lauricella case in Example $\ref{exm:toric-Lauricella}$.

And because $b^{\kappa}$ exists for type $A_{n}$, we would like to include
$k'$ in $\kappa$ for type $A_{n}$.
\end{remark}

We want to define a $W$-invariant projective structure on $H^{\circ}$. Then there exists an integrable system
of second order differential equations on $H^{\circ}$ according to Corollary $\ref{cor:local-coordinates}$. Inspired by
this, we make an ansatz on the second order differential operator. Besides the system should be of the most general
$W$-invariant form, we also hope the system is asymptotically free along
the mirror $H_{\alpha}$ and it has regular singularities along $H_{\alpha}$.

Then we define the vector fields
\[   X_{\alpha}:=\frac{e^{\alpha}+1}{e^{\alpha}-1}\partial_{\alpha^{\vee}}  \]
(notice that $X_{-\alpha}=X_{\alpha}$) and consider for $u,v\in\mathfrak{h}$,
such a second order differential operator on $\mathcal{O}_{H^{\circ}}$ defined by
\[ D^{\kappa}_{u,v}:=\partial_{u}\partial_{v}+\frac{1}{2}\sum\limits_{\alpha>0}
k_{\alpha}\alpha(u)\alpha(v)X_{\alpha}+\partial_{b^{\kappa}(u,v)}+a^{\kappa}(u,v). \]
It adds to the main linear second order term
a lower order perturbation which consists of a $W$-equivariant first order term and a $W$-invariant constant. Notice that $wD^{\kappa}_{u,v}w^{-1}=D^{\kappa}_{wu,wv}$ where $w\in W$.

We want this system to define a projective structure on $H^{\circ}$. That means
for each multiplicity parameter $\kappa$ and each equivariant bilinear map
$b^{\kappa}$ as above there exists a $W$-invariant bilinear form $a^{\kappa}$ such that
the system of differential equations $D^{\kappa}_{u,v}f=0$ for all $u,v \in \mathfrak{h}$
is integrable. It is obvious that
this projective structure is invariant under the action of $W$.

We associate to these data connections
$\nabla^{\kappa}=\nabla^{0}+\Omega^{\kappa}$
on the cotangent bundle of $H^{\circ}$ with
$\Omega^{\kappa}\in \mathrm{Hom}(\Omega_{H^{\circ}},\Omega_{H^{\circ}}\otimes\Omega_{H^{\circ}})$
given by
\begin{equation}\label{eqn:projective-connection}
\Omega^{\kappa}: \zeta\in\Omega_{H^{\circ}}\mapsto\frac{1}{2}\sum\limits_{\alpha>0}
k_{\alpha}\zeta(X_{\alpha})d\alpha\otimes d\alpha + (B^{\kappa})^{*}(\zeta).
\end{equation}
Then taking the cue from Proposition $\ref{prop:affine-projective}$, we define connections
$\tilde{\nabla}^{\kappa}=\tilde{\nabla}^{0}+\tilde{\Omega}^{\kappa}$
on the cotangent bundle of $H^{\circ}\times\mathbb{C}^{\times}$ with
$\tilde{\Omega}^{\kappa}\in \mathrm{Hom}(\Omega_{H^{\circ}\times\mathbb{C}^{\times}},
\Omega_{H^{\circ}\times\mathbb{C}^{\times}}\otimes\Omega_{H^{\circ}\times\mathbb{C}^{\times}})$
given by
\begin{equation}\label{eqn:affine-connection}
\tilde{\Omega}^{\kappa}:
\left\{
\begin{aligned}
\zeta\in\Omega_{H^{\circ}}&\mapsto\frac{1}{2}\sum\limits_{\alpha>0}
k_{\alpha}\zeta(X_{\alpha})d\alpha\otimes d\alpha+(B^{\kappa})^ {*}(\zeta)-\zeta\otimes\frac{dt}{t}-\frac{dt}{t}\otimes\zeta, \\
\frac{dt}{t}\in\Omega_{\mathbb{C}^{\times}} &\mapsto A^{\kappa}-\frac{dt}{t}\otimes\frac{dt}{t}.
\end{aligned}
\right.
\end{equation}
Here $t$ is the coordinate for $\mathbb{C}^{\times}$ and $A^{\kappa}$ and
$B^{\kappa}$ denote the translation invariant
tensor fields on $H$ or $H\times\mathbb{C}^{\times}$ defined by
$a^{\kappa}$ and $b^{\kappa}$ respectively. According to ($\ref{eqn:projective-connection}$), ($\ref{eqn:affine-connection}$),
we can write
$\Omega^{\kappa}$ and $\tilde{\Omega}^{\kappa}$ explicitly:
$$\Omega^{\kappa}:=\frac{1}{2}\sum\limits_{\alpha>0}k_{\alpha}d\alpha\otimes d\alpha\otimes X_{\alpha}+(B^{\kappa})^{*},$$
\begin{align*}
\tilde{\Omega}^{\kappa}:=&\frac{1}{2}\sum_{\alpha>0}k_{\alpha}d\alpha\otimes d\alpha\otimes X_{\alpha}+(B^{\kappa})^{*}
+c^{\kappa}\sum_{\alpha>0}d\alpha\otimes d\alpha\otimes t\frac{\partial}{\partial t}\\
&-\sum_{\alpha_{i}\in \mathfrak{B}}d\alpha_{i}\otimes\frac{dt}{t}\otimes\partial_{p_{i}}
-\frac{dt}{t}\otimes\frac{dt}{t}\otimes t\frac{\partial}{\partial t}
-\sum_{\alpha_{i}\in \mathfrak{B}}\frac{dt}{t}\otimes d\alpha_{i}\otimes\partial_{p_{i}}.
\end{align*}
Here $c^{\kappa}$ is a constant for each $\kappa$ such that $A^{\kappa}=c^{\kappa}\sum_{\alpha>0}d\alpha\otimes d\alpha$,
$\mathfrak{B}$ is a fundamental system for $R$ and ${p_{i}}$ is the dual basis of $\mathfrak{h}$ to $\alpha_{i}$
such that $\alpha_{i}(p_{j})=\delta^{i}_{j}$ where $\delta^{i}_{j}$ is the Kronecker delta.

We immediately have the following fact.

\begin{lemma}\label{lem:equations}
Let $f\in\mathcal{O}_{H^{\circ}}$. $\nabla^{\kappa}(df)+fA^{\kappa}=0$ can be written out in the form of
the system of differential equations $D^{\kappa}_{u,v}f=0$ for all $u,v\in\mathfrak{h}$, i.e.
\begin{equation*}
(\partial_{u}\partial_{v}+\frac{1}{2}\sum\limits_{\alpha>0}
k_{\alpha}\alpha(u)\alpha(v)\frac{e^{\alpha}(h)+1}{e^{\alpha}(h)-1}\partial_{\alpha^{\vee}}
+\partial_{b^{\kappa}(u,v)}+a^{\kappa}(u,v))f(h)=0 \;
\forall u, v\in\mathfrak{h}.
\end{equation*}
\end{lemma}

\begin{proof}
For $\omega\in\Omega^{1}(H^{\circ})$, we have
\[ \nabla^{\kappa}(\omega)=d\omega+\frac{1}{2}\sum\limits_{\alpha>0}
k_{\alpha}\frac{e^{\alpha}+1}{e^{\alpha}-1}d\alpha\otimes d\alpha
\cdot\alpha^{\vee}(\omega)+(B^{\kappa})^{*}(\omega).  \]
Its covariant derivative in direction $v\in\mathfrak{h}$ is:
\[ \nabla_{v}^{\kappa}(\omega)=\partial_{v}\omega+\frac{1}{2}\sum\limits_{\alpha>0}
k_{\alpha}\frac{e^{\alpha}+1}{e^{\alpha}-1}d\alpha\cdot\alpha(v)
\alpha^{\vee}(\omega)+(B_{v}^{\kappa})^{*}(\omega).  \]
Let $\omega=df$, we have
\begin{align*}
\nabla^{\kappa}_{v}(df)&=\partial_{v}(df)+\frac{1}{2}\sum\limits_{\alpha>0}
k_{\alpha}\frac{e^{\alpha}+1}{e^{\alpha}-1}d\alpha\cdot\alpha(v)
\alpha^{\vee}(df)+(B_{v}^{\kappa})^{*}(df)\\
&=d(\partial_{v}f)+\frac{1}{2}\sum\limits_{\alpha>0}
k_{\alpha}\frac{e^{\alpha}+1}{e^{\alpha}-1}\alpha(v)
\cdot\partial_{\alpha^{\vee}}f \cdot d\alpha+(B_{v}^{\kappa})^{*}(df).
\end{align*}
Now contraction with $u$:
\[ \nabla^{\kappa}_{v}(df)(u)=\partial_{u}\partial_{v}f+\frac{1}{2}\sum\limits_{\alpha>0}
k_{\alpha}\frac{e^{\alpha}+1}{e^{\alpha}-1}\alpha(u)\alpha(v)
\partial_{\alpha^{\vee}}f+\partial_{b^{\kappa}(u,v)}f,  \]
yields an element of $\mathcal{O}_{H^{\circ}}$.

So, $\nabla^{\kappa}(df)+fA^{\kappa}=0$ is equivalent to
\begin{align*}
\partial_{u}\partial_{v}f+\frac{1}{2}\sum_{\alpha>0}
k_{\alpha}\alpha(u)\alpha(v)\frac{e^{\alpha}+1}{e^{\alpha}-1}\partial_{\alpha^{\vee}}f
+\partial_{b^{\kappa}(u,v)}f+a^{\kappa}(u,v)f=0 \;
\forall u, v\in\mathfrak{h}
\end{align*}
\end{proof}

By Proposition $\ref{prop:affine-projective}$, the integrability of the above system can be told from whether
the curvature of $\nabla^{\kappa}$, viewed as a $\mathcal{O}_{H^{\circ}}$-homomorphism
$\nabla^{\kappa}\nabla^{\kappa}$: $\Omega_{H^{\circ}}\rightarrow\Omega_{H^{\circ}}^{2}\otimes\Omega_{H^{\circ}}$,
is given by wedging from the right
with the symmetric section $-A^{\kappa}$ of $\Omega_{H^{\circ}}\otimes\Omega_{H^{\circ}}: \zeta\mapsto -\zeta\wedge A^{\kappa}$.
Before we proceed to this, let us first look at an example.
\begin{example}
We take a root system of type $A_{2}$. We compute the curvature form of the
connection defined by this root system. For $\alpha, \beta, \gamma\in R_{+}$, we write
\begin{align*}
\Omega:=\frac{e^{\alpha}+1}{e^{\alpha}-1}d\alpha\otimes d\alpha\otimes\partial_{\alpha^{\vee}}
+\frac{e^{\beta}+1}{e^{\beta}-1}d\beta\otimes d\beta\otimes\partial_{\beta^{\vee}}
+\frac{e^{\gamma}+1}{e^{\gamma}-1}d\gamma\otimes d\gamma\otimes\partial_{\gamma^{\vee}}.
\end{align*}
\begin{align*}
\nabla=\nabla^{0}+\Omega \; \text{such} \; \text{that} \; \nabla^{0}(d\alpha)=0.
\end{align*}
Here we let $k_{\alpha}=2$ for all $\alpha$ and $k'=0$.

Let $\zeta=c_{1}d\alpha+c_{2}d\beta \in\Omega_{H^{\circ}}$, then we have
$$\nabla(\zeta)=\frac{e^{\alpha}+1}{e^{\alpha}-1}d\alpha\otimes\zeta(\partial_{\alpha^{\vee}})d\alpha
+\frac{e^{\beta}+1}{e^{\beta}-1}d\beta\otimes\zeta(\partial_{\beta^{\vee}})d\beta
+\frac{e^{\gamma}+1}{e^{\gamma}-1}d\gamma\otimes\zeta(\partial_{\gamma^{\vee}})d\gamma,$$
and furthermore,
\begin{align*}
&\nabla\nabla(\zeta)\\
=&-\frac{e^{\alpha}+1}{e^{\alpha}-1}d\alpha\wedge\nabla(\zeta(\partial_{\alpha^{\vee}})d\alpha)
   -\frac{e^{\beta}+1}{e^{\beta}-1}d\beta\wedge\nabla(\zeta(\partial_{\beta^{\vee}})d\beta)\\
   &-\frac{e^{\gamma}+1}{e^{\gamma}-1}d\gamma\wedge\nabla(\zeta(\partial_{\gamma^{\vee}})d\gamma)\\
=&-\frac{e^{\alpha}+1}{e^{\alpha}-1}d\alpha\wedge\zeta(\partial_{\alpha^{\vee}})
    \Big ( \frac{e^{\alpha}+1}{e^{\alpha}-1}d\alpha\otimes\partial_{\alpha^{\vee}}(d\alpha)d\alpha
       +\frac{e^{\beta}+1}{e^{\beta}-1}d\beta\otimes\partial_{\beta^{\vee}}(d\alpha)d\beta\\
       &+\: \frac{e^{\gamma}+1}{e^{\gamma}-1}d\gamma\otimes\partial_{\gamma^{\vee}}(d\alpha)d\gamma \Big )\\
&-\: \frac{e^{\beta}+1}{e^{\beta}-1}d\beta\wedge\zeta(\partial_{\beta^{\vee}})
    \Big (\frac{e^{\alpha}+1}{e^{\alpha}-1}d\alpha\otimes\partial_{\alpha^{\vee}}(d\beta)d\alpha
       +\frac{e^{\beta}+1}{e^{\beta}-1}d\beta\otimes\partial_{\beta^{\vee}}(d\beta)d\beta\\
       &+\: \frac{e^{\gamma}+1}{e^{\gamma}-1}d\gamma\otimes\partial_{\gamma^{\vee}}(d\beta)d\gamma \Big )\\
&-\: \frac{e^{\gamma}+1}{e^{\gamma}-1}d\gamma\wedge\zeta(\partial_{\gamma^{\vee}})
    \Big(\frac{e^{\alpha}+1}{e^{\alpha}-1}d\alpha\otimes\partial_{\alpha^{\vee}}(d\gamma)d\alpha
       +\frac{e^{\beta}+1}{e^{\beta}-1}d\beta\otimes\partial_{\beta^{\vee}}(d\gamma)d\beta\\
       &+\: \frac{e^{\gamma}+1}{e^{\gamma}-1}d\gamma\otimes\partial_{\gamma^{\vee}}(d\gamma)d\gamma \Big ),
\end{align*}
then making use of $\alpha+\beta+\gamma=0$ and $d\alpha\wedge d\beta=d\beta\wedge d\gamma=d\gamma\wedge d\alpha$,
we can write the curvature form as follows,
\begin{align*}
\nabla\nabla = &-\frac{e^{\alpha}+1}{e^{\alpha}-1}\frac{e^{\beta}+1}{e^{\beta}-1}
    d\alpha\wedge d\beta\otimes(d\alpha\otimes\partial_{\beta^{\vee}}-d\beta\otimes\partial_{\alpha^{\vee}})\\
&-\: \frac{e^{\gamma}+1}{e^{\gamma}-1}\frac{e^{\alpha}+1}{e^{\alpha}-1}
    d\gamma\wedge d\alpha\otimes(d\gamma\otimes\partial_{\alpha^{\vee}}-d\alpha\otimes\partial_{\gamma^{\vee}})\\
&-\: \frac{e^{\beta}+1}{e^{\beta}-1}\frac{e^{\gamma}+1}{e^{\gamma}-1}
    d\beta\wedge d\gamma\otimes(d\beta\otimes\partial_{\gamma^{\vee}}-d\gamma\otimes\partial_{\beta^{\vee}})\\
=& -d\alpha\wedge d\beta\otimes(d\beta\otimes\partial_{\alpha^{\vee}}-d\alpha\otimes\partial_{\beta^{\vee}}).
\end{align*}
Then let
\begin{align*}
A&=d\alpha\otimes d\alpha +d\beta\otimes d\beta +d\gamma\otimes d\gamma\\
&=2d\alpha\otimes d\alpha +2d\beta\otimes d\beta +d\alpha\otimes d\beta+ d\beta\otimes d\alpha,
\end{align*}
we can easily verify that
$$\nabla\nabla(\zeta)=-\zeta\wedge A.$$
\end{example}

Since $\Omega^{\kappa}$ and $\tilde{\Omega}^{\kappa}$ take values in the symmetric tensors, the connections they define are
torsion free. The inversion involution of $H^{\circ}\times\mathbb{C}^{\times}$ acts on its space
of logarithmic differentials, so that the latter decomposes into its subspace of invariants and the
subspace of anti-invariants. The collection
$\lbrace\frac{e^{\alpha}+1}{e^{\alpha}-1}d\alpha\rbrace_{\alpha\in R_{+}}$ consists of
linearly independent invariants, whereas the anti-invariants are the translation invariant
differentials $\lbrace d\alpha\rbrace_{\alpha\in R_{+}}$. In particular, $\tilde{\Omega}^{\kappa}$ is
logarithmic (in the sense that its matrix entries are logarithmic differentials) and is uniquely
given by the form ($\ref{eqn:affine-connection}$). An associated connection $\nabla^{\kappa}$ on the cotangent bundle
of $H^{\circ}$ is given by the corresponding connection matrix $\Omega^{\kappa}$.

Since the system of differential equations defined by $D^{\kappa}_{u,v}$ can be expressed as
$\nabla^{\kappa}(df)+fA^{\kappa}=0$ according to Lemma $\ref{lem:equations}$, then the integrability of the system
can be connected to the curvature of $\nabla^{\kappa}$ as follows.

\begin{theorem} \label{thm:specialeqn}
Let $f\in\mathcal{O}_{H^{\circ}}$. The system of $n(n+1)/2$ linearly independent differential
equations
\begin{equation}\label{eqn:specialeqn}
(\partial_{u}\partial_{v}+\frac{1}{2}\sum\limits_{\alpha>0}
k_{\alpha}\alpha(u)\alpha(v)\frac{e^{\alpha}(h)+1}{e^{\alpha}(h)-1}\partial_{\alpha^{\vee}}
+\partial_{b^{\kappa}(u,v)}+a^{\kappa}(u,v))f(h)=0 \;
\forall u, v\in\mathfrak{h}
\end{equation}
is an integrable system on $H^{\circ}$ if and only if $-A^{\kappa}$ represents the curvature
of $\nabla^{\kappa}$, where $A^{\kappa}$ is the translation invariant tensor field on
$H^{\circ}$ defined by $a^{\kappa}$.
\end{theorem}

\begin{proof}
From Lemma $\ref{lem:local-affine-functions}$, we can know that the function $\tilde{f}(z,t):=c+f(z)t$ has
a flat differential relative to $\tilde{\nabla}^{\kappa}$
if and only if $f\in\mathcal{O}_{H^{\circ}}$ satisfying
$\nabla^{\kappa}(df)+fA^{\kappa}=0$. Moreover the integrability of the system can be guaranteed by that
$-A^{\kappa}$ represents the curvature of $\nabla^{\kappa}$.
\end{proof}

We call ($\ref{eqn:specialeqn}$) the special hypergeometric system with multiplicity parameter $\kappa$.

In fact, for each system there exists an $A^{\kappa}$ such that
$\nabla^{\kappa}\nabla^{\kappa}(\zeta)=-\zeta\wedge A^{\kappa}$.
The direct verification of this
fact will be done in another paper, while here we shall see the integrability of this system
by checking the flatness of $\tilde{\nabla}^{\kappa}$ in the following section by another way.

\subsection{Flatness of $\tilde{\nabla}^{\kappa}$}\label{subsec:flatness-conditions}

We want to see whether or not $\tilde{\nabla}^{\kappa}$ is indeed flat. For this we wish to
apply the flatness criterion (1.2) of \cite{Looijenga-1999} to $\tilde{\nabla}^{\kappa}$. Here we
restate the criterion as follows.

\begin{lemma} \label{lem:flatness}
Let $U$ be a connected complex manifold. Suppose that $\overline{U}\supset U$ is a smooth
projective compactification of $U$ which adds to $U$ a simple normal crossing divisor $D$ whose irreducible
components $D_{i}$ are smooth. Suppose that $\overline{U}$ has no nonzero
regular 2-forms and that any irreducible component $D_{i}$ has no nonzero regular
1-forms. Then a logarithmic connection $E$ on the trivial vector bundle $U\times V$ over $U$ is
flat if and only if for every intersection
$I$ of two distinct irreducible components of $D$, the sum
$\sum_{D_{i}\supset I}\mathrm{Res}_{D_{i}}E$ commutes with each of its terms $\mathrm{Res}_{D_{i}}E$
($D_{i}\supset I$).
\end{lemma}

\begin{proof}
First we prove the following fact.

\noindent \textbf{Assertion:}
\emph{The condition that $[\sum_{D_{i}\supset I}\mathrm{Res}_{D_{i}}E,\mathrm{Res}_{D_{i}}E]=0$
is equivalent to $\mathrm{Res}_{I}\mathrm{Res}_{D_{i}}\mathrm{R}(\nabla)=0$.}

The connection $E$ on $U$ could be locally written as
$$E=\sum_{i}f_{i}\frac{dl_{i}}{l_{i}}\otimes E_{i}+\sum_{i}\omega_{i}\otimes E_{i}^{\prime},$$
where $f_{i}$'s are holomorphic functions, $D_{i}$ is given by $l_{i}=0$,
$\omega_{i}$'s are holomorphic 1-forms and $E_{i}$'s, $E_{i}^{\prime}$'s are the endomorphisms of $V$.
Then we have
\begin{align*}
\mathrm{Res}_{D_{i}}E=f_{i}\vert_{l_{i}=0}E_{i}
\end{align*}
and
\begin{align*}
E\wedge E
=&\sum_{i,j}f_{i}f_{j}\frac{dl_{i}}{l_{i}}
\frac{dl_{j}}{l_{j}}\otimes E_{i}E_{j}+\sum_{i,j}f_{i}\frac{dl_{i}}{l_{i}}\wedge\omega_{j}\otimes
(E_{i}E_{j}^{\prime}-E_{j}^{\prime}E_{i})\\
  &+\sum_{i,j}\omega_{i}\omega_{j}\otimes E_{i}^{\prime}E_{j}^{\prime}\\
=&\sum_{ i < j}f_{i}f_{j}\frac{dl_{i}}{l_{i}}
\frac{dl_{j}}{l_{j}}\otimes(E_{i}E_{j}-E_{j}E_{i})+\sum\limits_{i,j}f_{i}\frac{dl_{i}}{l_{i}}\wedge\omega_{j}\otimes
(E_{i}E_{j}^{\prime}-E_{j}^{\prime}E_{i})\\
&+\sum\limits_{i,j}\omega_{i}\omega_{j}\otimes E_{i}^{\prime}E_{j}^{\prime}.
\end{align*}
Then
\begin{align*}
\mathrm{Res}_{D_{i}}E\wedge E=\sum_{j:j\neq i}f_{i}f_{j}
\frac{dl_{j}}{l_{j}}\vert_{l_{i}=0}\otimes(E_{i}E_{j}-E_{j}E_{i})+
\sum_{j:j\neq i}f_{i}\omega_{j}\otimes [E_{i},E_{j}^{\prime}],
\end{align*}
As $I\subset D_{i}$ is given by $D_{j}\cap D_{i}$ for any $j\neq i$ with $D_{j}\supset I$,
we have
\begin{multline*}
\mathrm{Res}_{I}\mathrm{Res}_{D_{i}}E\wedge E=\sum\limits_{j: D_{j}\supset I, j\neq i}f_{i}f_{j}\vert
_{I}(E_{i}E_{j}-E_{j}E_{i})=\\
\sum\limits_{j: D_{j}\supset I, j\neq i}f_{i}f_{j}\vert
_{I}[E_{i},E_{j}]
=[f_{i}\vert_{I}E_{i},\sum\limits_{j}f_{j}\vert_{I}E_{j}]
=[\mathrm{Res}_{D_{i}}E,\sum \mathrm{Res}_{D_{j}}E].
\end{multline*}
Since the double residue of $dE$ is obviously zero (any term of $dE$ is of at most
simple pole), the assertion follows.

Let's continue to prove the lemma. Necessity is obvious, but it is also sufficient:
If the double residue of $\mathrm{R}(\nabla)$ is equal to zero, then
$\mathrm{Res}_{D_{i}}\mathrm{R}(\nabla)$
has no pole along $I\subset D_{j}\cap D_{i} \; \text{for} \; \forall D_{j}\neq D_{i}$ , hence
$\mathrm{Res}_{D_{i}}\mathrm{R}(\nabla)$ has as coefficients regular 1-form along $D_{i}$, but there is no nonzero regular
1-form along $D_{i}$, we then have $\mathrm{Res}_{D_{i}}\mathrm{R}(\nabla)=0$. Again, $\mathrm{R}(\nabla)$
has no pole along $D_{i}$, hence
$\mathrm{R}(\nabla)$ has as coefficients regular 2-form everywhere, but there is no nonzero regular
2-form on $\overline{U}$, we then have $\mathrm{R}(\nabla)=0$.
\end{proof}

From the lemma above, we can see that it requires a compactification of $H^{\circ}\times\mathbb{C}^{\times}$
with a boundary which is arrangement-like in order to invoke the flatness criterion. We shall
take this to be of the form $\hat{H}_{\Sigma}\times\mathbb{P}^{1}$, where the first factor is
defined below.

Recall that $H$ has a unique $\mathbb{Q}$-structure which is split, i.e., for which $H(\mathbb{Q})$
is isomorphic to a product of copies of $\mathbb{Q}^{\times}$. Each homomorphism
$u:\mathbb{C}^{\times}\rightarrow H$ defines a tangent vector
$Du(t\frac{\partial}{\partial t})\in\mathfrak{h}(\mathbb{Q})$ and these tangent vectors span
a lattice $\check{X}(H)\subset\mathfrak{h}(\mathbb{Q})$, called the \emph{cocharacter} \emph{lattice}.
\index{Lattice!cocharacter}
We
shall identify $\check{X}(H)$ with $\mathrm{Hom}(\mathbb{C}^{\times},H)=P^{\vee}$.

Since the $\mathfrak{h}_{\alpha}$'s are defined over $\mathbb{Q}$, they cut up $\mathfrak{h}(\mathbb{R})$
according to a rational cone decomposition $\Sigma$. The latter determines a compact torus
embedding $H\subset\hat{H}$ whose boundary is a union of toric divisors, indexed by the
one-dimensional faces of $\Sigma$. We denote by $\Pi$ the set of primitive elements of $\check{X}(H)$
in spanning a face of $\Sigma$ and associate an element $p$ of $\Pi$ with a boundary divisor $D_{p}$
at the place of infinity. Our assumption that $e^{\alpha}$ is primitive in $X(H)=Q$
implies that $H_{\alpha}$ is connected and hence irreducible. So an irreducible component
of $\hat{H}-H^{\circ}$ is now either the closure $\hat{H}_{\alpha}$ in $\hat{H}$ of
some $H_{\alpha}$ or is equal to some $D_{p}$ with $p\in\Pi$.

Two distinct boundary divisors meet precisely when the corresponding one dimensional
faces of $\Sigma$ span a two dimensional face. Clearly, the divisors ($t=0$) and ($t=\infty$)
meet all other divisors, and $\hat{H}_{\alpha}$ meets $D_{p}$ if and only if
$\alpha(p)=0$.

We shall not make any notational distinction between a connection on the cotangent
bundle of $H^{\circ}\times\mathbb{C}^{\times}$ and the associated
one on its tangent bundle, i.e.,
the connection on its tangent bundle is also denoted by $\tilde{\nabla}^{\kappa}$. In fact, the
associated (dual) connection on the tangent bundle of $H^{\circ}\times\mathbb{C}^{\times}$
is characterized by the property that the pairing between vector fields and differentials
is flat. So its connection form \index{Connection form} is $-(\tilde{\Omega}^{\kappa})^{*}$.

The residue of $(\tilde{\Omega}^{\kappa})^{*}$ along these divisors are as follows:
define elements
of $\mathrm{End}(\mathfrak{h})\subset\mathrm{End}(\mathfrak{h}\oplus\mathbb{C})$ by
\begin{align*}
u_{\alpha}&:=k_{\alpha}(\alpha^{\vee}\otimes\alpha),\\
U_{x}&:=-\frac{1}{4}\sum\limits_{\alpha\in R}\vert\alpha(x)\vert u_{\alpha},
\quad x\in\mathfrak{h}(\mathbb{R}).
\end{align*}
So $u_{-\alpha}=u_{\alpha}$ and $U_{-x}=U_{x}$. Notice the dependence of $U_{x}$
on $x$ is piecewise linear (relative to $\Sigma$) and continuous. For $z\in\mathfrak{h}$,
we define $b_{z}^{\kappa}\in \mathrm{End}(\mathfrak{h})$ and
$a_{z}^{\kappa}\in\mathfrak{h}^{\ast}$
as follows:
\begin{align*}
b_{z}^{\kappa}(w):=b^{\kappa}(z,w),\\
a_{z}^{\kappa}(w):=a^{\kappa}(z,w).
\end{align*}
We first need to compute the following residues. \index{Residue}

Notice that $d\alpha=d(\log e^{\alpha})=\frac{de^{\alpha}}{e^{\alpha}}$ and
the mappings:
\begin{align*}
\mathbb{C}^{\times}\xrightarrow{\gamma_{p}} &H\xrightarrow{e^{\alpha}} \mathbb{C}^{\times}\\
& t\mapsto t^{\alpha (p)},
\end{align*}
we have
\begin{align*}
\mathrm{Res}_{\hat{H}_{\alpha}\times\mathbb{P}^{1}}\frac{e^{\alpha}+1}{e^{\alpha}-1}d\alpha &=
\mathrm{Res}_{(e^{\alpha}=1)}\frac{e^{\alpha}+1}{e^{\alpha}-1}\frac{d(e^{\alpha}-1)}{e^{\alpha}}\\
&=2,
\end{align*}

\begin{align*}
\mathrm{Res}_{D_{p}\times\mathbb{P}^{1}}\frac{e^{\alpha}+1}{e^{\alpha}-1}d\alpha &=
\mathrm{Res}_{t=0}\gamma_{p}^{\ast}(\frac{e^{\alpha}+1}{e^{\alpha}-1}\frac{de^{\alpha}}{e^{\alpha}})\\
&=\mathrm{Res}_{t=0}\frac{t^{\alpha(p)}+1}{t^{\alpha(p)}-1}\alpha(p)\frac{dt}{t}\\
&=
\begin{cases}
\frac{t^{\alpha(p)}+1}{t^{\alpha(p)}-1}\alpha(p)\frac{dt}{t}=-\alpha(p)  \quad &\text{if} \quad \alpha(p)>0\\
\frac{1+t^{-\alpha(p)}}{1-t^{-\alpha(p)}}\alpha(p)\frac{dt}{t}=+\alpha(p)  \quad &\text{if} \quad \alpha(p)<0\\
\end{cases}
\\
&=-\vert\alpha(p)\vert,
\end{align*}

\begin{align*}
\mathrm{Res}_{D_{p}\times\mathbb{P}^{1}}d\alpha &=
\mathrm{Res}_{t=0}\gamma_{p}^{\ast}(\frac{de^{\alpha}}{e^{\alpha}})\\
&=\mathrm{Res}_{t=0}\alpha(p)\frac{dt}{t}\\
&=\alpha(p);
\end{align*}
then we can compute
\begin{align*}
\mathrm{Res}_{\hat{H}_{\alpha}\times\mathbb{P}^{1}}(\tilde{\Omega}^{\kappa})^{*}&=\frac{1}{2}\mathrm{Res}_{(e^{\alpha}=1)}
k_{\alpha}\frac{e^{\alpha}+1}{e^{\alpha}-1}d\alpha\cdot\alpha^{\vee}\otimes\alpha\\
&=k_{\alpha}(\alpha^{\vee}\otimes\alpha)\\
&=u_{\alpha},
\end{align*}

\begin{align*}
\mathrm{Res}_{D_{p}\times\mathbb{P}^{1}}(\tilde{\Omega}^{\kappa})^{*}\\
=&\frac{1}{4}\sum\limits_{\alpha\in R}
\mathrm{Res}_{D_{p}\times\mathbb{P}^{1}}k_{\alpha}\frac{e^{\alpha}+1}{e^{\alpha}-1}
        d\alpha\cdot\alpha^{\vee}\otimes\alpha+\mathrm{Res}_{D_{p}\times\mathbb{P}^{1}}B^{\kappa}\\
  &\: +c^{\kappa}\sum\limits_{\alpha>0}
    \mathrm{Res}_{D_{p}\times\mathbb{P}^{1}}d\alpha\cdot t\frac{\partial}{\partial t}\otimes \alpha
  -\sum\limits_{\alpha_{i}\in \mathfrak{B}}\mathrm{Res}_{D_{p}\times\mathbb{P}^{1}}d\alpha_{i}\cdot p_{i}\otimes \frac{dt}{t}\\
=&\frac{1}{4}\sum\limits_{\alpha\in R}k_{\alpha}(-\vert\alpha(p)\vert\alpha^{\vee}\otimes\alpha)+\mathrm{Res}_{D_{p}\times\mathbb{P}^{1}}B^{\kappa}\\
  &\: +c^{\kappa}\sum\limits_{\alpha>0}\alpha(p)\cdot t\frac{\partial}{\partial t}\otimes \alpha
  -\sum\limits_{\alpha_{i}\in \mathfrak{B}}\alpha_{i}(p)\cdot p_{i}\otimes \frac{dt}{t}\\
=& U_{p}+b_{p}^{\kappa}+t\frac{\partial}{\partial t}\otimes a_{p}^{\kappa}-p \otimes \frac{dt}{t},
\end{align*}
and
\begin{align*}
\mathrm{Res}_{t=0}(\tilde{\Omega}^{\kappa})^{*}&=\mathrm{Res}_{t=0}(-\frac{dt}{t})\cdot t\frac{\partial}{\partial t}\otimes \frac{dt}{t}
+\sum\limits_{\alpha_{i}\in \mathfrak{B}}\mathrm{Res}_{t=0}(-\frac{dt}{t})\cdot p_{i}\otimes \alpha_{i}\\
&=-t\frac{\partial}{\partial t}\otimes\frac{dt}{t} -\sum\limits_{\alpha_{i}\in \mathfrak{B}}p_{i}\otimes \alpha_{i}\\
&=-1_{\mathbb{C}}-1_{\mathfrak{h}}\\
&=-1_{\mathfrak{h}\oplus\mathbb{C}}\\
&=-\mathrm{Res}_{t=\infty}(\tilde{\Omega}^{\kappa})^{*}.
\end{align*}

We shall sometimes drop $\kappa$ from $\nabla^{\kappa}$, $\tilde{\nabla}^{\kappa}$,
$a^{\kappa}$ and $b^{\kappa}$ when no confusion arises, but we need to bear in mind all these notations appearing in
what follows in this section depend on $\kappa$ unless other specified.

Having these residues on hand and making use of Lemma $\ref{lem:flatness}$, we have the flatness criterion
for $\tilde{\nabla}$ as follows.

\begin{lemma} \label{lem:flatness-conditions}
The connection $\tilde{\nabla}$ is flat (and thus defines an affine structure on $H^{\circ}\times\mathbb{C}^{\times}$)
if and only if the following conditions hold:
\leftmargini=7mm
\begin{enumerate}
\item for every rank two sublattice $L\subset Q$, the sum
$\sum_{\alpha\in R\cap L}u_{\alpha}$ commutes with each of its terms,
\item each $u_{\alpha}$ is self-adjoint relative to $a$ (equivalently: $a(\alpha^{\vee},z)=c_{\alpha}\alpha(z)$)
for some $c_{\alpha}\in\mathbb{C}$),
\item for every $z\in\mathfrak{h}$, $b_{z}$ is self-adjoint relative to $a$
(equivalently: $a(b(z_{1},z_{2}),z_{3})$ is symmetric in its arguments),
\item if $\alpha(p)=0$, then \\
\emph{(a)} $[u_{\alpha},U_{p}]=0$ and \\
\emph{(b)} $[u_{\alpha},b_{p}]=0$,
\item if $p,q\in\Pi$ span a two dimensional face of $\Sigma$, then \\
\emph{(a)} $[U_{p},b_{q}]=[U_{q},b_{p}]$ and \\
\emph{(b)} $[U_{p},U_{q}]+[b_{p},b_{q}]=p\otimes a_{q}-q\otimes a_{p}$.
\end{enumerate}
\end{lemma}

\begin{proof}
From Lemma $\ref{lem:flatness}$, we can know that the connection is flat if and only if the $\tilde{\Omega}^{*}$-residues along
the added divisors have the property that the collection of $\tilde{\Omega}^{*}$-residues of divisors passing through
any preassigned codimension two intersection has a sum which commutes with each of its terms. We write
this out for the present case.

For an intersection $E\times\mathbb{P}^{1}$ of two distinct divisors of type $\hat{H}_{\alpha}\times\mathbb{P}^{1}$
we get (1): the characters of $H$ that are trivial on $E$ make up a primitive rank two sublattice $L$ of
$Q$ and $R\cap L$ is the set of $\alpha\in R$ for which $\hat{H}_{\alpha}\supset E$. Conversely,
for any rank two sublattice $L\subset Q$ which contains two independent elements of
$R$, $\cap_{\alpha\in R\cap L}\hat{H}_{\alpha}$ is nonempty and of codimension two in $\hat{H}$, then
the sum $\sum_{\alpha\in R\cap L}u_{\alpha}$ commutes with each of its terms.

The intersection of $\hat{H}_{\alpha}\times\mathbb{P}^{1}$ and $D_{p}\times\mathbb{P}^{1}$
is nonempty only if $\alpha(p)=0$ and in that case no other boundary divisor will contain
that intersection; since
\begin{align*}
&[U_{p}+b_{p}+t\frac{\partial}{\partial t}\otimes a_{p}-p\otimes \frac{dt}{t},u_{\alpha}]\\
=&[U_{p}+b_{p},u_{\alpha}]+k_{\alpha}(a(p,\alpha^{\vee})t\frac{\partial}{\partial t}\otimes \alpha\\
  &-\alpha(t\frac{\partial}{\partial t})\alpha^{\vee}\otimes a_{p}
  -\frac{dt}{t}(\alpha^{\vee})p\otimes \alpha+\alpha(p)\alpha^{\vee}\otimes\frac{dt}{t})\\
=&[U_{p}+b_{p},u_{\alpha}]+k_{\alpha}a(p,\alpha^{\vee})t\frac{\partial}{\partial t}\otimes \alpha,
\end{align*}
this yields $[U_{p}+b_{p},u_{\alpha}]=0$ and the condition that $a(p,\alpha^{\vee})=0$
when $\alpha(p)=0$. Since $U_{-p}=U_{p}$ and $b_{-p}=-b_{p}$, we get (4).
The hyperplane $\mathfrak{h}_{\alpha}$ is spanned by its intersection with $\Pi$.
So the fact that $a(\alpha^{\vee},p)=0$ for all $p\in\Pi$ with $\alpha(p)=0$ implies that
$a(\alpha^{\vee},y)=c_{\alpha}\alpha(y)$ for some $c_{\alpha}$. This tells us that
$a(u_{\alpha}(z),w)=c_{\alpha}\alpha(z)\alpha(w)$ is symmetric in $z$ and $w$,
in other words, $u_{\alpha}$ is self-adjoint relative to $a$. Conversely, if $u_{\alpha}$ is
self-adjoint relative to $a$, then clearly, $a(\alpha^{\vee},p)=0$ when $\alpha(p)=0$. So
this amounts to (2).

The intersection of $D_{p}\times\mathbb{P}^{1}$ and $D_{q}\times\mathbb{P}^{1}$ with
$p$ and $q$ distinct is nonempty only if $p$ and $q$ span a two dimensional face. In that case,
\begin{multline*}
[U_{p}+b_{p}+t\frac{\partial}{\partial t}\otimes a_{p}-p\otimes \frac{dt}{t},
U_{q}+b_{q}+t\frac{\partial}{\partial t}\otimes a_{q}-q\otimes \frac{dt}{t}]=\\
[U_{p}+b_{p},U_{q}+b_{q}]-p\otimes a_{q}+q\otimes a_{p}+t\frac{\partial}{\partial t}\otimes (a_{p}(U_{q}+b_{q})-a_{q}(U_{p}+b_{p})).
\end{multline*}
(We used that $a(p,q)$, $b(p,q)$ and $U_{p}(q)=-\frac{1}{2}\sum_{\alpha\in R:\alpha(p)> 0,\alpha(q)>0}k_{\alpha}\alpha(p)\alpha(q)\alpha^{\vee}$
are symmetric in $p$ and $q$.) We thus have $[U_{p}+b_{p},U_{q}+b_{q}]=p\otimes a_{q}-
q\otimes a_{p}$ and $(a_{p}(U_{q}+b_{q})-a_{q}(U_{p}+b_{p}))$. If we take its invariant
and anti-invariant part in the former equality,we immediately have
$[U_{p},U_{q}]+[b_{p},b_{q}]+[U_{p},b_{q}]-[U_{q},b_{p}]=p\otimes a_{q}-
q\otimes a_{p}$ and $[U_{p},U_{q}]+[b_{p},b_{q}]-[U_{p},b_{q}]+[U_{q},b_{p}]=p\otimes a_{q}-q\otimes a_{p}$, this yields (5). The latter is equivalent to
$a(p,U_{q}(z)+b(q,z))=a(q,U_{p}(z)+b(p,z))$ for all $z$. Since
$u_{\alpha}$ is self-adjoint relative to $a$, $U_{p}$ is self-adjoint relative to $a$
as well,
we then have $a(p,U_{q}(z))=a(U_{q}(p),z)=a(U_{p}(q),z)=a(q,U_{p}(z))$. The latter
condition hence simplifies to $a(p,b(q,z))$ is symmetric in $p$ and $q$. Since
$b$ itself is symmetric and $p$ and $q$ are basis roots of $P^{\vee}$, we get (3).

The residue on the divisors defined by $t=0$ and $t=\infty$ are scalars and hence yield
no conditions.
\end{proof}

\begin{remark}
Following \cite{Couwenberg-Heckman-Looijenga}, Condition (1) is precisely what one needs in order that for every sublattice $L$
of $X(H)$ spanned by elements of $R$ the `linearized connection' on $\mathfrak{h}-\cup_{\alpha\in R\cap L}
\mathfrak{h}_{\alpha}$ defined by the End($\mathfrak{h}$)-valued differential
\[ \Omega_{L}:=\sum_{\alpha\in R\cap L} k_{\alpha}\frac{d\alpha}{\alpha}\otimes \pi_{\alpha} \]
be flat. According to \emph{loc}. \emph{cit}., it is also true that the sum $\sum_{\alpha\in R\cap L}k_{\alpha}\pi_{\alpha}$
commutes with each of its terms. If $a$ is defined over $\mathbb{R}$ and positive definite, then
Conditions (1) and (2) define a Dunkl system in the sense of \cite{Couwenberg-Heckman-Looijenga}.
\end{remark}

Now we need to verify these conditions of Lemma $\ref{lem:flatness-conditions}$ in order to show that the connection
$\tilde{\nabla}$ in our case is flat if we choose an appropriate bilinear form $a$. But before we proceed to that,
it is absolutely necessary to investigate the toric Lauricella case which gives
a hint on these conditions.

\begin{example}[The toric Lauricella case]\label{exm:toric-Lauricella}
Let $N:=\{1,2,\cdots,n+1\}$ and assign each $i\in N$ a
positive real number $\mu_{i}$. Label the standard basis of $\mathbb{C}^{n+1}$
as $\varepsilon_{1},\cdots,\varepsilon_{n+1}$. We endow $\mathbb{C}^{n+1}$ with a
bilinear form as $a(z,w):=\sum_{i=1}^{n+1}\mu_{i}z^{i}w^{i}$ where $z$ is given by
$z=\sum z^{i}\varepsilon_{i}$. Let $\mathfrak{h}$ be the quotient of $\mathbb{C}^{n+1}$
by its main diagonal $\varepsilon:=\mathbb{C}\sum\varepsilon_{i}$, but we may often
identify it with the orthogonal complement of the main diagonal in $\mathbb{C}^{n+1}$,
that is, with the hyperplane defined by $\sum\mu_{i}z^{i}=0$. We take our $\alpha$'s
to be the collection $\alpha_{i,j}:=(z_{i}-z_{j})_{i\neq j}$ where $z_{i}$ is the dual
basis of $\varepsilon_{i}$ in $\mathfrak{h}^{*}$. We associate each $\alpha_{i,j}$
a $v_{i,j}:=v_{z_{i}-z_{j}}:=\mu_{j}\varepsilon_{i}-\mu_{i}\varepsilon_{j}$.

We immediately notice that the set $R:=\{\alpha_{i,j}\}$ generates a discrete subgroup
of $\mathfrak{h}^{*}$ whose $\mathbb{R}$-linear span defines a real form $\mathfrak{h}
(\mathbb{R})$ of $\mathfrak{h}$. It's easy to show that $a(v_{i,j},\beta)=0$ for
any $\beta\in \ker(\alpha_{i,j})$. According to \cite{Couwenberg-Heckman-Looijenga}, for every rank two
subgroup $L$ of the lattice generated by $R$, $\sum_{R\cap L} u_{\alpha_{i,j}}$
commutes with each of its terms, if $u_{\alpha_{i,j}}$ is the endomorphism of
$\mathfrak{h}$ defined by $u_{\alpha_{i,j}}(z)=\alpha_{i,j}(z)v_{i,j}$.

We still denote by $\Pi$ the set of primitive elements of the cocharacter lattice
in spanning a face of the rational cone decomposition $\Sigma$. Then the
elements of $\Pi$ correspond to proper subsets $I$ of $N$:
\[ p_{I}:=\frac{\mu_{I'}}{\mu_{N}}\varepsilon_{I}-\frac{\mu_{I}}{\mu_{N}}\varepsilon_{I'} \]
where $I'$ is the complement set of $I$ in $N$, $\mu_{I}:=\sum_{i\in I}\mu_{i}$,
and $\varepsilon_{I}:=\sum_{i\in I}\varepsilon_{i}$.

Let $U_{x}:=\sum_{\alpha_{i,j}\in
R:\alpha_{i,j}(x)>0}\alpha_{i,j}(x)u_{\alpha_{i,j}}$. Notice that $\alpha_{i,j}(p_{I})>0$
if and only if $i\in I$ and $j\in I'$, its value
then being $1$. Put $U_{I}:=U_{p_{I}}$, we thus have:
\begin{align*}
U_{I}(z)&=\sum_{i\in I,j\in I'}(z^{i}-z^{j})(\mu_{j}\varepsilon_{i}-\mu_{i}\varepsilon_{j})\\
&=\mu_{I'}\sum_{i\in I}z^{i}\varepsilon_{i}-\sum_{j\in I'}\mu_{j}z^{j}\sum_{i\in I}
\varepsilon_{i}-\sum_{i\in I}\mu_{i}z^{i}\sum_{j\in I'}\varepsilon_{j}+
\mu_{i}\sum_{j\in I'}z^{j}\varepsilon_{j}\\
&=(\mu_{I'}\sum_{i\in I}z^{i}\varepsilon_{i}-(\sum_{j\in I'}\mu_{j}z^{j})\varepsilon_{I})
+(\mu_{I}\sum_{j\in I'}z^{j}\varepsilon_{j}-(\sum_{i\in I}\mu_{i}z^{i})\varepsilon_{I'}).
\end{align*}
Notice that the coefficients of $\varepsilon_{k}$ and $\varepsilon_{l}$ are the same
whenever $k,l$ are both in $I$ or both in $I'$ and $z\in \mathrm{ker}(\alpha_{k,l})$. In other
words, $\alpha_{k,l}(U_{I}(z))=0$ for $z\in \mathrm{ker}(\alpha_{k,l})$ whenever $\alpha_{k,l}(p_{I})=0$.

We also find that
\[ U_{I}(p_{I})=\mu_{N}p_{I}. \]
Notice that $p_{I}$ and $p_{J}$ span a face if and only if $I$ and $J$ satisfy an
inclusion relation: $I\subset J$ or $I\supset J$. A straightforward computation
shows that
\begin{align*}
U_{J}(p_{I})&=U_{I}(p_{J})=\mu_{J'}p_{I}+\mu_{I}p_{J},\\
a(p_{I},p_{J})&=\frac{\mu_{I}\mu_{J'}}{\mu_{N}},\\
[U_{I},U_{J}](z)&=\mu_{N}(a(z,p_{J})p_{I}-a(z,p_{I})p_{J}).
\end{align*}

There actually exists a nonzero cubic form in this case. Let $\tilde{f}:
\mathbb{C}^{n+1}\rightarrow \mathbb{C}$ be defined by $\tilde{f}(z):=
\sum\mu_{i}(z^{i})^{3}$ and take for $f:\mathfrak{h}\rightarrow \mathbb{C}$
its restriction to $\mathfrak{h}$. The partial derivative of $\tilde{f}$
with respect to $v_{i,j}$ is $3\mu_{j}\mu_{i}(z_{i}^{2}-z_{j}^{2})$, which
is divisible by $\alpha_{i,j}$.

The symmetric bilinear map $\tilde{b}:\mathbb{C}^{n+1}\times\mathbb{C}^{n+1}
\rightarrow\mathbb{C}^{n+1}$ is given by $\tilde{b}(\varepsilon_{i},\varepsilon_{j}):=
\delta_{ij}\varepsilon_{i}$. Then the map $b:\mathfrak{h}\times\mathfrak{h}\rightarrow
\mathfrak{h}$ corresponding to $f$ is the restriction of $\pi\circ\tilde{b}$ to
$\mathfrak{h}\times\mathfrak{h}$ among which $\pi:\mathbb{C}^{n+1}\rightarrow\mathfrak{h}$
is the orthogonal projection from $\mathbb{C}^{n+1}$ to $\mathfrak{h}$. We also
find that $a(\tilde{b}(z,z),z)=\tilde{f}(z)$ and $a(b(z,z),z)=f(z)$
. So if we write $\tilde{b}_{i}(z)$ for
$\tilde{b}(\varepsilon_{i},z)$, we can write $\tilde{b}_{i}$ as $\tilde{b}_{i}=
\varepsilon_{i}\otimes z_{i}$. If we write $b_{i,j}(z)$ for $b(v_{i,j},z)$, then
\[ b_{i,j}=\mu_{j}\varepsilon_{i}\otimes z_{i}-\mu_{i}\varepsilon_{j}\otimes z_{j}
-\frac{\mu_{i}\mu_{j}}{\mu_{N}}\varepsilon_{N}\otimes(z_{i}-z_{j}).\]
If we write $a_{i,j}(z)=a(v_{i,j},z)$, then $a_{i,j}=\mu_{i}\mu_{j}(z_{i}-z_{j})$,
we can verify that $[b_{i,j},b_{k,l}]=-\mu_{N}^{-1}(v_{i,j}\otimes a_{k,l}-v_{k,l}\otimes
a_{i,j})$. Hence we have $[b_{z},b_{w}]=-\mu_{N}^{-1}(z\otimes a_{w}-w\otimes a_{z})$.
\end{example}

We first verify the conditions about the bilinear map $b$ in Lemma $\ref{lem:flatness-conditions}$
since as Lemma $\ref{lem:bmap}$ says, a nonzero $b$ only exists for a root system of type $A_{n}$.

\begin{lemma}
The conditions \emph{(3)}, \emph{(4)(b)} and \emph{(5)(a)} of Lemma $\ref{lem:flatness-conditions}$ hold for a root system of
type $A_{n}$.
\end{lemma}

\begin{proof}
Now we let all $\mu_{i}$ equal to $1$ in the above example, then the above example becomes the case of a root system
of type $A_{n}$. Since the dimension of $\mathrm{Hom}(\mathrm{Sym}^{2}\mathfrak{h},\mathfrak{h})^{W}$ is just $1$, then the
$b_{0}=\sum_{\alpha >0}\alpha\otimes\alpha\otimes\alpha'$ given in Lemma $\ref{lem:bmap}$ is
just a multiple of $b$ given in the above example. We thus have
$b_{i,j}(z)=z^{i}\varepsilon_{i}-z^{j}\varepsilon_{j}-\frac{1}{n+1}(z^{i}-z^{j})\varepsilon_{N}$.
If $i<j<k$, then
\[ a(b_{i,j}(z),\varepsilon_{j}-\varepsilon_{k})=-z^{j}=a(b_{j,k}(z),\varepsilon_{i}-\varepsilon_{j}); \]
if $i,j,k,l$ are pairwise distinct, then
\[ a(b_{i,j}(z),\varepsilon_{k}-\varepsilon_{l})=0. \]
Since $\{\varepsilon_{i}-\varepsilon_{i+1}\mid i=1,2,\cdots,n\}$ is a basis of $\mathfrak{h}$,
Condition (3) holds.

It's obvious that $u_{\alpha}$ is self-adjoint relative to $a$ where
$v_{\alpha}:=k_{\alpha}\alpha^{\vee}$. This is equivalent to
$a(z,v_{\alpha})=c_{\alpha}\alpha(z)$ for some $c_{\alpha}\in \mathbb{C}$. Since $a(b(z,z),z)=f(z)$ and $\partial_{\alpha^{\vee}}f$ is divisible by $\alpha$ for each $\alpha\in R$, there
exists a $g_{\alpha}\in \mathfrak{h}^{*}$ such that $a(v_{\alpha},b(z,w))=
\alpha(z)g_{\alpha}(w)+\alpha(w)g_{\alpha}(z)$.
If $p\in\mathfrak{h}_{\alpha}$, then
\[ a(w,b_{p}u_{\alpha}(z))=\alpha(z)a(w,b(p,v_{\alpha}))=\alpha(z)a(v_{\alpha},b(p,w))=
\alpha(z)\alpha(w)g_{\alpha}(p),\]
but also
\[a(w,u_{\alpha}b_{p}(z))=a(u_{\alpha}(w),b(p,z))=\alpha(w)a(v_{\alpha},b(p,z))=
\alpha(w)\alpha(z)g_{\alpha}(p).\]
This yields Condition (4)(b): $[u_{\alpha},b_{p}]=0$.

If $p\in\mathfrak{h}_{\alpha}(\mathbb{R})$ spans a $1$-face, then
$b_{p}u_{\alpha}=u_{\alpha}b_{p}$ implies that
$\alpha(z)b_{p}(v_{\alpha})=b_{p}u_{\alpha}(z)=u_{\alpha}b_{p}(z)=
\alpha(b_{p}(z))v_{\alpha}$.
This shows that $b_{p}$ has $v_{\alpha}$ as an eigenvector, with eigenvalue
$\lambda_{p,\alpha}$, say. It could also be written as: $b_{v_{\alpha}}(p)=\lambda_{p,\alpha}v_{\alpha}$. Since $\mathfrak{h}_{\alpha}$
is generated by the $1$-faces it contains, it follows that there is a unique
linear form $\lambda_{\alpha}$ on $\mathfrak{h}_{\alpha}$ such that
$b_{v_{\alpha}}(z)=\lambda_{\alpha}(z)v_{\alpha}$ for all $z\in\mathfrak{h}_{\alpha}$.
Choose $v'_{\alpha}\in\mathfrak{h}$ such that $\lambda_{\alpha}(z)=a(v'_{\alpha},z)$
for all $z\in\mathfrak{h}_{\alpha}$. We can see that this $v'_{\alpha}$ is unique up to
a multiple of $v_{\alpha}$ since $\mathfrak{h}_{\alpha}$ is the $a$-orthogonal complement
of $v_{\alpha}$. So $b_{v_{\alpha}}$ has rank at most two and will be of the form
$b_{v_{\alpha}}(z)=a(v'_{\alpha},z)v_{\alpha}+a(v_{\alpha},z)v''_{\alpha}$ for
some $v''_{\alpha}\in\mathfrak{h}$. Since $b_{v_{\alpha}}$ is self-adjoint relative
to $a$, $a(b_{v_{\alpha}}(z),w)=a(v'_{\alpha},z)a(v_{\alpha},w)+a(v_{\alpha},z)
a(v''_{\alpha},w)$ is symmetric in $z$ and $w$. This means $v''_{\alpha}$ and
$v'_{\alpha}$ differ by a multiple of $v_{\alpha}$. So by a suitable choice of
$v'_{\alpha}$, we can arrange that $v'_{\alpha}=v''_{\alpha}$. Then
\begin{align*}
a(b_{v_{\alpha}}(z),w)&=a(v'_{\alpha},z)a(v_{\alpha},w)+a(v_{\alpha},z)
a(v'_{\alpha},w)\\
&=c_{\alpha}(a(v'_{\alpha},z)\alpha(w)+\alpha(z)a(v'_{\alpha},w)).
\end{align*}
Let $p,q\in\mathfrak{h}(\mathbb{R})$ span a face of $\Sigma$. Then
\begin{align*}
a(b_{q}(z),&U_{p}(w))\\
&=-\frac{1}{2}\sum_{\alpha:\alpha(p)>0}\alpha(p)\alpha(w)
a(b_{q}(z),v_{\alpha})\\
&=-\frac{1}{2}\sum_{\alpha:\alpha(p)>0}\alpha(p)\alpha(w)
a(b_{v_{\alpha}}(z),q)\\
&=-\frac{1}{2}\sum_{\alpha:\alpha(p)>0}c_{\alpha}\alpha(p)\alpha(w)
(a(v'_{\alpha},z)\alpha(q)+\alpha(z)a(v'_{\alpha},q))\\
&=-\frac{1}{2}\sum_{\alpha:\alpha(p)>0}c_{\alpha}
(\alpha(p)\alpha(q)a(v'_{\alpha},z)\alpha(w)+\alpha(p)a(v'_{\alpha},q)\alpha(z)\alpha(w)).
\end{align*}
Since $U_{p}$ is self-adjoint relative to $a$, hence
\begin{align*}
a([U_{p},b_{q}](z),w)&=a(b_{q}(z),U_{p}(w))-a(b_{q}(w),U_{p}(z))\\
&=-\frac{1}{2}\sum_{\alpha:\alpha(p)>0}c_{\alpha}\alpha(p)\alpha(q)
(a(v'_{\alpha},z)\alpha(w)-\alpha(z)a(v'_{\alpha},w)).
\end{align*}
This is symmetric in $p$ and $q$, because $\alpha(p)>0$ implies $\alpha(q)\geq 0$
and the terms with $\alpha(q)=0$ vanish. So Condition (5)(a) holds: $[U_{p},b_{q}]=
[U_{q},b_{p}]$.
\end{proof}

Now we need to verify the other conditions for all the reduced irreducible root
systems.

\begin{theorem}
The connection $\tilde{\nabla}$ defined in ($\ref{eqn:affine-connection}$) is flat if we choose
an appropriate bilinear form $a$, and hence the connection
$\nabla$ defines a projective structure on $H^{\circ}$.
\end{theorem}

\begin{proof}
By Lemma $\ref{lem:bmap}$, Conditions (3), (4)(b) and (5)(a) are empty and $[b_{p},b_{q}]=0$
for all types other than $A_{n}$. So we only need to verify those remaining conditions
in Lemma $\ref{lem:flatness-conditions}$.

Because $W$ acts irreducibly on the space spanned by $L$, hence by Schur's lemma the sum
$\sum_{\alpha\in R\cap L}u_{\alpha}$
acts as a scalar operator, Condition (1) is thus satisfied. And there always exists a nondegenerate symmetric bilinear
form $a:\mathfrak{h}\times\mathfrak{h}\rightarrow\mathbb{R}$ such that $\alpha^{\vee}$ is $a$-perpendicular
to ker($\alpha$), i.e., $a(\alpha^{\vee},p)=0$ if $\alpha(p)=0$. This implies Condition (2).

As the reflection $s_{\alpha}$ and $u_{\alpha}$ have the relation: $u_{\alpha}=k_{\alpha}(1-s_{\alpha})$,
$u_{\alpha}$ commutes with $U_{p}$ is equivalent to that
$s_{\alpha}$ commutes with $U_{p}$. While $s_{\alpha}u_{\beta}s_{\alpha}^{-1}=u_{s_{\alpha}(\beta)}$,
we have
\begin{align*}
s_{\alpha}U_{p}s_{\alpha}^{-1}&=s_{\alpha}(-\frac{1}{4}\sum\limits_{\beta\in R}\vert\beta(p)\vert u_{\beta})s_{\alpha}^{-1}\\
    &=-\frac{1}{4}\sum\limits_{\beta\in R}\vert s_{\alpha}^{2}(\beta)(p)\vert u_{s_{\alpha}(\beta)}\\
    &=-\frac{1}{4}\sum\limits_{\beta^{\prime}\in R}\vert s_{\alpha}(\beta^{\prime})(p)\vert u_{\beta^{\prime}}
           \qquad (\text{here we let} \; \beta^{\prime}=s_{\alpha}(\beta))\\
    &=-\frac{1}{4}\sum\limits_{\beta^{\prime}\in R}\vert\beta^{\prime}(s_{\alpha}(p))\vert u_{\beta^{\prime}}\\
    &=-\frac{1}{4}\sum\limits_{\beta^{\prime}\in R}\vert\beta^{\prime}(p)\vert u_{\beta^{\prime}}
           \qquad (\text{here} \; s_{\alpha}(p)=p \; \text{because} \; \alpha(p)=0)\\
    &=U_{p}.
\end{align*}
So Condition (3) follows.

If $p\in \mathfrak{h}_{\alpha}$, then $U_{p}$ preserves $\mathfrak{h}_{\alpha}$ and since $U_{p}$ is
self-adjoint relative to $a$, $U_{p}$ will have $\alpha^{\vee}$ as an eigenvector. Since $\mathbb{C}p+\mathbb{C}q$
is an intersection of hyperplanes $\mathfrak{h}_{\alpha}$, the $a$-orthogonal complement of $\mathbb{C}p+\mathbb{C}q$
is spanned by the vectors $\alpha^{\vee}$ it contains. Hence we have a common eigenspace decomposition of this
subspace for $U_{p}$ and $U_{q}$. In particular, these endomorphisms commute there. Since $[U_{p},U_{q}]$
is an element of the Lie algebra of the orthogonal group of $a$ whose kernel contains the $a$-orthogonal complement
of $\mathbb{C}p+\mathbb{C}q$, it is necessarily a multiple of $p\otimes a_{q}-q\otimes a_{p}$, i.e., $[U_{p},U_{q}]=\lambda(p\otimes a_{q}-q\otimes a_{p})$. But for each pair of $(p,q)$, we can choose a chamber $C$, a member of $\Sigma$ that is
open and nonempty in $\mathfrak{h}(\mathbb{R})$, and this pair of $(p,q)$ could be pulled back by an
element of $W$ to some 2-face of the closure of the chamber $\bar{C}$. Notice $U_{x}$ is continuous
on $\bar{C}$ and the map $(x,y)\mapsto[U_{x},U_{y}]$ is bilinear on $\bar{C}\times\bar{C}$, we can know that
all the pair of $[U_{p},U_{q}]$ share the same coefficient of $\lambda$.
In particular, for $R$ of types other than $A_{n}$, we can normalize $a$ such that $\lambda$ becomes equal to $1$. For $R$ of type $A_{n}$, we know that
$[b_{p},b_{q}]$ is also a multiple of $p\otimes a_{q}-q\otimes a_{p}$ and share the
same coefficient $\mu$ for any pair of $(p,q)$ from Example $\ref{exm:toric-Lauricella}$,
so we can also normalize $a$ such that
$\lambda+\mu=1$. Then, Condition (4) is satisfied.
\end{proof}

In fact, we can write out the explicit form of $a^{\kappa}$ in terms of a given
inner product $(\cdot,\cdot)$ according to Condition (5)(b) of Lemma $\ref{lem:flatness-conditions}$ if
we want to construct a projective structure on $H^{\circ}$.

\begin{theorem}
If we use the construction of root systems in Bourbaki and take the inner product
$(\cdot,\cdot)$ such
that $(\varepsilon_{i},\varepsilon_{j})=\delta_{ij}$, then the $a^{\kappa}$ such that
$\tilde{\nabla}^{\kappa}$ is flat is given as follows:
\begin{align*}
&A_{n}: a^{\kappa}(u,v)=\frac{(n+1)}{4}(k^{2}-k'^{2})(u,v);\\
&B_{n}: a^{\kappa}(u,v)=((n-2)k^{2}+kk')(u,v);\\
&C_{n}: a^{\kappa}(u,v)=((n-2)k^{2}+2kk')(u,v);\\
&D_{n}: a^{\kappa}(u,v)=(n-2)k^{2}(u,v);\\
&E_{n}: a^{\kappa}(u,v)=ck^{2}(u,v);\quad c=6,12,30 \; \text{for} \; n=6,7,8;\\
&F_{4}: a^{\kappa}(u,v)=(k+k')(2k+k')(u,v);\\
&G_{2}: a^{\kappa}(u,v)=\frac{3}{4}(k+3k')(k+k')(u,v).
\end{align*}
\end{theorem}

\begin{proof}
We already know that $a^{\kappa}$ is a multiple of the given inner product by the
Schur's lemma if $R$ is irreducible. Then it's a straightforward computation by
Condition (5)(b) of Lemma $\ref{lem:flatness-conditions}$: $[U_{p},U_{q}]+[b_{p},b_{q}]=p\otimes a^{\kappa}_{q}-
q\otimes a^{\kappa}_{p}$. In fact, the $a^{\kappa}$ for type $A_{n}$ can be obtained directly from
Example $\ref{exm:toric-Lauricella}$.

Let's determine the $a^{\kappa}$ for type $C_{n}$ for example. Put $p:=\varepsilon_{1}+\cdots+
\varepsilon_{s}$ and $q:=\varepsilon_{1}+\cdots+\varepsilon_{t}$. Assume that $s<t$
without loss of generality. It's obvious that $(p,p)=s$ and $(p,q)=s$.
A straightforward computation
shows that
\begin{align*}
U_{p}(\varepsilon_{m})&=-(((n-2)k+2k')\varepsilon_{m}+kp) \quad \; \text{for} \; 1\leq m\leq s;\\
U_{p}(\varepsilon_{m})&=-sk\varepsilon_{m} \quad \; \text{for} \; s+1\leq m\leq n.
\end{align*}
Then we have
\begin{align*}
U_{p}(p)&=-((n-2+s)k+2k')p;\\
U_{p}(q)&=U_{q}(p)=-((n-2)k+2k')p+skq);\\
U_{q}(q)&=-((n-2+t)k+2k')q.
\end{align*}
Hence
\[ [U_{p},U_{q}](p)=sk((n-2)k+2k')p-sk((n-2)k+2k')q. \]
We thus have
\[ a^{\kappa}(u,v)=((n-2)k^{2}+2kk')(u,v). \]

The calculation for all the cases can be found on the corresponding website \cite{Shen-2018}.
\end{proof}

Therefore, we have constructed a $W$-invariant projective structure on $H^{\circ}$
where $H$ is an adjoint torus. \index{Torus!adjoint}

\section{Hyperbolic structures}\label{sec:hyperbolic-structures}

In this section, we show that the toric arrangement complement $H^{\circ}$ \index{Toric arrangement complement}
admits a hyperbolic structure when $\kappa$
lies in some region so that its image under the projective evaluation map lands in a complex ball. In Section
$\ref{subsec:geometric-structures}$,
we review the basic theory of geometric structures with logarithmic singularities. In Section $\ref{subsec:eigenvalues}$,
we compute the eigenvalues of
the residue endomorphisms along those added divisors, which almost equals to obtaining
the logarithmic exponents along those divisors.
In Section $\ref{subsec:reflection-representation}$, we use the method of reflection representation to investigate
the corresponding Hermitian form so that we
can determine the hyperbolic region for $H^{\circ}$. In Section $\ref{subsec:evaluation-map}$, we set up the
(projective) evaluation map and
even give out the evaluation map around those divisors in the form of local coordinates. In Section $\ref{subsec:complex-ball}$,
we provide
a proof showing that the dual Hermitian form is greater than $0$ when $\kappa$ lies in the hyperbolic region which means
its image under the projective evaluation map lands in a complex ball.

\subsection{Geometric structures with logarithmic singularities}\label{subsec:geometric-structures}

We shall in this section introduce the geometric structures on a complex manifold in a very brief way. A good exposition on this
topic is Chapter 1 of \cite{Couwenberg-Heckman-Looijenga}.

Let $N$ be a connected complex manifold and $\tilde{N}\rightarrow N$ be a holonomy covering \index{Covering!holonomy}
and denote its Galois group by $\Gamma$. So $\mathrm{Aff}(\tilde{N}):=
\mathrm{H}^{0}(\tilde{N},\mathrm{Aff}_{\tilde{N}})$ is a $\Gamma$-invariant vector space of affine-linear functions on $\tilde{N}$. Then
the set $A$ of linear forms $\mathrm{Aff}(\tilde{N})\rightarrow\mathbb{C}$ which are the identity on $\mathbb{C}$ is an affine
$\Gamma$-invariant hyperplane in $\mathrm{Aff}(\tilde{N})^{*}$.

\begin{definition}
Given a holonomy cover as above, the evaluation map $ev: \tilde{N}\rightarrow A$ which assigns to $\tilde{z}$ the linear form
$ev_{\tilde{z}}\in A: \mathrm{Aff}(\tilde{N})\rightarrow\mathbb{C}; \tilde{f}\mapsto \tilde{f}(\tilde{z})$ is called the
$\emph{developing map}$ \index{Developing map} of the affine structure; it is $\Gamma$-equivariant and a local affine isomorphism.
\end{definition}

This tells us that a developing map determines a natural affine atlas on $N$ whose charts take values in $A$ and whose transition
maps lie in $\Gamma$.

\begin{definition}
Suppose an affine structure \index{Affine structure}
is given on a complex manifold $N$ by a torsion free, flat connection $\nabla$. \index{Connection!flat}
\index{Connection!torsion free} A nowhere zero
holomorphic vector field $E$ on $N$ is called a $\emph{dilatation field}$ \index{Dilatation field}
with factor $\lambda\in\mathbb{C}$, where $\nabla_{X}(E)=\lambda X$ for
every local vector field $X$.
\end{definition}

If $X$ is flat, then the torsion freeness yields: $[E,X]=\nabla_{E}(X)-\nabla_{X}(E)=-\lambda X$. This tells us that Lie derivative
with respect to $E$ acts on flat vector fields simply as multiplication by $-\lambda$. Hence it acts on flat differentials as
multiplication by $\lambda$.

Let $h$ be a flat Hermitian form \index{Hermitian form}
on the tangent bundle of $N$ such that $h(E,E)$ is nowhere zero. Then the leaf space $N/E$ of the
dimension one foliation induced by $E$ inherits a Hermitian form $h_{N/E}$ in much the same way as the projective space of a finite
dimensional Hilbert space acquires its Fubini-Study metric. We are especially interested in the case when $h_{N/E}$ is positive
definite:

\begin{definition}
Let $N$ be a complex manifold with an affine structure and there is a dilatation field $E$ on $N$ with factor $\lambda$. We say
that a flat Hermitian form $h$ on $N$ is $\emph{admissible}$ \index{Hermitian form!admissible}
relative to $E$ if it is in one of the following three cases:
\leftmargini=7mm
\begin{enumerate}
  \item elliptic: $\lambda\neq 0$ and $h>0$;
  \item parabolic: $\lambda=0$ and $h\geq 0$ with kernel spanned by $E$;
  \item hyperbolic: $\lambda\neq 0$, $h(E,E)<0$ and $h>0$ on $E^{\perp}$.
\end{enumerate}
\end{definition}

Then the leaf space $N/E$ acquires a metric $h_{N/E}$ of constant holomorphic sectional curvature, for it is locally isometric
to a complex projective space with Fubini-Study metric, to a complex-Euclidean space or to a complex-hyperbolic space respectively.

In order to understand the behavior of an affine structure near a given smooth subvariety of its singular locus, we need to blow up
that subvariety so that we are dealing with the codimension one case. Let's first look at the simplest degenerating affine
structures as follows which is also in \cite{Couwenberg-Heckman-Looijenga}.

\begin{definition}\label{def:infinitesimal}
Let $D$ be a smooth connected hypersurface in a complex manifold $N$ and an affine structure on $N-D$ is endowed. We say that
the affine structure on $N-D$ has an $\emph{infinitesimally simple degeneration along D}$ of logarithmic exponent
$\lambda\in\mathbb{C}$ if
\leftmargini=7mm
\begin{enumerate}
  \item $\nabla$ extends to $\Omega_{N}(\log D)$ with a logarithmic pole along $D$,

  \item the residue of this extension along $D$ preserves the subsheaf
  $\Omega_{D}\subset\Omega_{N}(\log D)\otimes\mathcal{O}_{D}$
and its eigenvalue on the quotient sheaf $\mathcal{O}_{D}$ is $\lambda$, \\
\noindent and

  \item the residue endomorphism restricted to $\Omega_{D}$ is semisimple and all of its eigenvalues are $\lambda$ or 0.
\end{enumerate}
\end{definition}

We have the following local model for the behavior of the developing map for such a degenerating affine structure \cite{Couwenberg-Heckman-Looijenga}.

\begin{proposition}
Let be given a smooth connected hypersurface $D$ in a complex manifold $N$, an affine structure on $N-D$ and a point
$p\in D$. Then the infinitesimally simple degeneration along $D$ at $p$ of logarithmic exponent $\lambda\in\mathbb{C}$
can also be described in the form of local coordinates.

Namely, there exists a local equation $t$ for $D$ and a local chart
\[ (F_{0},t,F_{\lambda}):N_{p}\rightarrow (T_{0}\times\mathbb{C}\times T_{\lambda})_{(0,0,0)} \]
($T_{\lambda}$ incorporates into $T_{0}$ when $\lambda=0$), where $T_{0}$ and $T_{\lambda}$ are vector spaces, such that the
developing map near $p$ is affine equivalent to a multivalued map
for which the explicit form can be found in Proposition 1.10 of \cite{Couwenberg-Heckman-Looijenga},
depending on whether $\lambda$ is a non-integer, a positive integer, a negative integer, or zero.
\end{proposition}

In fact, we need to understand what happens in case $D$ is a normal crossing divisor in the complex manifold $N$ and the affine
structure on $N-D$ degenerates infinitesimally simply along some irreducible component of $D$.

\begin{proposition}[\cite{Couwenberg-Heckman-Looijenga}]
Let be given a complex manifold $N$ with a simple normal crossing divisor $D$ on it, whose irreducible components
$D_{1},\cdots,D_{k}$ are smooth. An affine structure on $N-D$ is endowed with infinitesimally simply degeneration
along $D_{i}$ of logarithmic exponent $\lambda_{i}$.
Suppose
that $\lambda_{i}>-1$ and that the holonomy around $D_{i}$ is semisimple unless $\lambda_{i}=0$.
Let $p$ be a point of $\cap D_{i}$.
Then $\lambda_{i}\neq 0$ for $i<k$ and the
local affine retraction $r_{i}$ at the generic point of $D_{i}$ extends to $r_{i}: N_{p}\rightarrow D_{i}$ in such a manner
that $r_{i}r_{j}=r_{i}$ for $i<j$.
\end{proposition}

This proposition makes it possible to perform the so-called \emph{Looijenga compactifiaction} on an arrangement complement,
determined by the arrangement. For this operation to be done, we first need to carry out a sequence of iterated blowups
in terms of the \emph{intersection lattice} of the arrangement so as to get a big resolution.
Via this we shall arrive at the situation mentioned in the proposition. Then we can contract those exceptional
divisors to get the Looijenga compactification without worrying about the compatibility of their local affine formations.
This very technical process, a successive blowups followed by contractions, can be found in
\cite{Looijenga-2003}.

\subsection{Eigenvalues of the residue endomorphisms}\label{subsec:eigenvalues}

We can see from the preceding section that it is important to know that the eigenvalues of the residue maps of the connection
$\tilde{\nabla}^{\kappa}$. Those residues are as follows from the last
section.
\begin{align*}
\mathrm{Res}_{\hat{H}_{\alpha}\times\mathbb{P}^{1}}(\tilde{\Omega}^{\kappa})^{*}&=u_{\alpha},\\
\mathrm{Res}_{D_{p}\times\mathbb{P}^{1}}(\tilde{\Omega}^{\kappa})^{*}
&= U_{p}+b_{p}^{\kappa}+t\frac{\partial}{\partial t}\otimes a_{p}^{\kappa}-p \otimes \frac{dt}{t},\\
\mathrm{Res}_{t=0}(\tilde{\Omega}^{\kappa})^{*}&=-\mathrm{Res}_{t=\infty}(\tilde{\Omega}^{\kappa})^{*}
=-1_{\mathfrak{h}\oplus\mathbb{C}}.
\end{align*}

Now let us compute the eigenvalues of these residues. We look at the residue map along the mirror $\hat{H}_{\alpha}\times\mathbb{P}^{1}$, view $u_{\alpha}$ as an endomorphism of $\mathfrak{h}$ instead of
$\mathfrak{h}\oplus\mathbb{C}$ at first. We immediately have
\begin{align*}
u_{\alpha}(\alpha^{\vee})&=2k_{\alpha}\alpha^{\vee},\\
u_{\alpha}(p)&=0 \quad \text{for} \quad \forall p\perp \alpha^{\vee}.
\end{align*}
Then we have the following eigenvalues
\begin{equation*}
\left\{
\begin{aligned}
&u_{\alpha}((\alpha^{\vee},0))=2k_{\alpha}(\alpha^{\vee},0),\\
&u_{\alpha}((p,0))=0 \quad \text{for} \quad \forall p\perp \alpha^{\vee},\\
&u_{\alpha}((0,\lambda t\frac{\partial}{\partial t}))=0
\end{aligned}
\right.
\end{equation*}
if we view $u_{\alpha}$ as an endomorphism of $\mathfrak{h}\oplus\mathbb{C}$.

Then we need to compute the eigenvalues of the residue maps of the connection along the boundary divisor $D_{p}\times\mathbb{P}^{1}$,
which is much more
complicated. Let us first look at an example. Then we shall have some feeling about how these eigenvalues come up.

\begin{example}
We take type $A_{n}$ and regard $U_{p}$ and $b_{p}$ as endomorphisms of $\mathfrak{h}$ at first as before.
Here we still use the construction for root systems from Bourbaki. Let root system $R$ of type $A_{n}$ sit inside a Euclidean
space $\mathbb{R}^{n+1}$ and denote its orthonormal basis by $e^{1},e^{2},\cdots,e^{n+1}$, so its positive roots
\index{Root!positive}
are all of the
form $e^{i}-e^{j}$ for $1\leq i< j\leq n+1$. Its dual root system $R^{\vee}$ is also of type $A_{n}$ and we denote the dual
orthonormal basis by $\varepsilon_{1},\varepsilon_{2},\cdots,\varepsilon_{n+1}$. Then its positive coroots are all of the form
$\varepsilon_{i}-\varepsilon_{j}$ for $1\leq i<j\leq n+1$ and simple coroots are of the form of $\varepsilon_{i}-\varepsilon_{i+1}$
for $i=1,2,\cdots,n$.

Let $p=\varpi_{m}^{\vee}:=\frac{n+1-m}{n+1}(\varepsilon_{1}+\cdots +\varepsilon_{m})-\frac{m}{n+1}(\varepsilon_{m+1}+\cdots
+\varepsilon_{n+1})$. We know that all the positive roots $\alpha$ of $R$ such that $\alpha(p)\neq 0$ are of the form
$e^{i}-e^{j}$ for $1\leq i \leq m,\; m+1\leq j \leq n+1$ and in fact $\alpha(p)=1$ for all these $\alpha$'s.
Write
\begin{align*}
\sigma_{0}&=\sum_{\alpha\in R_{+}}|\alpha(p)|(\alpha^{\vee}\otimes\alpha)\\
&=\sum_{\scriptstyle 1\leq i \leq m \atop \scriptstyle m+1\leq j \leq n+1}(\varepsilon_{i}-\varepsilon_{j})\otimes(e^{i}-e^{j}),
\end{align*}
we then have
\begin{eqnarray*}
\sigma_{0}(\varepsilon_{s})=
\left\{
\begin{aligned}
&\sum_{m+1\leq j \leq n+1}(\varepsilon_{s}-\varepsilon_{j}) \quad &&\text{for} \quad 1\leq s \leq m \\
&\sum_{1\leq i \leq m}-(\varepsilon_{i}-\varepsilon_{s}) \quad &&\text{for} \quad m+1\leq s\leq n+1.
\end{aligned}
\right.
\end{eqnarray*}
After a straightforward computation, we have
\begin{eqnarray*}
\left\{
\begin{aligned}
&\sigma_{0}(\varepsilon_{s}-\varepsilon_{t})=(n+1-m)(\varepsilon_{s}-\varepsilon_{t}) \quad &&\text{for} \quad 1\leq s<t \leq m \\
&\sigma_{0}(p)=(n+1)p \\
&\sigma_{0}(\varepsilon_{s}-\varepsilon_{t})=m(\varepsilon_{s}-\varepsilon_{t}) \quad &&\text{for} \quad m+1\leq s<t \leq n+1.
\end{aligned}
\right.
\end{eqnarray*}
Since $U_{p}=-\frac{1}{4}\sum_{\alpha\in R}|\alpha(p)|k(\alpha^{\vee}\otimes\alpha)=-\frac{1}{2}k\sigma_{0}$, then the
above tells us
that
\begin{eqnarray*}
\left\{
\begin{aligned}
&U_{p}(\alpha_{i}^{\vee})=-\frac{1}{2}(n+1-m)k\alpha_{i}^{\vee} \quad &&\text{for} \quad 1\leq i \leq m-1 \\
&U_{p}(p)=-\frac{1}{2}(n+1)kp \\
&U_{p}(\alpha_{i}^{\vee})=-\frac{1}{2}mk\alpha_{i}^{\vee} \quad &&\text{for} \quad m+1\leq i \leq n.
\end{aligned}
\right.
\end{eqnarray*}
Since $b_{p}=\frac{1}{2}k'\sum_{\alpha >0}\alpha(p)(\alpha'\otimes\alpha)$ where $\alpha'=\varepsilon_{i}+\varepsilon_{j}-
\frac{2}{n+1}(\varepsilon_{1}+\cdots+\varepsilon_{n+1})$, the computation for $b_{p}$ is similar to the situation of $U_{p}$
and hence we have
\begin{eqnarray*}
\left\{
\begin{aligned}
&b_{p}(\alpha_{i}^{\vee})=\frac{1}{2}(n+1-m)k'\alpha_{i}^{\vee} \quad &&\text{for} \quad 1\leq i \leq m-1 \\
&b_{p}(p)=\frac{1}{2}(n+1-2m)k'p \\
&b_{p}(\alpha_{i}^{\vee})=-\frac{1}{2}mk'\alpha_{i}^{\vee} \quad &&\text{for} \quad m+1\leq i \leq n.
\end{aligned}
\right.
\end{eqnarray*}
If we write $\sigma=\mathrm{Res}_{D_{p}\times\mathbb{P}^{1}}(\tilde{\Omega}^{\kappa})^{*}
= U_{p}+b_{p}+t\frac{\partial}{\partial t}\otimes a_{p}-p \otimes \frac{dt}{t}$, regard $U_{p}$ and $b_{p}$ as endomorphisms of
$\mathfrak{h}\oplus\mathbb{C}$, we then have the following eigenvalues
after a little bit more effort
\begin{eqnarray*}
\left\{
\begin{aligned}
&\sigma((\alpha_{i}^{\vee},0))=-\frac{1}{2}(n+1-m)(k-k')(\alpha_{i}^{\vee},0) \quad \text{for} \quad 1\leq i \leq m-1 \\
&\sigma((p,-\frac{1}{2}m(k+k')t\frac{\partial}{\partial t}))
=-\frac{1}{2}(n+1-m)(k-k')(p,-\frac{1}{2}m(k+k')t\frac{\partial}{\partial t}) \\
&\sigma((p,-\frac{1}{2}(n+1-m)(k-k')t\frac{\partial}{\partial t}))
=-\frac{1}{2}m(k+k')(p,-\frac{1}{2}(n+1-m)(k-k')t\frac{\partial}{\partial t}) \\
&\sigma((\alpha_{i}^{\vee},0))=-\frac{1}{2}m(k+k')(\alpha_{i}^{\vee},0) \quad \text{for} \quad m+1\leq i \leq n.
\end{aligned}
\right.
\end{eqnarray*}

\end{example}

From this example, we notice that the eigenvalue of $U_{p}(+b_{p})$ on the space $\mathbb{C}p$ is the sum of the two eigenvalues
on the spaces $\mathfrak{h}_{1}=\mathrm{span}\{\alpha_{1}^{\vee},\cdots,\alpha_{m-1}^{\vee}\}$ and
$\mathfrak{h}_{2}=\mathrm{span}\{\alpha_{m+1}^{\vee},\cdots,\alpha_{n}^{\vee}\}$ respectively. And the product of the two eigenvalues is
$a(p,p)$. In fact, this holds for all the root systems. In the end, $\sigma$ has two eigenvalues on the space
$\mathfrak{h}\oplus\mathbb{C}$.

\begin{theorem}
Let $\sigma=\mathrm{Res}_{D_{p}\times\mathbb{P}^{1}}(\tilde{\Omega}^{\kappa})^{*}
= U_{p}+b_{p}+t\frac{\partial}{\partial t}\otimes a_{p}-p \otimes \frac{dt}{t}$, then $\sigma$ has at most two eigenvalues on the
space $\mathfrak{h}\oplus\mathbb{C}$ with multiplicites $m$ and $n+1-m$ respectively.
In fact, these two eigenvalues satisfy such a quadratic equation
$\lambda^{2}-\varphi\lambda+a(p,p)=0$ where $\varphi$ is the eigenvalue of $U_{p}(+b_{p})$ (if viewed as an endomorphism
of $\mathfrak{h}$) on $\mathbb{C}p$.
\end{theorem}

\begin{proof}
For type $A_{n}$, from above example, the result follows.

For other types, $b_{p}=0$, suppose $p=\varpi_{m}^{\vee}$, we have a decomposition of $\mathfrak{h}$:
\[ \mathfrak{h}=\mathbb{C}p \oplus \sum_{i}\mathfrak{h}_{i}  \]
where $\mathfrak{h}_{i}$ is just the space spanned by the irreducible root subsystem after deleting the $m$-th node from the original
root system $R$. These subspaces have the corresponding Weyl groups, denoted by $W_{i}$ respectively. In fact,
$U_{p}$ is a $W_{i}$-invariant endomorphism in $\mathfrak{h}_{i}$, so $U_{p}$ is just a scalar action in
$\mathfrak{h}_{i}$ by Schur's lemma.
We write $U_{p}(v)=\lambda_{i}v$ if $v\in \mathfrak{h}_{i}$ and $U_{p}(p)=\varphi_{1} p$.

We check the square of $\sigma$, we have
\begin{align*}
\sigma^{2}=U_{p}^{2}-p\otimes a_{p}-U_{p}(p)\otimes\frac{dt}{t}+t\frac{\partial}{\partial t}\otimes a_{p}(U_{p})-a(p,p)t\frac{\partial}{\partial t}
\otimes\frac{dt}{t}.
\end{align*}
From the computation above, below and online appendix \cite{Shen-2018}, we find that for all the root systems,
we have (at most) two eigenvalues on the
space perpendicular to $p$ whose sum is $\varphi_{1}$ and product $a(p,p)$, i.e.,
\begin{align*}
U_{p}^{2}(q)-\varphi_{1}U_{p}(q)+a(p,p)q=0 \quad \text{for} \quad \forall q\in p^{\bot}
\end{align*}
Then we can easily check that
\[ \sigma^{2}-\varphi_{1}\sigma+a(p,p)=0 .\]
The multiplicities follow from the decomposition of the root system.
\end{proof}

\begin{remark}
The value $m$ and $n+1-m$ for the multiplicities in
the above theorem is true except for an extremal node of $D_{n}$ and some nodes of $E_{n}$, which one can see from
below. In fact, we only need to bear in mind that in principle the multiplicities follow from the decomposition of the
root system.
\end{remark}

The computation for all the other types are similar for which the reader could check the online appendix \cite{Shen-2018},
we just list the results over here.

For type $B_{n}$, corresponding to $p=\varepsilon_{1}+\cdots+\varepsilon_{m}$ for $1\leq m \leq n$, we have
\begin{eqnarray*}
\left\{
\begin{aligned}
&\sigma((\alpha_{i}^{\vee},0))=-((n-2)k+k')(\alpha_{i}^{\vee},0) \quad \text{for} \quad 1\leq i \leq m-1 \\
&\sigma((p,-mk t\frac{\partial}{\partial t}))
=-((n-2)k+k')(p,-mk t\frac{\partial}{\partial t}) \\
&\sigma((p,-((n-2)k+k')t\frac{\partial}{\partial t}))
=-mk(p,-((n-2)k+k')t\frac{\partial}{\partial t}) \\
&\sigma((\alpha_{i}^{\vee},0))=-mk(\alpha_{i}^{\vee},0) \quad \text{for} \quad m+1\leq i \leq n.
\end{aligned}
\right.
\end{eqnarray*}

For type $C_{n}$, corresponding to $p=\varepsilon_{1}+\cdots+\varepsilon_{m}$ for $1\leq m < n$, we have
\begin{eqnarray*}
\left\{
\begin{aligned}
&\sigma((\alpha_{i}^{\vee},0))=-((n-2)k+2k')(\alpha_{i}^{\vee},0) \quad \text{for} \quad 1\leq i \leq m-1 \\
&\sigma((p,-mk t\frac{\partial}{\partial t}))
=-((n-2)k+2k')(p,-mk t\frac{\partial}{\partial t}) \\
&\sigma((p,-((n-2)k+2k')t\frac{\partial}{\partial t}))
=-mk(p,-((n-2)k+2k')t\frac{\partial}{\partial t}) \\
&\sigma((\alpha_{i}^{\vee},0))=-mk(\alpha_{i}^{\vee},0) \quad \text{for} \quad m+1\leq i \leq n;
\end{aligned}
\right.
\end{eqnarray*}
and corresponding to $p=\frac{1}{2}(\varepsilon_{1}+\cdots+\varepsilon_{n})$, we have
\begin{eqnarray*}
\left\{
\begin{aligned}
&\sigma((\alpha_{i}^{\vee},0))=-\frac{1}{2}((n-2)k+2k')(\alpha_{i}^{\vee},0) \quad \text{for} \quad 1\leq i \leq n-1 \\
&\sigma((p,-\frac{1}{2}nk t\frac{\partial}{\partial t}))
=-\frac{1}{2}((n-2)k+2k')(p,-\frac{1}{2}nk t\frac{\partial}{\partial t}) \\
&\sigma((p,-\frac{1}{2}((n-2)k+2k')t\frac{\partial}{\partial t}))
=-\frac{1}{2}nk(p,-\frac{1}{2}((n-2)k+2k')t\frac{\partial}{\partial t}).
\end{aligned}
\right.
\end{eqnarray*}

For type $D_{n}$, corresponding to $p=\varepsilon_{1}+\cdots+\varepsilon_{m}$ for $1\leq m \leq n-2$, we have
\begin{eqnarray*}
\left\{
\begin{aligned}
&\sigma((\alpha_{i}^{\vee},0))=-(n-2)k(\alpha_{i}^{\vee},0) \quad \text{for} \quad 1\leq i \leq m-1 \\
&\sigma((p,-mk t\frac{\partial}{\partial t}))
=-(n-2)k(p,-mk t\frac{\partial}{\partial t}) \\
&\sigma((p,-(n-2)kt\frac{\partial}{\partial t}))
=-mk(p,-(n-2)kt\frac{\partial}{\partial t}) \\
&\sigma((\alpha_{i}^{\vee},0))=-mk(\alpha_{i}^{\vee},0) \quad \text{for} \quad m+1\leq i \leq n;
\end{aligned}
\right.
\end{eqnarray*}
and corresponding to $p=\frac{1}{2}(\varepsilon_{1}+\cdots+\varepsilon_{n-1}-\varepsilon_{n})$ or
$p=\frac{1}{2}(\varepsilon_{1}+\cdots+\varepsilon_{n-1}+\varepsilon_{n})$, we have
\begin{eqnarray*}
\left\{
\begin{aligned}
&\sigma((\alpha_{i}^{\vee},0))=-\frac{1}{2}(n-2)k(\alpha_{i}^{\vee},0) \quad \text{for} \quad \forall \alpha_{i}^{\vee}\perp p \\
&\sigma((p,-\frac{1}{2}nk t\frac{\partial}{\partial t}))
=-\frac{1}{2}(n-2)k(p,-\frac{1}{2}nk t\frac{\partial}{\partial t}) \\
&\sigma((p,-\frac{1}{2}(n-2)kt\frac{\partial}{\partial t}))
=-\frac{1}{2}nk(p,-\frac{1}{2}(n-2)kt\frac{\partial}{\partial t}).
\end{aligned}
\right.
\end{eqnarray*}

For type $F_{4}$, the eigenvalues are $\{-(m+1)(k+k'),-m(2k+k')\}$ corresponding to $p=\varpi_{m}^{\vee}$ for $m=1,2,3$ and
$\{-2(k+k'),-2(2k+k')\}$ for $p=\varpi_{4}^{\vee}$.

For type $G_{2}$, the eigenvalues are $\{-(k+3k'),-\frac{3}{2}(k+k')\}$ and $\{-\frac{1}{2}(k+3k'),-(k+k')\}$
for $p=\varpi_{1}^{\vee}$ and $\varpi_{2}^{\vee}$ respectively.

For type $E_{n}$, the computation becomes more complicated since the construction for their fundamental coweights is
somehow irregular one by one. Then in order to compute their eigenvalues, we have the following observation:
the collection of $\alpha\in R$ with $\alpha(p)>0$ is a union of $W_{p}$-orbits. So if we put for any $W_{p}$-orbit of roots
that are positive on $p$ as follows:
\[ E_{O}:=\frac{1}{2|W_{p}|}\sum_{w\in W_{p}}w\alpha^{\vee}\otimes w\alpha=\frac{1}{2|O|}\sum
_{\alpha'\in O}\alpha'^{\vee}\otimes\alpha'    \]
where $\alpha$ is a member of $O$, then $U_{p}$ is a linear combination of such $E_{O}$'s.

We denote the orthogonal complement of $p$ by $\bar{\mathfrak{h}}$. Since $E_{O}$ is a $W_{p}$-invariant endomorphism in each
summand $\mathfrak{h}_{i}$, $E_{O}$ is a scalar action on $\mathbb{C}p$ and
each summand $\mathfrak{h}_{i}$ by Schur's lemma. These eigenvalues only depend on the $W_{p}$-orbit $O$ of $\alpha$ and so we denote them by
$\lambda_{p,O}$ and $\lambda_{i,O}$. First we notice that the trace of $E_{O}$ is equal to $\frac{1}{2}\alpha(\alpha^{\vee})=1$
and we shall show that the traces of $E_{O}$ on these eigenspaces are distributed in a simple manner. First we observe that
\[ a(E_{O}(x),y)=\frac{a(\alpha^{\vee},\alpha^{\vee})}{4|W_{p}|}\sum_{w\in W_{p}}\alpha(wx)\cdot\alpha(wy). \]
When $x=y=p$ the left hand side is $a(p,p)\lambda_{p,O}$ and the right hand side becomes $\frac{1}{4}a(\alpha^{\vee},\alpha^{\vee})
\alpha(p)^{2}=a(\alpha^{\vee},p)^{2}/a(\alpha^{\vee},\alpha^{\vee})$ and so
\[ \lambda_{p,O}=\frac{a(\alpha,p)^{2}}{a(p,p)a(\alpha^{\vee},\alpha^{\vee})}=
\frac{\parallel\pi_{p}(\alpha^{\vee})\parallel_{a}^{2}}{\parallel\alpha^{\vee}\parallel_{a}^{2}}, \]
where $\parallel\;\parallel_{a}$ denotes the norm associated to $a$. This is just the cosine squared of the angle between
$\alpha^{\vee}$ and $\pi_{p}(\alpha^{\vee})$.

When $x=y\in \mathfrak{h}_{i}$, the left hand side is $a(x,x)\lambda_{i,O}$. Now we only consider the case for which $R$ is a single
$W$-orbit for simplicity. Then we have $R_{i}:=R\cap \mathfrak{h}_{i}^{*}$ is a $W_{p}$-orbit. We denote the Coxeter number
\index{Coxeter!number}
of
$W(R_{i})$ by $h_{i}$. It is known that $|R_{i}|=h_{i}\dim \mathfrak{h}_{i}$ and that $\sum_{\beta\in R_{i}}\beta\otimes\beta$ is
a $W(R_{i})$-invariant form on $\mathfrak{h}_{i}$ which gives each coroot the squared length $4h_{i}$. So if we take in the above
formula $x=y$ a coroot of $R_{i}$, then
\begin{align*}
\lambda_{i,O}&=\frac{1}{4|W_{p}|}\sum_{w\in W_{p}}\alpha(w\beta^{\vee})^{2}=\frac{1}{4|R_{i}|}\sum_{\beta\in R_{i}}
\alpha(\beta^{\vee})^{2}\\
&=\frac{1}{4|R_{i}|}\sum_{\beta\in R_{i}}\beta(\alpha^{\vee})^{2}=\frac{1}{4|R_{i}|}\cdot
4h_{i}\frac{\parallel\pi_{\mathfrak{h}_{i}}(\alpha^{\vee})\parallel_{a}^{2}}{\parallel \text{coroot} \parallel_{a}^{2}}
=\frac{\parallel\pi_{\mathfrak{h}_{i}}(\alpha^{\vee})\parallel_{a}^{2}}{\dim \mathfrak{h}_{i}\parallel \text{coroot} \parallel_{a}^{2}}.
\end{align*}
We can see that the trace of $E_{O}$ on $\mathfrak{h}_{i}$ ($=\lambda_{i,O}\dim \mathfrak{h}_{i}$) is the cosine squared of the angle
between $\alpha^{\vee}$ and $\pi_{\mathfrak{h}_{i}}(\alpha^{\vee})$.

Then we look at these orbits. Let $\tilde{\alpha}$ be the highest root of $R$ relative to the root basis $\mathfrak{B}$
and put $n_{p}:=\tilde{\alpha}(p)$. By inspection one finds that for $c=1,2,\cdots,n_{p}$ the set of $\alpha\in R$
with $\alpha(p)=c$ make up a single $W_{p}$-orbit $O(c)$ and that the orthogonal projection $O(c)_{i}$ of $O(c)$ in
$\mathfrak{h}_{i}^{*}$ is either $\{0\}$ or the orbit of a fundamental weight of $R_{i}$. So the $W_{p}$-orbit $O(c)$ projects in
$\bar{\mathfrak{h}}$ bijectively onto $\prod_{i} O(c)_{i}$. We also see that there is a unique
$\alpha(c)\in O(c)$ such that $\alpha(c)$
defines a fundamental coweight in $\mathfrak{h}_{i}$ relative to $\mathfrak{B}_{i}:=\mathfrak{B}\cap R_{i}$ or $0$. Therefore, our $U_{p}$
is proportional to
\[ E:=\sum_{c=1}^{n_{p}}c|O(c)|E_{O(c)}.\]

\begin{example}
We do a branch point of type $E_{7}$. Let $p$ be chosen corresponding to $\alpha_{4}$:
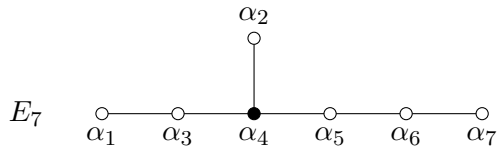
\[
\begin{tikzpicture}

	\draw (0,0) -- (5,0);
	\draw (2,0) -- (2,1);
	
	\draw[fill=white] (0,0) circle(.08);
	\draw[fill=white] (1,0) circle(.08);
	\draw[fill=black] (2,0) circle(.08);
	\draw[fill=white] (2,1) circle(.08);
	\draw[fill=white] (3,0) circle(.08);
	\draw[fill=white] (4,0) circle(.08);
	\draw[fill=white] (5,0) circle(.08);

	\node at (-1,0) {$E_7$};
	\node at (0,-.3) {$\alpha_1$};
	\node at (1,-.3) {$\alpha_3$};
	\node at (2,-.3) {$\alpha_4$};
	\node at (3,-.3) {$\alpha_5$};
	\node at (4,-.3) {$\alpha_6$};
	\node at (5,-.3) {$\alpha_7$};
    \node at (2,1.3) {$\alpha_2$};
	
\end{tikzpicture}
\]
The decomposition of $\mathfrak{B}-\alpha_{4}$ into irreducible root basis is $\mathfrak{B}_{1}:=\{\alpha_{1},\alpha_{3}\}$,
$\mathfrak{B}_{2}:=\{\alpha_{2}\}$, and $\mathfrak{B}_{3}:=\{\alpha_{5},\alpha_{6},\alpha_{7}\}$ (of type $A_{2}$, $A_{1}$
and $A_{3}$
respectively). The highest root is $\tilde{\alpha}=\alpha_{1}+2\alpha_{2}+3\alpha_{3}+4\alpha_{4}+3\alpha_{5}+2\alpha_{6}
+\alpha_{7}$. Since $\tilde{\alpha}(p)=4$, we have $4$ corresponding $W_{p}$-orbits and each of them is represented by
\begin{align*}
\alpha(1)&:=\alpha_{1}+\alpha_{2}+\alpha_{3}+\alpha_{4}+\alpha_{5}+\alpha_{6}+\alpha_{7},\\
\alpha(2)&:=\alpha_{1}+\alpha_{2}+2\alpha_{3}+2\alpha_{4}+2\alpha_{5}+\alpha_{6}+\alpha_{7},\\
\alpha(3)&:=\alpha_{1}+\alpha_{2}+2\alpha_{3}+3\alpha_{4}+2\alpha_{5}+\alpha_{6}+\alpha_{7},\\
\alpha(4)&:=\alpha_{1}+2\alpha_{2}+3\alpha_{3}+4\alpha_{4}+3\alpha_{5}+2\alpha_{6}+\alpha_{7}.
\end{align*}

For the case $c=1$, observe that $\alpha(1)$ defines the coweight sum $p_{1}(\mathfrak{B}_{1})+p_{2}(\mathfrak{B}_{2})+p_{7}(\mathfrak{B}_{3})$ in
$\bar{\mathfrak{h}}$. We have
\begin{align*}
\frac{\parallel \pi_{p}(\alpha(1))\parallel^{2}}{\parallel \alpha(1)\parallel^{2}}=\frac{1}{24},\;
\frac{\parallel p_{1}(\mathfrak{B}_{1})\parallel^{2}}{\parallel \alpha(1)\parallel^{2}}=\frac{1}{3},\;
\frac{\parallel p_{2}(\mathfrak{B}_{2})\parallel^{2}}{\parallel \alpha(1)\parallel^{2}}=\frac{1}{4},\;
\frac{\parallel p_{7}(\mathfrak{B}_{3})\parallel^{2}}{\parallel \alpha(1)\parallel^{2}}=\frac{3}{8}
\end{align*}
(summing to $1$). So the eigenvalues of $E_{O(1)}$ are $(\frac{1}{24},\frac{1}{6},\frac{1}{4},\frac{1}{8})$. It's easy to see
that $|O(1)|=2\cdot 3\cdot 4=24$ and so $|O(1)|E_{O(1)}$ has eigenvalues $(1,4,6,3)$.

For the case $c=2$, observe that $\alpha(2)$ defines the coweight sum $p_{3}(\mathfrak{B}_{1})+p_{6}(\mathfrak{B}_{3})$ in $\bar{\mathfrak{h}}$. We have
\begin{align*}
\frac{\parallel \pi_{p}(\alpha(2))\parallel^{2}}{\parallel \alpha(2)\parallel^{2}}=\frac{1}{6},\quad
\frac{\parallel p_{3}(\mathfrak{B}_{1})\parallel^{2}}{\parallel \alpha(2)\parallel^{2}}=\frac{1}{3},\quad
\frac{\parallel p_{6}(\mathfrak{B}_{3})\parallel^{2}}{\parallel \alpha(2)\parallel^{2}}=\frac{1}{2}
\end{align*}
(summing to $1$). So the eigenvalues of $E_{O(2)}$ are $(\frac{1}{6},\frac{1}{6},0,\frac{1}{6})$. It's easy to see
that $|O(2)|=3\cdot 6=18$ and so $2|O(2)|E_{O(2)}$ has eigenvalues $(6,6,0,6)$.

For the case $c=3$, observe that $\alpha(3)$ defines the coweight sum $p_{2}(\mathfrak{B}_{2})+p_{7}(\mathfrak{B}_{3})$ in $\bar{\mathfrak{h}}$. We have
\begin{align*}
\frac{\parallel \pi_{p}(\alpha(3))\parallel^{2}}{\parallel \alpha(3)\parallel^{2}}=\frac{3}{8},\quad
\frac{\parallel p_{2}(\mathfrak{B}_{2})\parallel^{2}}{\parallel \alpha(3)\parallel^{2}}=\frac{1}{4},\quad
\frac{\parallel p_{7}(\mathfrak{B}_{3})\parallel^{2}}{\parallel \alpha(3)\parallel^{2}}=\frac{3}{8}
\end{align*}
(summing to $1$). So the eigenvalues of $E_{O(3)}$ are $(\frac{3}{8},0,\frac{1}{4},\frac{1}{8})$. It's easy to see
that $|O(3)|=2\cdot 4=8$ and so $3|O(3)|E_{O(3)}$ has eigenvalues $(9,0,6,3)$.

For the case $c=4$, observe that $\alpha(4)$ defines the coweight sum $p_{1}(\mathfrak{B}_{1})$ in $\bar{\mathfrak{h}}$. We have
\begin{align*}
\frac{\parallel \pi_{p}(\alpha(4))\parallel^{2}}{\parallel \alpha(4)\parallel^{2}}=\frac{2}{3},\quad
\frac{\parallel p_{1}(\mathfrak{B}_{1})\parallel^{2}}{\parallel \alpha(4)\parallel^{2}}=\frac{1}{3}
\end{align*}
(summing to $1$). So the eigenvalues of $E_{O(4)}$ are $(\frac{2}{3},\frac{1}{6},0,0)$. It's easy to see
that $|O(4)|=3$ and so $4|O(4)|E_{O(4)}$ has eigenvalues $(8,2,0,0)$.

We conclude that $U_{p}=-kE=-k\sum_{c=1}^{4}c|O(c)|E_{O(c)}$ has as eigenvalues the system
$(-24k,-12k,-12k,-12k)$. In particular,
$a(p,p)=144k^{2}$.
\end{example}

Using this way, we have the eigenvalues for type $E_{n}$ as follows:

\begin{table}[H]
\caption{Eigenvalues for type $E_{n}$}
\vspace{5mm}
\begin{tabular}{|c|c|c|}
\hline
\multicolumn{3}{|c|}{$E_{6}$} \\
\hline
p & eigenvalues & multiplicities \\
\hline
$\varpi_{1}^{\vee}$,$\varpi_{6}^{\vee}$ & $(-4k,-2k)$ & $(1,6)$ \\
\hline
$\varpi_{2}^{\vee}$ & $(-4k,-3k)$ & $(1,6)$ \\
\hline
$\varpi_{3}^{\vee}$,$\varpi_{5}^{\vee}$ & $(-5k,-4k)$ & $(2,5)$ \\
\hline
$\varpi_{4}^{\vee}$ & $-6k$ & $7$ \\
\hline
\end{tabular}
\end{table}

\begin{table}[H]
\begin{tabular}{|c|c|c|}
\hline
\multicolumn{3}{|c|}{$E_{7}$} \\
\hline
p & eigenvalues & multiplicities \\
\hline
$\varpi_{1}^{\vee}$ & $(-6k,-4k)$ & $(1,7)$ \\
\hline
$\varpi_{2}^{\vee}$ & $(-7k,-6k)$ & $(1,7)$ \\
\hline
$\varpi_{3}^{\vee}$ & $(-9k,-8k)$ & $(2,6)$ \\
\hline
$\varpi_{4}^{\vee}$ & $-12k$ & $8$ \\
\hline
$\varpi_{5}^{\vee}$ & $(-9k,-10k)$ & $(5,3)$ \\
\hline
$\varpi_{6}^{\vee}$ & $(-6k,-8k)$ & $(6,2)$ \\
\hline
$\varpi_{7}^{\vee}$ & $(-3k,-6k)$ & $(7,1)$ \\
\hline
\end{tabular}
\end{table}

\begin{table}[H]
\begin{tabular}{|c|c|c|}
\hline
\multicolumn{3}{|c|}{$E_{8}$} \\
\hline
p & eigenvalues & multiplicities \\
\hline
$\varpi_{1}^{\vee}$ & $(-12k,-10k)$ & $(1,8)$ \\
\hline
$\varpi_{2}^{\vee}$ & $(-16k,-15k)$ & $(1,8)$ \\
\hline
$\varpi_{3}^{\vee}$ & $(-21k,-20k)$ & $(2,7)$ \\
\hline
$\varpi_{4}^{\vee}$ & $-30k$ & $9$ \\
\hline
$\varpi_{5}^{\vee}$ & $(-24k,-25k)$ & $(5,4)$ \\
\hline
$\varpi_{6}^{\vee}$ & $(-18k,-20k)$ & $(6,3)$ \\
\hline
$\varpi_{7}^{\vee}$ & $(-12k,-15k)$ & $(7,2)$ \\
\hline
$\varpi_{8}^{\vee}$ & $(-6k,-10k)$ & $(8,1)$ \\
\hline
\end{tabular}
\end{table}

\hfill\\

These computation show that there are at most 2 eigenvalues along the toric divisor for any root system.

\begin{remark}
We notice from the above computation that if we extend the bilinear symmetric form $a$ on $\mathfrak{h}$ to the space
$\mathfrak{h}\oplus\mathbb{C}$
by the way such that
\begin{eqnarray*}
\left\{
\begin{aligned}
&a(q,t\frac{\partial}{\partial t})=0 \quad \text{for} \quad \forall q\in\mathfrak{h} \\
&a(t\frac{\partial}{\partial t},t\frac{\partial}{\partial t})=-1,
\end{aligned}
\right.
\end{eqnarray*}
the two eigenvectors in the space which is spanned by the
fundamental coweight $p$ and $t\frac{\partial}{\partial t}$ are perpendicular
to each other with respect to this $a$.
\end{remark}

We also have the dilatation field as follows.
\begin{theorem}
Suppose an affine structure on $H^{\circ}\times\mathbb{C}^{\times}$ is given by the torsion free flat connection $\tilde{\nabla}$
defined by ($\ref{eqn:affine-connection}$), then the vector field $t\frac{\partial}{\partial t}$ is in fact a
dilatation field on $H^{\circ}\times\mathbb{C}^{\times}$ with factor $\lambda=1$.
\end{theorem}

\begin{proof}
It's a straightforward computation. Suppose a local vector field $\tilde{v}$ on $H^{\circ}\times\mathbb{C}^{\times}$ is of the form
$\tilde{v}:=v+\mu t\frac{\partial}{\partial t}$ where $v$ is a vector field on $H^{\circ}$, we have
\begin{align*}
\tilde{\nabla}_{\tilde{v}}(t\frac{\partial}{\partial t})&=\tilde{\nabla}_{v+\mu t\frac{\partial}{\partial t}}^{0}(t\frac{\partial}{\partial t})
-\tilde{\Omega}_{v+\mu t\frac{\partial}{\partial t}}^{*}(t\frac{\partial}{\partial t})\\
&=0+\sum_{\alpha_{i}\in\mathfrak{B}}\alpha_{i}(v)\partial_{p_{i}}+\mu t\frac{\partial}{\partial t}\\
&=\tilde{v}
\end{align*}
since $t\frac{\partial}{\partial t}$ is flat with respect to $\tilde{\nabla}^{0}$.
\end{proof}

We then could decompose the vector bundle $\mathcal{E}$ (with its flat connection) naturally according to these eigenspaces.

\begin{lemma}
The vector bundle $\mathcal{E}$ (with its flat connection) decomposes naturally according to the images
in $\mathbb{C}/\mathbb{Z}$ of
the eigenvalues of the residue endomorphism: $\mathcal{E}=\bigoplus_{\zeta\in\mathbb{C}^{\times}}\mathcal{E}^{\zeta}$,
where $\mathcal{E}^{\zeta}$ has a residue
endomorphism whose eigenvalues $\nu$ are such that $\exp(2\pi\sqrt{-1}\nu)=\zeta$.

And suppose $D$ is a hypersurface in $N$. Then the affine structure on $N-D$ has a degeneration along $D^{\circ}$ which could be
decomposed naturally as above with logarithmic exponent $\nu-1$ in each corresponding eigenspace.
\end{lemma}

\begin{proof}
\cite{Couwenberg-Heckman-Looijenga}.
\end{proof}

\subsection{Reflection representation}\label{subsec:reflection-representation}

We recall that $R\subset \mathfrak{a}^{\ast}$ is a reduced irreducible finite root system where $\mathfrak{a}^{\ast}$ is a Euclidean
vector space of dimension $n$. Let $\mathfrak{a}=\mathrm{Hom}(\mathfrak{a}^{\ast},\mathbb{R})$ be the dual Euclidean
vector space, and let $R^{\vee}$ in $\mathfrak{a}$ be the dual root system and denote the corresponding
coroot lattice \index{Lattice!coroot}
by $Q^{\vee}=\mathbb{Z}R^{\vee}$. We then have the weight lattice $P=\mathrm{Hom}(Q^{\vee},\mathbb{Z})$
of $R$ in $\mathfrak{a}^{\ast}$. The torus having $P$ as (rational) character lattice,
$H'=\mathrm{Hom}(P,\mathbb{C}^{\times})$ is often called
the simply connected torus. \index{Torus!simply connected}
Put $H'^{\circ}=H'-\cup_{\alpha\in R_{+}}H'_{\alpha}$
where $H'_{\alpha}=\{h\in H'\mid e^{\alpha}(h)=1\}$. Let
\[ C=\{ h\in H'\mid e^{\alpha}(h)=1 \; \text{for all} \; \alpha\in R\}\cong P^{\vee}/Q^{\vee}, \]
so the adjoint torus $H$ is just $H'/C$. Then as discussed in Theorem $\ref{thm:specialeqn}$ the special differential equation system ($\ref{eqn:specialeqn}$) associated with the root system
$R$ gives a $W^{\prime}$-invariant projective structure \index{Projective structure} on $H'^{\circ}$ where $W'=W\rtimes C$
is the extended Weyl group.

For the example of the root system of type $A_{1}$ this equation specializes to (take $u=v=\alpha^{\vee}/2$, $k'=0$
with variable $z=e^{\alpha}(h)$ and derivative $\theta=z\partial$)
\[(\theta^{2}+k\frac{1+z^{-1}}{1-z^{-1}}\theta+\frac{1}{4}k^{2})f(z)=0\]
which is the pull back of the classical Euler-Gauss hypergeometric equation under a
suitable coordinate transformation \cite{Heckman-Looijenga-2010}.

We shall now construct the reflection representation \index{Reflection representation}
of the affine Artin group $\mathrm{Art}(M)$ with generators $\sigma_{0},\cdots,
\sigma_{n}$ and braid relations
$$\underbrace{\sigma_{i}\sigma_{j}\sigma_{i}\cdots}_{m_{ij}}=\underbrace{\sigma_{j}\sigma_{i}\sigma_{j}\cdots}_{m_{ij}}$$
for all $i \neq j$ where both members are words comprising $m_{ij}$ letters. These results can be traced back to
Coxeter and Kilmoyer, see e.g. \cite{Coxeter} and \cite{Curtis-Iwahori-Kilmoyer}.

First we need to investigate
what the two complex reflections \index{Complex reflection}
look like if they satisfy a braid relation.

\begin{proposition}\label{prop:refleigen}
Let $s_{1},s_{2}\in\mathrm{GL}_{2}(\mathbb{C})$ be the complex reflections as follows.
\[
\left( \begin{array}{cc}
-q_{1} &  d_{1}\\
0  &  1\\
\end{array}
\right)
,\quad
\left( \begin{array}{cc}
1 &  0\\
d_{2}  &  -q_{2}\\
\end{array}
\right)
\]
where $q_{1},q_{2}\in\mathbb{C}^{\times}$.

If $m=2r+1$ ($r\geq 1$) is odd, then $s_{1}$ and $s_{2}$ satisfy the braid relation of length $m$
\[ (s_{1}s_{2})^{r}s_{1}=(s_{2}s_{1})^{r}s_{2}  \]
if and only if
\[ q_{1}=q_{2}(=q \; \text{say}),\; d_{1}d_{2}=(2+\xi+\xi^{-1})q \; \text{with} \; \xi^{m}=1,\xi\neq 1. \]

If $m=2r$ ($r\geq 1$) is even, then $s_{1}$ and $s_{2}$ satisfy the braid relation of length $m$
\[ (s_{1}s_{2})^{r}=(s_{2}s_{1})^{r}  \]
if and only if
\[ d_{1}=d_{2}=0 \;\text{for} \;m=2; \]
and
\[ d_{1}d_{2}=q_{1}+q_{2}+(\xi+\xi^{-1})q_{1}^{\frac{1}{2}}q_{2}^{\frac{1}{2}} \;\text{with} \; \xi^{m}=1,\xi^{2}\neq 1
\; \text{for} \; m\geq 4.\]
\end{proposition}

\begin{proof}
It is obvious that $\det (s_{1})=-q_{1}$ and $\det (s_{2})=-q_{2}$.

Let us do the odd case first. Suppose $m=2r+1$ is odd. Put $T_{1}=(s_{1}s_{2})^{r}s_{1}$ and $T_{2}=(s_{2}s_{1})^{r}s_{2}$.
Suppose the braid relation holds, i.e., $T_{1}=T_{2}$ (=$T$ say). It is clear that $d_{1}d_{2}\neq 0$, for otherwise it
contradicts the braid relation. So the group generated by $s_{1}$ and $s_{2}$ acts irreducibly on $\mathbb{C}^{2}$.
We also have that $Ts_{1}=s_{2}T$ and $Ts_{2}=s_{1}T$ which follows that $s_{1}$ and $s_{2}$ are conjugated, hence we have
$q_{1}=q_{2}$ (=$q$ say). Moreover, $T^{2}=(s_{1}s_{2})^{m}=(s_{2}s_{1})^{m}$ commutes with both $s_{1}$ and $s_{2}$, and
hence is a scalar matrix by Schur's lemma, i.e., $T^{2}-q^{m}I=0$. But $T$ is not a scalar matrix, so it has eigenvalues
$\pm \lambda$ with $\lambda^{2}=q^{m}$ since $\det(T)=-q^{m}$. Hence $s_{1}s_{2}$ has eigenvalues $\xi q,\xi^{-1}q$ with
$\xi^{m}=1,\xi\neq 1$. Since we have $\mathrm{tr}(s_{1}s_{2})=d_{1}d_{2}-q_{1}-q_{2}$, we have also that
$d_{1}d_{2}=(2+\xi+\xi^{-1})q$ with $\xi^{m}=1,\xi\neq 1$.

Conversely, suppose $q_{1}=q_{2}$ ($=q$ say), $d_{1}d_{2}=(2+\xi+\xi^{-1})q$ with $\xi^{m}=1.\xi\neq 1$. Then
$\det(s_{1}s_{2})=\det(s_{2}s_{1})=q^{2}$ and $\mathrm{tr}(s_{1}s_{2})=\xi q+\xi^{-1} q$,
so $s_{1}s_{2}$ and $s_{2}s_{1}$ have eigenvalues $\xi q,\xi^{-1} q$ with $\xi^{m}=1,\xi\neq 1$. Hence
$T_{1}T_{2}=T_{2}T_{1}=q^{m}I$. We notice that the matrix
$
\left( \begin{array}{cc}
0 &  d_{1}\\
d_{2}  &  0\\
\end{array}
\right)
$
conjugates $s_{1}$ to $s_{2}$, and therefore also $T_{1}$ to $T_{2}$. If $T_{1}$ and $T_{2}$ have eigenvalues
$\lambda_{1},\lambda_{2}$, then $\lambda_{1}\lambda_{2}=\det(T_{1})=-q^{m}$. We also have
\[
\lambda_{1}+\lambda_{2}=\mathrm{tr}(T_{1})=\mathrm{tr}(T_{2})=q^{m}(\lambda_{1}^{-1}+\lambda_{2}^{-1})=-(\lambda_{1}+\lambda_{2})
\]
so $\lambda_{1}=-\lambda_{2}$. In turn this implies $T_{1}^{2}=q^{m}I$ and hence the braid relation
$T_{1}=T_{2}$ follows.

Now suppose $m=2r$ ($r\geq 1$) is even. In case $m=2$ it is direct to verify that the relation $s_{1}s_{2}=s_{2}s_{1}$
holds if and only if $d_{1}=d_{2}=0$. Therefore assume $r\geq 2$ now. Then the group generated by $s_{1}$ and $s_{2}$
acts irreducibly on $\mathbb{C}^{2}$. Put $T_{1}=(s_{1}s_{2})^{r}$ and $T_{2}=(s_{2}s_{1})^{r}$. The braid relation
$T_{1}=T_{2}$ ($=T$ say) implies that $T$ commutes with $s_{1}$ and $s_{2}$, and so $T$ is a scalar matrix, $\lambda I$ say,
with $\lambda^{2}=(q_{1}q_{2})^{r}$. Hence the matrix $s_{1}s_{2}$ has eigenvalues $\xi^{\pm 1}q_{1}^{\frac{1}{2}}
q_{2}^{\frac{1}{2}}$ with $\xi^{m}=1,\xi^{2}\neq 1$. The trace computation of $s_{1}s_{2}$ shows that
$d_{1}d_{2}=q_{1}+q_{2}+(\xi+\xi^{-1})q_{1}^{\frac{1}{2}}q_{2}^{\frac{1}{2}}$.

Conversely, suppose the above equality holds. Then the eigenvalues of $s_{1}s_{2}$ and $s_{2}s_{1}$ are
$\xi^{\pm 1}q_{1}^{\frac{1}{2}}q_{2}^{\frac{1}{2}}$, and $T_{1}=T_{2}=\xi^{r}(q_{1}^{\frac{1}{2}}q_{2}^{\frac{1}{2}})^{r}I$.
\end{proof}

In fact, finite complex reflection groups are already classified by Shephard and Todd in 1954 \cite{Shephard-Todd}.

Let $M=(m_{ij})_{0\leq i,j \leq n}$ be the affine Coxeter matrix
\index{Coxeter!matrix}
associated with the extended Dynkin diagram of the \index{Dynkin diagram!extended}
affine root system \index{Root system!affine}
$\tilde{R}$ associated with $R$. If $\alpha_{1},\cdots,\alpha_{n}$ are the simple roots in $R_{+}$,
then $a_{0}=-\tilde{\alpha},a_{1}=\alpha_{1},\cdots,a_{n}=\alpha_{n}$ are the simple roots in $\tilde{R}_{+}$ with
$\tilde{\alpha}$ the highest root in $R_{+}$.

Recall that $K$ is the space of multiplicity parameters for $R$ defined by
\[  \kappa=(k_{\alpha})_{\alpha\in R}\in \mathbb{C}^{R}  \]
where $\kappa$ is invariant on $W$-orbits in $R$.
It is clear that $K$ is isomorphic to $\mathbb{C}^{r}$ as a $\mathbb{C}$-vector space if $r$ is the number
of $W$-orbits in $R$ (i.e., $r=1$ or $2$). Like in Section $\ref{sec:affine-structures}$, we shall
still sometimes write $k_{i}$ instead of $k_{\alpha_{i}}$
if $\{\alpha_{1},\cdots,\alpha_{n}\}$ is a basis of simple roots in $R_{+}$. And sometimes we
also write $k$ for $k_{1}$ and $k'$ for $k_{n}$ if
$\alpha_{n}\notin W\alpha_{1}$ when no confusion can arise. But note that $k'$ has a different meaning for
type $A_{n}$, which can be seen from Remark $\ref{rem:k'-for-An}$.

For $R$ of any type other than $A_{n}$, let $q_{i}^{\frac{1}{2}}$
and $s_{ij}=s_{ji}$ for $0\leq i\neq j\leq n$ be indeterminants with additional relations
\begin{equation*}
\left\{
\begin{aligned}
&s_{ij}=0
& \text{if} \quad m_{ij}=2 \\
& q_{i}^{\frac{1}{2}}=q_{j}^{\frac{1}{2}} \; \text{and} \; s_{ij}=1
& \text{if} \quad m_{ij}=3 \\
& s_{ij}^{2}=q_{i}^{\frac{1}{2}}q_{j}^{-\frac{1}{2}}+q_{i}^{-\frac{1}{2}}q_{j}^{\frac{1}{2}}+2\cos(\frac{2\pi}{m_{ij}})
& \text{if} \quad m_{ij}\geq 4
\end{aligned}
\right.
\end{equation*}
and let $\mathcal{A}$ be the domain $\mathbb{C}[q_{i}^{\frac{1}{2}},q_{i}^{-\frac{1}{2}},s_{ij}\mid 0\leq i\neq j\leq n]$.
Let $\bar{\quad}$ be the involution of $\mathcal{A}$ defined by $\bar{q}_{i}^{\frac{1}{2}}=q_{i}^{-\frac{1}{2}}$ and
$\bar{s}_{ij}=s_{ij}$. Let $K^{\prime}$ be the space of restricted multiplicity parameters defined by
\[ K'=\{ \kappa=(k_{\alpha_{i}})\in K \mid k_{i}\in (-\frac{1}{2},\frac{1}{2}),| k_{i}-k_{j}| <1-\frac{2}{m_{ij}}
\; \text{if} \; m_{ij}\geq 4 \; \text{and even}
\}. \]
For $z\in \mathbb{C}\setminus(-\infty,0]$, let $z^{\frac{1}{2}}$ denote the branch of the square root with
$1^{\frac{1}{2}}=1$. If $\kappa\in K^{\prime}$ is a restricted multiplicity parameter on $R$ then the substitutions
$$q_{j}^{\frac{1}{2}}=\exp(-\pi \sqrt{-1} k_{j}), \quad s_{ij}=(2\cos\pi (k_{i}-k_{j})+2\cos(\frac{2\pi}{m_{ij}}))^{\frac{1}{2}}$$
for all $j$ and $i\neq j$ with $m_{ij}\geq 3$ induce a homomorphism $\mathcal{A}\rightarrow \mathbb{C}$ called specialization
of $\mathcal{A}$ at $\kappa\in K^{\prime}$.

The root system $R$ of type $A_{n}$ is somewhat peculiar due to the fact that the extended Dynkin diagram is a cycle
rather than a tree. For $R$ of type $A_{n}$, we take
$$s_{0,1}=\cdots =s_{n-1,n}=s_{n,0}=q^{\prime -\frac{1}{2}}, \qquad
s_{1,0}=\cdots =s_{n,n-1}=s_{0,n}=q^{\prime \frac{1}{2}}$$
and put $\bar{q}^{\frac{1}{2}}=q^{-\frac{1}{2}}$, $\bar{q^{\prime}}^{\frac{1}{2}}=q^{\prime -\frac{1}{2}}$
(with $q^{\frac{1}{2}}=q_{i}^{\frac{1}{2}}$ for $i=0,\cdots,n$). Let the restricted parameter space $K^{\prime}$
be defined by
\[ K^{\prime}=\lbrace (k,k')\mid k\in(-\frac{1}{2},\frac{1}{2}), k'\in(-\frac{1}{2},\frac{1}{2}) \rbrace \]
and let the specialization of $\mathcal{A}$ at $\kappa\in K^{\prime}$ be given by the substitutions
$$q^{\frac{1}{2}}=\exp(-\pi \sqrt{-1} k),\quad q^{\prime \frac{1}{2}}=\exp(-\pi \sqrt{-1} k^{\prime}).$$

\begin{remark}
In all cases the involution $\bar{\quad}$ of $\mathcal{A}$ becomes complex conjugation under specialization.
\end{remark}

Now let $e_{i}(i=0,\cdots, n)$ be the standard basis of $\mathcal{A}^{n+1}$ and define a Hermitian form on $\mathcal{A}^{n+1}$
with Gram matrix of the standard basis given by
\begin{equation}\label{eqn:hem}
h_{ij}=
\left\{
\begin{aligned}
&q_{i}^{\frac{1}{2}}+q_{i}^{-\frac{1}{2}}   &
&\text{if} \quad i=j \\
&-s_{ij}               &
&\text{if} \quad i\neq j
\end{aligned}
\right.
\end{equation}
Indeed we have $h^{\dagger}=h$ with $h^{\dagger}=\bar{h}^{t}$. The unitary reflection $T_{j}$ of $\mathcal{A}^{n+1}$
having $e_{j}$ as eigenvector with eigenvalue $-q_{j}$ satisfies
\begin{equation*}
T_{j}(e_{i})=e_{i}-q_{i}^{\frac{1}{2}}h_{ij}e_{j} \quad \text{for all} \; i.
\end{equation*}
From this identity we can see that $T_{j}$ also satisfies the quadratic relation
$$(T_{j}-1)(T_{j}+q_{j})=0$$
as well as the braid relation
$$\underbrace{T_{i}T_{j}T_{i}\cdots}_{m_{ij}}=\underbrace{T_{j}T_{i}T_{j}\cdots}_{m_{ij}}$$
for all $i\neq j$. Therefore there exists a unique (unitary) representation
\begin{equation}\label{eqn:rep}
\begin{split}
\rho^{\mathrm{refl}} :\mathrm{Art}(M)&\rightarrow \mathrm{GL}_{n+1}(\mathcal{A}) \\
\sigma_{j}&\mapsto T_{j}
\end{split}
\end{equation}
and this is the reflection representation of $\mathrm{Art}(M)$.

\begin{theorem}\label{thm:uniquerep}
For $R$ an irreducible root system not of type $A_{n}$ the reflection representation $\rho^{\mathrm{refl}}$ in
($\ref{eqn:rep}$) is up to equivalence the unique irreducible representation
$\rho' :\mathrm{Art}(M) \rightarrow \mathrm{GL}_{n+1}(\mathcal{A})$
such that $\rho'(\sigma_{j})$ is a reflection with $\det (\rho'(\sigma_{j}))=-q_{j}$
for all $j$.

For $R$ of type $A_{n}$ there is a one parameter family of such representations, depending on the
additional parameter $q'^{\frac{1}{2}}$, and for each additional $q'^{\frac{1}{2}}$ this is again the
reflection representation ($\ref{eqn:rep}$).
\end{theorem}

\begin{proof}
Suppose $\rho': \mathrm{Art}(M) \rightarrow \mathrm{GL}_{n+1}(\mathcal{A})$ is an irreducible representation
satisfying that $\rho'(\sigma_{j})$ is a reflection with $\det (\rho'(\sigma_{j}))=-q_{j}$
for all $j$. Then there exists a basis $\{e_{0},e_{1},\cdots,e_{n}\}$ of $\mathcal{A}^{n+1}$ such that
\begin{equation*}
\rho'({\sigma_{j}})(e_{j})=-q_{j}e_{j},\quad \rho'(\sigma_{j})(e_{i})=e_{i}+d_{ij}e_{j} \; \text{for} \; i\neq j.
\end{equation*}
By Proposition $\ref{prop:refleigen}$ we have $d_{ij}=d_{ji}=0$ if $m_{ij}=2$; $d_{ij}d_{ji}=q_{i}=q_{j}$ if $m_{ij}=3$;
$d_{ij}d_{ji}=q_{i}+q_{j}$ if $m_{ij}=4$ and $d_{ij}d_{ji}=q_{i}+q_{j}-q_{i}^{\frac{1}{2}}q_{j}^{\frac{1}{2}}$ if $m_{ij}=6$.
So we have $d_{ij}d_{ji}=s_{ij}s_{ji}q_{i}^{\frac{1}{2}}q_{j}^{\frac{1}{2}}$ if $m_{ij}\geq 3$. By rescaling the basis
$e_{0},e_{1},\cdots,e_{n}$, we would like to arrive at $d_{ij}=s_{ij}q_{j}^{\frac{1}{2}}$ for all $i\neq j$.
This can be done if the extended Dynkin diagram is a tree, by induction on $n$ assuming the $n$-th node to be
an extremal node. Therefore the first part of the theorem follows.

For $R$ of type $A_{n}$ (with $q_{i}^{\frac{1}{2}}=q^{\frac{1}{2}}$ for $i=0,1,\cdots,n$), we have
$d_{ij}=0$ if $2\leq |i-j|\leq n-1$ and $d_{01}d_{10}=d_{12}=d_{21}=\cdots =d_{n-1,n}d_{n,n-1}=d_{n,0}d_{0,n}=q$.
Rescaling the basis $e_{0},e_{1},\cdots, e_{n}$ suitably we can arrange that
$z_{01}=\cdots =z_{n-1,n}=z_{n,0}=q^{\frac{1}{2}}q'^{-\frac{1}{2}}$ and
$z_{10}=\cdots=z_{n,n-1}=z_{0,n}=q^{\frac{1}{2}}q'^{\frac{1}{2}}$ with $q'^{\frac{1}{2}}$ some additional
parameter. This also brings us back to the reflection representation ($\ref{eqn:rep}$).
\end{proof}

The special hypergeometric system ($\ref{eqn:specialeqn}$) in Theorem $\ref{thm:specialeqn}$ is defined on the complex manifold $H^{\circ}$ and is invariant for $W$. Therefore it lives on the complex orbifold $W\backslash H^{\circ}$
as well.

This in turn implies that the monodromy of the system ($\ref{eqn:specialeqn}$) is a homomorphism:
\begin{equation}
\rho^{\mathrm{mon}}:\pi_{1}^{\mathrm{orb}}(W\backslash H^{\circ},Wp_{0})\rightarrow \mathrm{GL}_{n+1}(\mathbb{C})
\end{equation}
with $p_{0}\in D_{+}\subset H^{\circ}$ a fixed base point where $D_{+}$ is the fundamental alcove of $H^{\circ}$.

\begin{remark}
There is a natural isomorphism using Brieskorn's theorem \cite{Brieskorn}
\[ \pi_{1}(W\backslash H'^{\circ}, W'p_{0})\cong \mathrm{Art}(M) \]
and then there is a corresponding isomorphism
\[ \pi_{1}^{\mathrm{orb}}(W\backslash H^{\circ},Wp_{0})\cong \mathrm{Art}(M)\rtimes C \]
with $C$ viewed as group of diagram automorphisms of the extended Dynkin diagram of $R$ and
hence acting naturally on the generators of $\mathrm{Art}(M)$. The reflection representation ($\ref{eqn:rep}$) can be extended
in a natural way to a representation of $\mathrm{Art}(M)\rtimes C$. We write
$\mathrm{Art}'(M)=\mathrm{Art}(M)\rtimes C$.

If $\mathrm{Aut}(M)$ denotes the full group of diagram automorphisms of the
extended Dynkin diagram, then the reflection representation ($\ref{eqn:rep}$) extends in a natural way to a representation
of $\mathrm{Art}(M)\rtimes \mathrm{Aut}(M)$ with the exception of type $A_{n}$ for which case this extension
is only possible when $q'=1$.
\end{remark}

Then we identify the two representations as follows.

\begin{theorem}
The monodromy representation \index{Monodromy!representation}
of the special hypergeometric system given in Theorem $\ref{thm:specialeqn}$ for a parameter
$\kappa\in K'$ is equal to the specialization at $\kappa\in K'$ of the reflection representation of the (partially) extended
affine Artin group. In particular for $\kappa\in K^{\prime}$ the local solution space at $p_{0}\in D_{+}$
admits a Hermitian form $h(\kappa)$ invariant under monodromy.
\end{theorem}

\begin{proof}
For $\kappa\in K'$ being generic, the monodromy is easily seen to be irreducible.
From the computation on the eigenvalues of the residue endomorphisms of the connection along the mirrors,
the special hypergeometric system has exponents along the wall of $D_{+}$
corresponding to the simple root $\alpha_{j}$ equal to $1$ and $1-2k_{j}$ with multiplicities $n$ and $1$ respectively.
Therefore the monodromy $\rho^{\mathrm{mon}}(\sigma_{j})$ of the special hypergeometric
system corresponding to a half turn
around that wall is a complex reflection with a quadratic relation
\[ (\rho^{\mathrm{mon}}(\sigma_{j})-1)(\rho^{\mathrm{mon}}(\sigma_{j})+q_{j})=0  \]
with $q_{j}=\exp(-2\pi\sqrt{-1}k_{j})$. We then also have $\det(\rho^{\mathrm{mon}}(\sigma_{j}))=-q_{j}$.
The result hence follows by Theorem $\ref{thm:uniquerep}$, if for type $A_{n}$ we have
the specialization $q'=\exp(-2\pi\sqrt{-1}k')$.
\end{proof}

\begin{theorem} \label{theorem:det}
The specialization $\det(h(\kappa))$ at $\kappa\in K^{\prime}$ is given for type $ABCFG$ by
\begin{equation}
\det(h(\kappa))=-4\sin(\pi x)\sin(\pi y),
\end{equation}
with $(x,y)=((n+1)(k+k')/2,(n+1)(k-k')/2),((n-2)k+k',2k),((n-2)k+2k',k),
(k+k',2k+k'),((k+3k')/2,(k+k')/2)$ respectively,

and for type $D_{n}$ and $E_{n}$ by
\begin{equation}
\det(h(\kappa))=2^{n+1}\prod _{j=0}^{n}(\cos(\pi k)-\cos(\pi \tilde{m}_{j}/\tilde{h}))
\end{equation}
with $\tilde{h}$ and $\{ \tilde{m}_{j} \}$ given by
\begin{align*}
&D_{n}:\tilde{h}=2(n-2),            &
&\lbrace\tilde{m}_{j}\rbrace=\lbrace 0,2,\cdots ,2(n-2),(n-2),(n-2)\rbrace \\
&E_{6}:\tilde{h}=6,                   &
&\lbrace\tilde{m}_{j}\rbrace=\lbrace 0,2,2,3,4,4,6\rbrace \\
&E_{7}:\tilde{h}=12,                 &
&\lbrace\tilde{m}_{j}\rbrace=\lbrace 0,3,4,6,6,8,9,12\rbrace \\
&E_{8}:\tilde{h}=30,                 &
&\lbrace\tilde{m}_{j}\rbrace=\lbrace 0,6,10,12,15,18,20,24,30\rbrace
\end{align*}
\end{theorem}

\begin{proof}
The first identity is obtained directly from the classification theory of connected extended Dynkin diagrams, which
could be found on the online appendix \cite{Shen-2018}. And the
second identity appears as Exercise 4 of Ch. V, $\S$ 6 in Bourbaki \cite{Bourbaki}.
\end{proof}

\begin{corollary}\label{cor:hyperbolic-domain}
For $R$ of type $ABCFG$ put
\[ K'_{\mathrm{hyp}}=\{\kappa\in K'\mid 0<x<1,0<y<1\}   \]
and for type DE put
\[ K'_{\mathrm{hyp}}=(0,\frac{1}{n-2}) \quad \text{and} \quad (0,\frac{1}{n-3}) \]
respectively.
Then the monodromy representation has an invariant Hermitian form of Lorentz signature (n,1).
\end{corollary}

\begin{proof}
Observe that for $k\in \sqrt{-1}\mathbb{R}^{\times}$, $k^{\prime}=0$ (for $R$ of type $A$) and for
$k=k^{\prime}\in \sqrt{-1}\mathbb{R}^{\times}$ (for $R$ of other types), the form $h(\kappa)$ is positive definite,
and for $\kappa\in K'_{\mathrm{hyp}}$, one has $\det(h(\kappa))<0$. Since on the line $k^{\prime}=0$ (type $A$) and $k=k^{\prime}$
(other types) the function $\det(h(\kappa))$ has a double zero at the origin, the result follows.
\end{proof}

\subsection{The evaluation map}\label{subsec:evaluation-map}

Recall that $D_{u,v}^{\kappa}$ denotes the second order differential operator in ($\ref{eqn:specialeqn}$):
\[\partial_{u}\partial_{v}+\frac{1}
{2}\sum_{\alpha>0}
k_{\alpha}\alpha(u)\alpha(v)\frac{e^{\alpha}+1}{e^{\alpha}-1}\partial_{\alpha^{\vee}}
+\partial_{b^{\kappa}(u,v)}+a^{\kappa}(u,v). \]
Take $p\in H^{\circ}$ and $\kappa\in K$. Consider $D_{u,v}^{\kappa}$ as an operator on the stalk of holomorphic germs
$\mathcal{O}_{\kappa,p}$. Then the solutions form a free $\mathcal{O}_{\kappa}$ module of rank $n+1$. Hence the local
solutions of $D_{u,v}^{\kappa}f=0$ for $\forall u,v\in\mathfrak{h}$ near $p$ can be considered as a
vector bundle $\mathcal{F}_{p}$ over $K$. Any
$w\in W$ induces an isomorphism of vector bundle $\mathcal{F}_{p}$ and
$\mathcal{F}_{wp}$. Then we can identify all these vector bundles induced by a regular
$W$-orbit $S$. This yields a vector bundle $\mathcal{F}_{S}$ over $K$ of rank
$n+1$, the fiber of which is denoted by $\mathcal{F}_{S}(\kappa)$. According to the preceding section, we have
the following representation $\rho(\kappa)$ on the vector bundle $\mathcal{F}_{S}(\kappa)$ by specializing $\kappa$:
\begin{equation*}
\rho:\pi_{1}^{\mathrm{orb}}(W\backslash H^{\circ},S)\cong
\mathrm{Art}'(M)\rightarrow \mathrm{End}(\mathcal{F}_{S}).
\end{equation*}
Transposing $\rho$ yields a following representation
\begin{equation*}
\rho^{*}: \mathrm{Art}'(M)\rightarrow \mathrm{End}(\mathcal{F}_{S}^{*}).
\end{equation*}
Let $h^{*}(\kappa)$ be given as follows:
\begin{equation}\label{eqn:dual-herm}
h^{*}(\kappa)=\det(h(\kappa))(h(\kappa))^{-1}
\end{equation}
We then have the dual Hermitian form $h^{*}(\kappa)$ of $h(\kappa)$ for
the transpose $\rho^{*}(\kappa)$ if $\kappa\in K'$ is real valued.

\begin{lemma}
The dual Hermitian form $h^{*}(\kappa)$ of $h(\kappa)$ given as above is a nontrivial invariant Hermitian form for
the transpose $\rho^{*}(\kappa)$ if $\kappa\in K'$ is real valued. Moreover, if $h(\kappa)$ is positive definite, then
$h^{*}(\kappa)$ is also positive definite; if $h(\kappa)$ is parabolic, then $h^{*}(\kappa)$ is positive semidefinite with
$n$ dimensional kernel; and if $h(\kappa)$ is hyperbolic, then $h^{*}(\kappa)$ is of the signature $(1,n)$.
\end{lemma}

\begin{proof}
It's clear that $h^{*}(\kappa)$ is a nontrivial invariant Hermitian form for $\rho^{*}(\kappa)$ since $h(\kappa)$ is of rank
at least $n$ and then the matrix $h^{*}(\kappa)$ is of rank at least $1$. In fact, $h^{*}(\kappa)$ is just the minor matrix of
$h(\kappa)$ and we have $h^{*}(\kappa)h(\kappa)=\det(h(\kappa))I$, the statement easily follows.
\end{proof}

Since the differential operator $D_{u,v}^{\kappa}$ defines an affine structure on $H^{\circ}\times \mathbb{C}^{\times}$,
then the locally affine-linear functions which are of the form $c+tf$ by Lemma $\ref{lem:local-affine-functions}$ make up a
subsheaf $\mathrm{Aff}_{H^{\circ}\times \mathbb{C}^{\times}}$ in the structure sheaf
$\mathcal{O}_{H^{\circ}\times \mathbb{C}^{\times}}$. This locally free sheaf is of rank $n+2$ and contains the constants
$\mathbb{C}_{H^{\circ}\times \mathbb{C}^{\times}}$. Then the quotient $\mathrm{Aff}_{H^{\circ}\times \mathbb{C}^{\times}}/
\mathbb{C}_{H^{\circ}\times \mathbb{C}^{\times}}$ is a locally free sheaf whose underlying vector bundle is the cotangent bundle
of $H^{\circ}\times \mathbb{C}^{\times}$.
So there exists a multivalued evaluation map given by \index{Evaluation map}
\begin{equation*}
ev:K\times ((W\backslash H^{\circ})\times\mathbb{C}^{\times})\dashrightarrow \mathcal{F}_{S}^{*}
\end{equation*}
such that $ev(\kappa,p,t)(f)=tf(p)$,
and a corresponding projective evaluation map
\begin{equation*}
Pev:K\times (W\backslash H^{\circ})\dashrightarrow \mathbb{P}(\mathcal{F}_{S}^{*}).
\end{equation*}
In order to eliminate
the multivaluedness of this map, let us denote by $\widehat{W\backslash H^{\circ}}$ the
$\Gamma$-covering space \index{Covering!$\Gamma$-}
of $W\backslash H^{\circ}$ with
$\Gamma=\pi_{1}(W\backslash H^{\circ})/\mathrm{Ker}(Pr\circ\rho)$ the projective monodromy group. Here we write
$Pr:\mathrm{GL}(\mathcal{F}_{S}(\kappa))\twoheadrightarrow \mathrm{PGL}(\mathcal{F}_{S}(\kappa))$ for the natural projection.
In other words,
$\widehat{W\backslash H^{\circ}}$ is equal to $\mathrm{Ker}(Pr\circ\rho)\backslash (\widetilde{W\backslash H^{\circ}})$
with $\widetilde{W\backslash H^{\circ}}$ the universal covering \index{Covering!universal}
of $W\backslash H^{\circ}$.
Then we have the commutative diagram
 \[
\begin{CD}
\widehat{W\backslash H^{\circ}}   @>\widehat{Pev} >>     \mathbb{P}(\mathcal{F}_{S}^{*}(\kappa)) \\
@VVV                                                                                         @VVV                         \\
W\backslash H^{\circ}                       @> Pev  >>    \Gamma\backslash  \mathbb{P}(\mathcal{F}_{S}^{*}(\kappa))
\end{CD}
\]
Note that $\Gamma\backslash  \mathbb{P}(\mathcal{F}_{S}^{*})$ is an ill defined space unless the action of
$\Gamma$ on $\mathbb{P}(\mathcal{F}_{S}^{*})$ is properly discontinuous.

Recall that the Wronskian of $D_{u,v}^{\kappa}$ is defined up to a scalar multiplication as follows:
$$J:=\det(\partial_{\xi_{i}}f_{j})_{0\leq i,j \leq n}$$
where $\xi_{1},\cdots ,\xi_{n}\in \mathfrak{h}$ being an orthonormal basis and let $f_{0},\cdots,f_{n}$ be a basis of
local solutions of $D_{u,v}^{\kappa}f=0$.

\begin{lemma}
The Wronskian of $D_{u,v}^{\kappa}$ is given by:
\begin{equation*}
J=\prod_{\alpha > 0}(e^{\alpha /2}-e^{-\alpha /2})^{-2k_{\alpha}}.
\end{equation*}
\end{lemma}

\begin{proof}
First we compute
\begin{align*}
\partial_{\xi}J=&\partial_{\xi}\det
\left( \begin{array}{cccc}
f_{0} & \partial_{\xi_{1}}f_{0} & \cdots & \partial_{\xi_{n}}f_{0}\\
f_{1} & \partial_{\xi_{1}}f_{1} & \cdots & \partial_{\xi_{n}}f_{1}\\
\vdots & \vdots & \ddots & \vdots \\
f_{n} & \partial_{\xi_{1}}f_{n} & \cdots & \partial_{\xi_{n}}f_{n}\\
\end{array}
\right)\\
=&\det
\left( \begin{array}{cccc}
\partial_{\xi}f_{0} & \partial_{\xi_{1}}f_{0} & \cdots & \partial_{\xi_{n}}f_{0}\\
\partial_{\xi}f_{1} & \partial_{\xi_{1}}f_{1} & \cdots & \partial_{\xi_{n}}f_{1}\\
\vdots & \vdots & \ddots & \vdots \\
\partial_{\xi}f_{n} & \partial_{\xi_{1}}f_{n} & \cdots & \partial_{\xi_{n}}f_{n}\\
\end{array}
\right)
\\
&\qquad \qquad +
\sum_{i}\det
\left( \begin{array}{ccccc}
f_{0} & \cdots & \partial_{\xi}\partial_{\xi_{i}}f_{0} & \cdots & \partial_{\xi_{n}}f_{0}\\
f_{1} & \cdots & \partial_{\xi}\partial_{\xi_{i}}f_{1} & \cdots & \partial_{\xi_{n}}f_{1}\\
\vdots & \ddots & \vdots & \ddots & \vdots \\
f_{n} & \cdots & \partial_{\xi}\partial_{\xi_{i}}f_{n} & \cdots & \partial_{\xi_{n}}f_{n}\\
\end{array}
\right) \\
=&
\sum_{i}\det
\left( \begin{array}{ccccc}
f_{0} & \cdots & \partial_{\xi}\partial_{\xi_{i}}f_{0} & \cdots & \partial_{\xi_{n}}f_{0}\\
f_{1} & \cdots & \partial_{\xi}\partial_{\xi_{i}}f_{1} & \cdots & \partial_{\xi_{n}}f_{1}\\
\vdots & \ddots & \vdots & \ddots & \vdots \\
f_{n} & \cdots & \partial_{\xi}\partial_{\xi_{i}}f_{n} & \cdots & \partial_{\xi_{n}}f_{n}\\
\end{array}
\right)
\end{align*}
while because of ($\ref{eqn:specialeqn}$), for type $A_{n}$, we have
\begin{align*}
\partial_{\xi}\partial_{\xi_{i}}f_{j}\\
=&
-\frac{1}{2}\sum_{\alpha >0}k_{\alpha}\frac{e^{\alpha}+1}{e^{\alpha}-1}
\alpha(\xi)\alpha(\xi_{i})\partial_{\alpha^{\vee}}f_{j}-\partial_{b^{\kappa}(\xi,\xi_{i})}f_{j}-a^{\kappa}(\xi,\xi_{i})f_{j}\\
=&-\frac{1}{2}\sum_{\alpha >0}k_{\alpha}\frac{e^{\alpha}+1}{e^{\alpha}-1}
\alpha(\xi)\alpha(\xi_{i})\partial_{\alpha^{\vee}}f_{j}
  -\frac{1}{2}k'\sum_{\alpha >0}\alpha(\xi)\alpha(\xi_{i})\partial_{\alpha'}f_{j}-a^{\kappa}(\xi,\xi_{i})f_{j}\\
=&-\frac{1}{2}\sum_{\alpha >0}\sum_{i}k_{\alpha}\frac{e^{\alpha}+1}{e^{\alpha}-1}
\alpha(\xi)\alpha(\xi_{i})(\alpha^{\vee},\xi_{i})\partial_{\xi_{i}}f_{j}\\
  &-\frac{1}{2}k'\sum_{\alpha >0}
  \sum_{i}\alpha(\xi)\alpha(\xi_{i})(\alpha',\xi_{i})\partial_{\xi_{i}}f_{j}-a^{\kappa}(\xi,\xi_{i})f_{j};
\end{align*}
and then
\begin{align*}
&\partial_{\xi}J\\
=&
\sum_{i}(-\frac{1}{2}\sum_{\alpha >0}k_{\alpha}\frac{e^{\alpha}+1}{e^{\alpha}-1}
\alpha(\xi)\alpha(\xi_{i})(\alpha^{\vee},\xi_{i}))J
  +\sum_{i}(-\frac{1}{2}k'\sum_{\alpha >0}\alpha(\xi)\alpha(\xi_{i})(\alpha',\xi_{i}))J\\
=&-\sum_{\alpha >0}k_{\alpha}\frac{e^{\alpha}+1}{e^{\alpha}-1}
\alpha(\xi)J.
\end{align*}
For other types, we have
\begin{align*}
\partial_{\xi}\partial_{\xi_{i}}f_{j}&=
-\frac{1}{2}\sum_{\alpha >0}k_{\alpha}\frac{e^{\alpha}+1}{e^{\alpha}-1}
\alpha(\xi)\alpha(\xi_{i})\partial_{\alpha^{\vee}}f_{j}-a^{\kappa}(\xi,\xi_{i})f_{j}\\
&=-\frac{1}{2}\sum_{\alpha >0}k_{\alpha}\frac{e^{\alpha}+1}{e^{\alpha}-1}
\alpha(\xi)\alpha(\xi_{i})(\alpha^{\vee},\xi_{i})\partial_{\xi_{i}}f_{j}
  -a^{\kappa}(\xi,\xi_{i})f_{j};
\end{align*}
and then
\begin{align*}
\partial_{\xi}J&=
\sum_{i}(-\frac{1}{2}\sum\limits_{\alpha >0}k_{\alpha}\frac{e^{\alpha}+1}{e^{\alpha}-1}
\alpha(\xi)\alpha(\xi_{i})(\alpha^{\vee},\xi_{i}))J\\
&=-\sum_{\alpha >0}k_{\alpha}\frac{e^{\alpha}+1}{e^{\alpha}-1}
\alpha(\xi)J.
\end{align*}
Then we verify that the proposed product formula for $J$ satisfies all these formulas.
\begin{align*}
\partial_{\xi}J
=&\sum_{\alpha >0}\Big(-2k_{\alpha}(e^{\alpha /2}-e^{-\alpha /2})^{-2k_{\alpha}-1}
  \cdot (e^{\alpha /2}\cdot\frac{\alpha(\xi)}{2}+e^{-\alpha /2}\cdot\frac{\alpha(\xi)}{2})\\
  &\quad \cdot \prod_{\scriptstyle \beta\neq\alpha \atop \scriptstyle \beta >0}(e^{\beta /2}-e^{-\beta /2})^{-2k_{\beta}}\Big)\\
=&-\sum_{\alpha >0}\Big(k_{\alpha}\alpha(\xi)\frac{e^{\alpha}+1}{e^{\alpha}-1}
  \prod_{\beta >0}(e^{\beta /2}-e^{-\beta /2})^{-2k_{\beta}}\Big)\\
=&-\sum_{\alpha >0}k_{\alpha}\alpha(\xi)\frac{e^{\alpha}+1}{e^{\alpha}-1}J
\end{align*}
The lemma follows.
\end{proof}

A basis $f_{0},\cdots ,f_{n}$ of $\mathcal{F}_{S}(\kappa)$ identifies $\mathbb{P}(\mathcal{F}_{S}^{*}(\kappa))$
with $\mathbb{P}^{n}(\mathbb{C})$ by the following way
\begin{equation*}
\begin{split}
Pev(\kappa):\widehat{W\backslash H^{\circ}}&\rightarrow \mathbb{P}^{n}(\mathbb{C}) \\
q&\mapsto [f_{0}(q):f_{1}(q):\cdots :f_{n}(q)]
\end{split}
\end{equation*}

Since an irreducible component
of $\hat{H}-H^{\circ}$ is either the closure $\hat{H}_{\alpha}$ in $\hat{H}$ of
some $H_{\alpha}$ or is equal to some $D_{p}$ with $p\in\Pi$. Let $I\subset \{1,2,\cdots ,n\}$ have $m$ elements. The subset
\begin{align*}
\{h\in \hat{H}\mid e^{\alpha_{i}}(h)=c \; \text{if and only if} \; i\in I \}
\end{align*}
is called a $(n-m)$-dimensional face where $c$ takes one value of $\{0,1,\infty\}$ for each $i$.
In particular, it is a type $i$ reflection hypertorus if $I=\{ i \}$ and $c=1$ and
a type $i$ boundary divisor if $I=\{ i \}$ and $c=0$ or $\infty$. The union of all
$(n-1)$-dimensional facets is called the set of subregular points. Then we analyze the local situation near
a subregular point $x$. This will be used in proving the hyperbolic structure of $H^{\circ}$.

\begin{lemma}
For any $\kappa\in K$ the map $ev(\kappa)$ satisfies the following properties:
\leftmargini=7mm
\begin{enumerate}
\item [(1)] It maps locally biholomorphically into $\mathcal{F}_{S}^{*}(\kappa)$.

\item [(2)] Continuing $ev(\kappa)$ along a loop $\sigma\in \mathrm{Art}'(M)$ yields $\rho^{*}(\kappa,\sigma)ev(\kappa)$.
\end{enumerate}
\end{lemma}

\begin{proof}
That evaluation map $ev(\kappa)$ is locally biholomorphic everywhere since (say $f_{0}\neq 0$)
the Wronskian
$$J=f_{0}^{n+1}\det(\partial_{\xi_{i}}(f_{j}/f_{0}))_{0\leq i,j\leq n}$$
is precisely the Jacobian of the projective evaluation mapping in the affine chart $\{f_{0}\neq 0\}$.

Statement 2 is clear since $\mathcal{F}_{S}^{*}(\kappa)$ is a fiber living on the base
orbifold $W\backslash H^{\circ}$.
\end{proof}

\begin{lemma}\label{lem:strata-local-coordinates}
We can pick local coordinates $y_{1},y_{2},\cdots ,y_{n+1}$ and certain linear coordinates of
$\mathcal{F}_{S}^{*}(\kappa)$ near $x$ such that the evaluation map has the following form:
\begin{align*}
&ev(\kappa)=(y_{1}^{\frac{1}{2}-k_{\alpha}},y_{2},\cdots , y_{n+1}) \;  &&\text{if} \; x\in H_{\alpha}^{\circ} \\
&ev(\kappa)=(y_{1}^{1-k'_{p}},\cdots,y_{m}^{1-k'_{p}},y_{m+1}^{1-k''_{p}},\cdots , y_{n+1}^{1-k''_{p}}) \;  &&\text{if} \; x\in D_{p}^{\circ}.
\end{align*}
\end{lemma}

\begin{proof}
From the computation in Section $\ref{subsec:eigenvalues}$, we know that the eigenvalues of the residue map of the connection $\tilde{\nabla}$ along the
mirror are $2k_{\alpha}$ and $0$ with multiplicities $1$ and $n$ respectively while a half turn corresponds to a loop in
$W\backslash H^{\circ}$.

Similarly, the eigenvalues of the residue map of the connection $\tilde{\nabla}$ along the boundary divisor are
$k'_{p}$ and $k''_{p}$, with multiplicities $m$ and $n+1-m$ respectively say.

Then the evaluation map could be written as in the statement with respect to these coordinates.
\end{proof}

\subsection{Complex hyperbolic ball}\label{subsec:complex-ball}

Then we finally arrive at the main result of this section. Inspired by the idea of Section 3.8 in Couwenberg \cite{Couwenberg}, we
can show the following fact.

\begin{theorem}
For $\kappa\in K'_{\mathrm{hyp}}$, the image of the projective evaluation map
$$\widehat{Pev}:\widehat{W\backslash H^{\circ}}\rightarrow\mathbb{P}^{n}(\mathbb{C})$$
is contained in the ball $\mathbb{B}^{n}(\mathbb{C})$.
\end{theorem}

\begin{proof}
Let $e_{0},e_{1},\cdots ,e_{n}$ be the standard basis of $\mathcal{F}_{S}$, denote their dual sections in $\mathcal{F}_{S}^{*}$
by $e_{0}^{*},e_{1}^{*},\cdots,e_{n}^{*}$, then through the equivalence of
monodromy representation and the reflection representation we can transfer the Hermitian structure in space $\mathcal{A}^{n+1}$
to the vector bundle $\mathcal{F}_{S}^{*}$ over the restricted real valued multiplicity function $K'$
by defining $h^{*}(e^{*}_{i},e^{*}_{j})=h^{*}_{ij}$ as in ($\ref{eqn:dual-herm}$). To prove this desired result, it suffices to show that
\[ h^{*}(ev(\kappa,\cdot),ev(\kappa,\cdot))>0 \]
on $H^{\circ}$ for hyperbolic $\kappa$, i.e., for $\kappa\in K'_{\mathrm{hyp}}$.

Using the action of the Weyl group $W$ on the adjoint torus $H$ which corresponds to a complete Weyl
Chamber decompostion $\Sigma$ for the real vector space $P_{\mathbb{R}}^{\vee}=P^{\vee}\otimes\mathbb{R}$ spanned by the
coweight lattice $P^{\vee}$, a smooth full compactification $H\rightarrow \hat{H}$ added by a boundary divisor
with normal crossings could be realized. These boundary divisors are called the toric strata.
Also we could compactify $\mathbb{C}^{\times}$ by adding two points $\{ 0,\infty\}$ which actually corresponds to a type $A_{1}$
Weyl chamber decomposition.

Let $\mathcal{D}=\cup_{\alpha}H_{\alpha}$ and let
\[
\begin{CD}
\mathbb{P}^{1}\times\mathbb{C}^{\times}   @>\iota >>     \hat{H}\times\mathbb{C}^{\times}  \\
@V pr_{1} VV                                                                                         @VV pr_{1} V                         \\
\mathbb{P}^{1}                                   @> \iota^{\prime}  >>    \hat{H}
\end{CD}
\]
be a commutative diagram such that $\iota^{\prime}$ is an embedding of a projective line
in $\hat{H}$ which intersects every mirror and every irreducible toric divisor
only in subregular points. In particular, the image $\iota^{\prime}(\mathbb{C}^{\times})$ is not contained in any
reflection hypertorus or toric divisor. In fact, this desired projective line could be realized as follows: as $P^{\vee}\cong\mathbb{Z}^{n}$, any $p=\sum_{i}b_{i}p_{i}$ gives rise
to a homomorphism $\gamma_{p}:\mathbb{C}^{\times}\rightarrow H$ which sends
$\lambda\in\mathbb{C}^{\times}$ to $(\lambda^{b_{1}},\lambda^{b_{2}},\cdots,\lambda^{b_{n}})\in(\mathbb{C}^{\times})^{n}$.
We now choose an element $p_{1}=\sum_{i}b_{i}p_{i}$
with $b_{1}=1,b_{i}=0 \; \text{for} \; i=2,\cdots, n$, then we have a homomorphism
\begin{align*}
\gamma_{1}:\mathbb{C}^{\times}& \rightarrow H\\
\lambda & \mapsto (\lambda,1,\cdots, 1)
\end{align*}
If we compactify this 1-dimensional subtorus by adding two points $\{0,\infty\}$, this
induced projective line intersects $D_{p_{1}}$ and $D_{-p_{1}}$ only in subregular points.
Then we translate this projective line a little bit, i.e., multiply its coordinates by a
complex number $1+\epsilon$ with $\epsilon$ very close to $0$. By this way, we get
a projective line which intersects mirrors and toric divisors only in subregular points.

In the diagram above, let $\iota$ map the second factor unchanged and $pr_{1}$ denote the projection map to
the first factor. Let $a_{1},\cdots,a_{m}$ be the points in $\mathbb{P}^{1}$ which are mapped by $\iota^{\prime}$
into $\mathcal{D}$ and $a_{0},a_{\infty}$ be the points in $\mathbb{P}^{1}$ which are mapped by $\iota'$ into the
toric divisor. Define a real valued function $\phi$ on $K^{\prime}\times((\mathbb{P}^{1}
\setminus \lbrace a_{0},a_{1},\cdots ,a_{m},a_{\infty}\rbrace)\times\mathbb{C}^{\times})$ by:
$$\phi(\kappa,x):=h^{*}(ev(\kappa,\iota(x)),ev(\kappa,\iota(x))).$$
Here we write $x$ instead of a point $(q,t)\in \mathbb{P}^{1}\times\mathbb{C}^{\times}$.

Note that by monodromy invariance of $h^{*}$ this defines a single valued continuous function.
Then we conclude by the
characterization in Lemma $\ref{lem:strata-local-coordinates}$ that $\phi$ extends to a continuous function (also called $\phi$)
on $K^{\prime}\times(\mathbb{P}^{1}\times\mathbb{C}^{\times})$.

We now investigate if this $\phi$ can take on negative values. If we denote the
parabolic region by $K_{0}$, we
observe that $\phi(\kappa,x)>0$ for $\kappa\in K_{0}$. For parabolic $\kappa$, we always have the projection of $x$ onto the
$\mathbb{C}^{\times}$ part nonzero so that $\phi$ must be greater than $0$. Define $N$ by:
$$N:=\lbrace (\kappa,x)\in K^{\prime}\times(\mathbb{P}^{1}\times\mathbb{C}^{\times})
\mid \phi(\kappa,x)\leq 0\rbrace.$$
Then $N$ is closed by the continuity of the function $\phi$. Because $N$ is invariant under
scalar multiplication in the second factor, we have that the projection $N_{K}$ of $N$ on $K^{\prime}$
along $\mathbb{P}^{1}\times\mathbb{C}^{\times}$ is also closed.

Now suppose $\kappa\in \partial N_{K}$, then $\phi(\kappa,x)\geq 0$ otherwise $\kappa$ cannot belong to the boundary
of $N_{K}$ by the continuity of $\phi$. And we also have $\phi(\kappa,x_{0})=0$ for some
$x_{0}\in \mathbb{P}^{1}\times\mathbb{C}^{\times}$. Suppose also that $\kappa\in K'_{\mathrm{hyp}}$.
Because $ev(\kappa)$ is locally biholomorphic on $H^{\circ}\times\mathbb{C}^{\times}$ and
the image $\iota'(\mathbb{C}^{\times})$ of $\mathbb{C}^{\times}$ under $\iota'$ is not
contained in any single irreducible component of the added divisor by a previous remark, we conclude that
$\phi(\kappa,x)=0$ implies that $x\in (a_{0}\times\mathbb{C}^{\times})\cup (a_{1}\times\mathbb{C}^{\times})
\cup\cdots\cup (a_{m}\times\mathbb{C}^{\times}) \cup (a_{\infty}\times\mathbb{C}^{\times})$
by the maximal principle. Hence $\phi(\kappa,\cdot)$ vanishes along some $\mathbb{C}^{\times}$-orbit.

Either $i\in\{ 1,\cdots,m\}$ or $\{0,\infty \}$, we know that at a non zero point $x_{0}$ in $a_{i}\times\mathbb{C}^{\times}$
we can write the evaluation map $ev(\kappa,x)$ locally of the form given in Lemma $\ref{lem:strata-local-coordinates}$:
\begin{align*}
&ev(\kappa,\iota(x))=(y_{1}^{\frac{1}{2}-k_{\alpha}},y_{2},\cdots , y_{n+1}) \;  &&\text{if} \; i\in \{1,\cdots,m\} \\
&ev(\kappa,\iota(x))=(y_{1}^{1-k'_{p}},\cdots,y_{m}^{1-k'_{p}},y_{m+1}^{1-k''_{p}},\cdots , y_{n+1}^{1-k''_{p}}) \;  &&\text{if} \; i\in \{0,\infty \}.
\end{align*}
Since $\kappa$ lies in the hyperbolic region, so $\det(h^{*})$ is of the signature $(1,n)$. While $\phi(\kappa,x)\geq 0$,
then $\phi$ must possess the unique positive signature if we transform these above coordinates into the standard coordinates
$z_{i}$ for $1\leq i \leq n+1$. Since the image $\iota(\mathbb{P}^{1}\times\mathbb{C}^{\times})$ is of dimension $2$, we then
have the following formula:
\[  \phi(\kappa,x)=|z_{i}|^{2}-|z_{j}|^{2} \quad \text{for some} \; i,j\]
for $x$ near $x_{0}$. Then we know that it must take value in an open interval containing $0$ if $x$ lies in
a neighbourhood of $x_{0}$ which is
in contradiction to that $\phi(\kappa,x)\geq 0$ for $x$ takes value in a neighbourhood of $x_{0}$.

Therefore, we conclude that if $\kappa\in \partial N_{K}$ then $\kappa$ is outside of $K'_{\mathrm{hyp}}$. Since $K'_{\mathrm{hyp}}\cup K_{0}$
is connected and
not contained in $N_{K}$, we conclude that $K'_{\mathrm{hyp}}\cup K_{0}$ is disjoint from $N_{K}$ and hence $K'_{\mathrm{hyp}}$ is disjoint from
$N_{K}$ as well. This shows that
$\phi(\kappa,x)>0$ when $\kappa\in K'_{\mathrm{hyp}}$. In particular we have that on the $\iota$ image of $\mathbb{C}^{\times}\times
\mathbb{C}^{\times}$, evaluation maps into the lift of $\mathbb{B}^{n}$ in $\mathbb{C}^{n+1}$, hence
projective evaluation maps into $\mathbb{B}^{n}$.  And the desired result follows by varying the map
$\iota$ so that the images of $\iota'$ cover $H^{\circ}$.
\end{proof}

\begin{remark}
From previous computation, we can know that the monodromy along toric
strata has only two different eigenvalues. But only those strata for which one eigenvalue has
multiplicity $1$ and the other multiplicity $n$ are
mapped under the projective evaluation map to mirrors in the complex hyperbolic ball.
And in fact, monodromy around these strata acts like a complex reflection.
\end{remark}

\section{Ball quotient structures}

In this section we shall give a finite list for which the orbifold $W\backslash H^{\circ}$ could be
biholomorphically mapped onto a Heegner divisor complement of a ball quotient. This requires us to
check the Schwarz conditions for root systems, and fortunately we have already obtained all the exponents
along those mirror divisors from the preceding section. In the end, we explain a modular interpretation
for a ball quotient from type $A_{n}$ according to the Deligne-Mostow theory.
We hope this could shed some light on looking for modular interpretations for other types.

\subsection{The Schwarz conditions}

Now in order to extend geometric structures across the arrangement ``nicely'', we have to impose the so-called
$\emph{Schwarz condition}$ \index{Schwarz condition}
on this ``Dunkl-type'' system in the sense of \cite{Couwenberg-Heckman-Looijenga}.
But before we proceed to that, let us have a look at a simple example
first so that we can have some feeling why the Schwarz conditions are introduced in the way as later.
It is a one-dimensional example. Let $M$ be $\mathbb{C}^{\times}$, the toric arrangement consists of the identity and
$\Omega=k\frac{z+1}{z-1}\frac{dz}{z}\otimes dz\otimes \frac{\partial}{\partial z}$.
Assume we have finite holonomy, which means we can write $1-k=p/q$ with
$p,q$ relatively prime integers and $q>0$. Then the holonomy cover can be extended to a $q$-fold cover with
ramification over the identity $\hat{M}\rightarrow M$ defined by $(\hat{z}-1)^{q}=z-1$. On the other hand, the
projective developing map $\hat{M}\rightarrow \mathbb{P}^{1}$ is given by $w-1=(\hat{z}-1)^{p}$ and hence extends across the identity
only if $p>0$, i.e., $k<1$. But we could note that the connection is invariant under the $p$th roots of unity
$\xi_{p}$ which means the $\xi_{p}$-orbit space of $M$ is covered by the $\xi_{p}$-orbit space of $\hat{M}$ and
the projective developing map factors through the latter as a local isomorphism onto $\mathbb{P}^{1}$.
This example suggests the definition for Schwarz condition, which has also been discussed in \cite{Heckman-Looijenga-2010}.

\begin{definition}
Let be given a connected complex manifold $M$ with a smooth projective compactification $\overline{M}=M\bigsqcup D$
by adding a normal crossing divisor $D$ to it. Let the following system of second order differential equations
\[ (\partial_i\partial_j+\sum \Gamma^k_{ij}\partial_k+A_{ij})f=0 \]
defines a projective structure on $M$ with regular singularities along $D$. Then the compactification is said
to be well adapted to the given projective structure on $M$ if (on a finite cover)
the projective developing map extends locally across $D$ as a rational map.

Now assume the compactification
$\overline{M}$ is well adapted to the given projective structure on $M$, then we say the system satisfies the
Schwarz conditions if the projective developing map extends across those codimension one strata of $D$
as a local biholomorphism, as long as they are not contracted by the projective developing map.
\end{definition}

Suppose now the ``Dunkl-type'' system satisfies the Schwarz condition.
As illustrated by the previous one-dimensional example,
for $L$ an irreducible intersection, it's easy to extend the developing map across $L^{\circ}$ when $1-k_{L}>0$
where $k_{L}$ is the exponent associated to $L$.
Here we mean an irreducible intersection by it cannot be decomposed into two nontrivial intersections for which
the union of the sets of hypersurfaces containing each intersection is just the set of hypersurfaces containing
the intersection.
But when $1-k_{L}\leq 0$, the situation becomes
quite different because if we approach $L^{\circ}$ from $M$ along a curve, the image of a lift in $\widehat{M}$
under the projective developing map tends to infinity with limit a point of $\mathbb{P}(A)$.
In fact, these limit points lie in
a well-defined $\Gamma$-orbit of linear subspaces of $\mathbb{P}(A)$ of codimension $\dim(L)$, which is called a
$\emph{special subspace}$ in $\mathbb{P}(A)$. So we say that $M$ has geometric structures of
$\emph{elliptic, parabolic, hyperbolic}$ type
\index{Geometric structures of!elliptic type} \index{Geometric structures of!parabolic type}
\index{Geometric structures of!hyperbolic type}
according to whether $k_{0}<1$, $=1$ or $>1$. That is because
$k_{0}<1$ (resp. $k_{0}=1$) ensures $k_{L}<1$ for all the irreducible intersections $L$ (resp.
all the irreducible intersections except for $\{0\}$) due to the partial order of $k_{L}$'s.

While the most interesting
case is the one of hyperbolic type, we need to treat $L^{\circ}$ with $k_{L}\geq 1$ very carefully. Now we assume
the ``Dunkl-type'' system with the flat Hermitian form $h$ is of the hyperbolic type. We notice that the restriction of $h$ to
the fibers of the natural retraction $r:\overline{M}_{L^{\circ}}\rightarrow L^{\circ}$ is positive, semipositive and hyperbolic
according to whether $1-k_{L}>0$, $=0$ or $<0$. From this we see that when $k_{L}<1$ the $L^{\circ}$ still keeps
its hyperbolic type while other cases not. So for each irreducible intersection $L$ with $k_{L}\geq 1$
we need to blow it up and contract it in its own direction so that each of them has a hyperbolic structure in the end.
Of course, we need to blow up those $L$'s in an appropriate order (starting from the smallest dimension) until
the members of the arrangement become disjoint so that proper contractions can be done.
And we should point out that for those $L$'s with
$k_{L}=1$, we need to blow up each of them in a $\emph{real-oriented manner}$. The process of blowing up and
followed by contraction is a very technical tool, interested reader could consult \cite{Looijenga-2003} and
\cite{Couwenberg-Heckman-Looijenga} for a detailed and complete discussion. After we finish these operations,
we can extend the corresponding structure across the arrangement.

\subsection{Schwarz conditions for root systems}

Fortunately we have already computed the eigenvalues of the residue endomorphisms along the mirror and toric strata
in section $\ref{subsec:eigenvalues}$,
and from \cite{Couwenberg-Heckman-Looijenga} we also have the exponent near the identity element by
$k_{\{1\}}=\frac{1}{\mathrm{codim}(\{1\})}\sum_{\alpha>0}k_{H_{\alpha}}$. Then we list these exponents
of toric strata, mirror strata and identity element for all the root systems as follows.

\begin{lemma}\label{lem:relative-exponent}
The relative exponents
of toric strata, mirror strata and identity element for all the root systems are as follows.

\newcommand{\tabincell}[2]{\begin{tabular}{@{}#1@{}}#2\end{tabular}}
\begin{table}[H]
\caption{Relative exponents}
\vspace{5mm}
\begin{tabular}{|c|c|c|c|}
\hline
type & toric strata & mirror strata & identity element (after blown up) \\
\hline
$A_{n}$ & \tabincell{c}{$\frac{1}{2}(n-1)k-\frac{1}{2}(n+1)k'$ \\ $\frac{1}{2}(n-1)k+\frac{1}{2}(n+1)k'$} & $\frac{1}{2}-k$ & $((n+1)k-1)/2$ \\
\hline
$B_{n}$ & $(n-3)k+k'$, $2k-k'$ & $\frac{1}{2}-k$, $\frac{1}{2}-k'$ & $(2(n-1)k+2k'-1)/2$ \\
\hline
$C_{n}$ & $(n-3)k+2k'$, $k-k'$ & $\frac{1}{2}-k$, $\frac{1}{2}-k'$ & $(2(n-1)k+2k'-1)/2$ \\
\hline
$D_{n}$ & $(n-3)k$, $k$ & $\frac{1}{2}-k$ & $(2(n-1)k-1)/2$\\
\hline
$E_{n}$ & $k$, $2k$, $(n-4)k$ & $\frac{1}{2}-k$ & \tabincell{c}{$(hk-1)/2$ \\ with $h$ the Coxeter number} \\
\hline
$F_{4}$ & $k'$, $2k$ & $\frac{1}{2}-k$, $\frac{1}{2}-k'$ & $(6(k+k')-1)/2$ \\
\hline
$G_{2}$ & $-\frac{1}{2}k+\frac{3}{2}k'$, $\frac{1}{2}k-\frac{1}{2}k'$ & $\frac{1}{2}-k$, $\frac{1}{2}-k'$ & $(3(k+k')-1)/2$ \\
\hline
\end{tabular}
\end{table}

\end{lemma}

\begin{proof}
Straightforward, from the eigenvalues computation.
\end{proof}

In Corollary $\ref{cor:hyperbolic-domain}$ we have already obtained the hyperbolic region for which $H^{\circ}$ is endowed with
a hyperbolic structure. Within this region, only those $W\backslash H^{\circ}$ with $(k,k')$ satisfying the Schwarz conditions
could be extended to a ball quotient, by which we mean of the form $\Gamma\backslash\mathbb{B}$ with $\Gamma$ a discrete group
of $\mathrm{Aut}(\mathbb{B})$ acting on $\mathbb{B}$ with finite covolume.

The Schwarz conditions in the present case are easy to state: those relative exponents appeared in
Lemma $\ref{lem:relative-exponent}$, denoted by $relative \; exponent$, should satisfy:
\[ \emph{relative exponent} \in 1/\mathbb{N} \; \text{if} \;  >0.   \]

Write $\frac{1}{2}-k=1/p$ with $p\in\mathbb{N}$ and $p\geq 3$ for all the types and $\frac{1}{2}-k'=1/p'$
with $p'\in\mathbb{N}$ and $p'\geq 3$ for type $BCFG$.
Then we have the following ball quotients list for the root system $R$ of rank at least $2$.

\begin{table}[H]
\caption{ball quotients for toric mirror arrangement}
\vspace{5mm}
\begin{tabular}{|c|c|c|c|}
\hline
\multicolumn{4}{|c|}{type $A$} \\
\hline
$n$ & $k$ & $p$ & $k'$ \\
\hline
$2$ & $\frac{1}{6}$ & $3$ & $0$, $\pm\frac{1}{90}$, $\pm\frac{1}{54}$, $\pm\frac{1}{36}$, $\pm\frac{5}{126}$,
$\pm\frac{1}{18}$, $\pm\frac{7}{90}$, $\pm\frac{1}{9}$\\
\hline
$2$ & $\frac{1}{4}$ & $4$ & $0$, $\pm\frac{1}{36}$, $\pm\frac{1}{20}$, $\pm\frac{1}{12}$, $\pm\frac{5}{36}$ \\
\hline
$2$ & $\frac{3}{10}$ & $5$ & $\pm\frac{1}{30}$, $\pm\frac{1}{15}$, $\pm\frac{11}{90}$, $\pm\frac{7}{30}$ \\
\hline
$2$ & $\frac{1}{3}$ & $6$ & $0$, $\pm\frac{1}{18}$, $\pm\frac{1}{9}$, $\pm\frac{2}{9}$ \\
\hline
$2$ & $\frac{5}{14}$ & $7$ & $\pm\frac{13}{126}$, $\pm\frac{3}{14}$ \\
\hline
$2$ & $\frac{3}{8}$ & $8$ &  $\pm\frac{1}{24}$, $\pm\frac{7}{72}$, $\pm\frac{5}{24}$ \\
\hline
$2$ & $\frac{7}{18}$ & $9$ & $\pm\frac{5}{54}$, $\pm\frac{11}{54}$ \\
\hline
$2$ & $\frac{2}{5}$ & $10$ & $0$, $\pm\frac{4}{45}$, $\pm\frac{1}{5}$ \\
\hline
$2$ & $\frac{5}{12}$ & $12$ & $\pm\frac{1}{36}$, $\pm\frac{1}{12}$, $\pm\frac{7}{36}$ \\
\hline
$2$ & $\frac{3}{7}$ & $14$ & $\pm\frac{4}{21}$ \\
\hline
$2$ & $\frac{4}{9}$ & $18$ & $\pm\frac{2}{27}$, $\pm\frac{5}{27}$ \\
\hline
$2$ & $\frac{7}{15}$ & $30$ & $\pm\frac{8}{45}$ \\
\hline
$3$ & $\frac{1}{6}$ & $3$ & $0$, $\pm\frac{1}{24}$, $\pm\frac{1}{12}$ \\
\hline
$3$ & $\frac{1}{4}$ & $4$ & $0$, $\pm\frac{1}{24}$, $\pm\frac{1}{8}$ \\
\hline
$3$ & $\frac{3}{10}$ & $5$ & $\pm\frac{1}{10}$ \\
\hline
$3$ & $\frac{1}{3}$ & $6$ & $0$, $\pm\frac{1}{12}$ \\
\hline
$3$ & $\frac{3}{8}$ & $8$ & $\pm\frac{1}{16}$ \\
\hline
$3$ & $\frac{5}{12}$ & $12$ & $\pm\frac{1}{24}$ \\
\hline
$4$ & $\frac{1}{6}$ & $3$ & $0$, $\pm\frac{1}{30}$, $\pm\frac{1}{10}$ \\
\hline
$4$ & $\frac{1}{4}$ & $4$ & $\pm\frac{1}{20}$ \\
\hline
$4$ & $\frac{1}{3}$ & $6$ & $0$ \\
\hline
$5$ & $\frac{1}{6}$ & $3$ & $0$, $\pm\frac{1}{18}$ \\
\hline
$5$ & $\frac{1}{4}$ & $4$ & $0$ \\
\hline
$6$ & $\frac{1}{6}$ & $3$ & $\pm\frac{1}{42}$ \\
\hline
$7$ & $\frac{1}{6}$ & $3$ & $0$ \\
\hline
$9$ & $\frac{1}{6}$ & $3$ & $\pm\frac{1}{15}$ \\
\hline
\end{tabular}
\end{table}

\begin{table}[H]
\begin{tabular}{|c|c|c|c|c|}
\hline
\multicolumn{5}{|c|}{type $B$} \\
\hline
$n$ & $k$ & $p$ & $k'$ & $p'$ \\
\hline
$2$ & $\frac{1}{6}$ & $3$ & $\frac{1}{6}$, $\frac{1}{4}$, $\frac{1}{3}$, $\frac{5}{12}$
& $3$, $4$, $6$, $12$ \\
\hline
$2$ & $\frac{1}{4}$ & $4$ & $\frac{1}{6}$, $\frac{1}{4}$, $\frac{3}{10}$, $\frac{1}{3}$, $\frac{3}{8}$,
$\frac{5}{12}$, $\frac{9}{20}$
& $3$, $4$, $5$, $6$, $8$, $12$, $20$ \\
\hline
$2$ & $\frac{3}{10}$ & $5$ & $\frac{2}{5}$ & $10$ \\
\hline
$2$ & $\frac{1}{3}$ & $6$ & $\frac{1}{6}$, $\frac{1}{3}$, $\frac{5}{12}$
& $3$, $6$, $12$ \\
\hline
$2$ & $\frac{3}{8}$ & $8$ & $\frac{1}{4}$ & $4$ \\
\hline
$2$ & $\frac{7}{18}$ & $9$ & $\frac{4}{9}$ & $18$ \\
\hline
$2$ & $\frac{2}{5}$ & $10$ & $\frac{3}{10}$ & $5$ \\
\hline
$2$ & $\frac{5}{12}$ & $12$ & $\frac{1}{3}$ & $6$ \\
\hline
$2$ & $\frac{4}{9}$ & $18$ & $\frac{7}{18}$ & $9$ \\
\hline
$3$ & $\frac{1}{6}$ & $3$ & $\frac{1}{6}$, $\frac{1}{4}$, $\frac{1}{3}$ & $3$, $4$, $6$ \\
\hline
$3$ & $\frac{1}{4}$ & $4$ & $\frac{1}{6}$, $\frac{1}{4}$, $\frac{1}{3}$ & $3$, $4$, $6$ \\
\hline
$3$ & $\frac{1}{3}$ & $6$ & $\frac{1}{6}$, $\frac{1}{3}$ & $3$, $6$ \\
\hline
$3$ & $\frac{3}{8}$ & $8$ & $\frac{1}{4}$ & $4$ \\
\hline
$4$ & $\frac{1}{6}$ & $3$ & $\frac{1}{6}$, $\frac{1}{3}$ & $3$, $6$ \\
\hline
$4$ & $\frac{1}{4}$ & $4$ & $\frac{1}{4}$ & $4$ \\
\hline
$5$ & $\frac{1}{6}$ & $3$ & $\frac{1}{6}$ & $3$ \\
\hline
\end{tabular}
\end{table}

\begin{table}[H]
\begin{tabular}{|c|c|c|c|c|}
\hline
\multicolumn{5}{|c|}{type $C$} \\
\hline
$n$ & $k$ & $p$ & $k'$ & $p'$ \\
\hline
$2$ & $\frac{1}{6}$ & $3$ & $\frac{1}{6}$, $\frac{1}{4}$, $\frac{1}{3}$
& $3$, $4$, $6$ \\
\hline
$2$ & $\frac{1}{4}$ & $4$ & $\frac{1}{6}$, $\frac{1}{4}$, $\frac{3}{8}$
& $3$, $4$, $8$ \\
\hline
$2$ & $\frac{3}{10}$ & $5$ & $\frac{1}{4}$, $\frac{2}{5}$ & $4$, $10$ \\
\hline
$2$ & $\frac{1}{3}$ & $6$ & $\frac{1}{6}$, $\frac{1}{4}$, $\frac{1}{3}$, $\frac{5}{12}$
& $3$, $4$, $6$, $12$ \\
\hline
$2$ & $\frac{3}{8}$ & $8$ & $\frac{1}{4}$ & $4$ \\
\hline
$2$ & $\frac{7}{18}$ & $9$ & $\frac{4}{9}$ & $18$ \\
\hline
$2$ & $\frac{2}{5}$ & $10$ & $\frac{3}{10}$ & $5$ \\
\hline
$2$ & $\frac{5}{12}$ & $12$ & $\frac{1}{6}$, $\frac{1}{4}$, $\frac{1}{3}$ & $3$, $4$, $6$ \\
\hline
$2$ & $\frac{4}{9}$ & $18$ & $\frac{7}{18}$ & $9$ \\
\hline
$2$ & $\frac{9}{20}$ & $20$ & $\frac{1}{4}$ & $4$ \\
\hline
$3$ & $\frac{1}{6}$ & $3$ & $\frac{1}{6}$, $\frac{1}{4}$ & $3$, $4$ \\
\hline
$3$ & $\frac{1}{4}$ & $4$ & $\frac{1}{6}$, $\frac{1}{4}$ & $3$, $4$ \\
\hline
$3$ & $\frac{1}{3}$ & $6$ & $\frac{1}{6}$ & $3$ \\
\hline
$3$ & $\frac{3}{8}$ & $8$ & $\frac{1}{4}$ & $4$ \\
\hline
$3$ & $\frac{5}{12}$ & $12$ & $\frac{1}{6}$ & $3$ \\
\hline
$4$ & $\frac{1}{6}$ & $3$ & $\frac{1}{6}$ & $3$ \\
\hline
$5$ & $\frac{1}{6}$ & $3$ & $\frac{1}{3}$ & $6$ \\
\hline
\end{tabular}
\end{table}

\begin{table}[H]
\begin{tabular}{|c|c|c|}
\hline
\multicolumn{3}{|c|}{type $D$} \\
\hline
$n$ & $k$ & $p$  \\
\hline
$4$ & $\frac{1}{6}$, $\frac{1}{4}$, $\frac{1}{2}$ & $3$, $4$, $6$ \\
\hline
$5$ & $\frac{1}{6}$, $\frac{1}{4}$ & $3$, $4$ \\
\hline
$6$ & $\frac{1}{6}$ & $3$ \\
\hline
\end{tabular}
\end{table}

\begin{table}[H]
\begin{tabular}{|c|c|c|}
\hline
\multicolumn{3}{|c|}{type $E$} \\
\hline
$n$ & $k$ & $p$  \\
\hline
$6$ & $\frac{1}{6}$, $\frac{1}{4}$ & $3$, $4$ \\
\hline
$7$ & $\frac{1}{6}$ & $3$ \\
\hline
\end{tabular}
\end{table}

\begin{table}[H]
\begin{tabular}{|c|c|c|c|c|}
\hline
\multicolumn{5}{|c|}{type $F$} \\
\hline
$n$ & $k$ & $p$ & $k'$ & $p'$ \\
\hline
$4$ & $\frac{1}{6}$ & $3$ & $\frac{1}{6}$, $\frac{1}{3}$ & $3$, $6$ \\
\hline
$4$ & $\frac{1}{4}$ & $4$ & $\frac{1}{4}$ & $4$ \\
\hline
\end{tabular}
\end{table}

\begin{table}[H]\label{tab:list}
\begin{tabular}{|c|c|c|c|c|}
\hline
\multicolumn{5}{|c|}{type $G$} \\
\hline
$n$ & $k$ & $p$ & $k'$ & $p'$ \\
\hline
$2$ & $\frac{1}{6}$ & $3$ & $\frac{1}{6}$, $\frac{7}{18}$ & $3$, $9$ \\
\hline
$2$ & $\frac{1}{4}$ & $4$ & $\frac{1}{6}$, $\frac{1}{4}$, $\frac{5}{12}$
& $3$, $4$, $12$ \\
\hline
$2$ & $\frac{3}{10}$ & $5$ & $\frac{1}{6}$ & $3$ \\
\hline
$2$ & $\frac{1}{3}$ & $6$ & $\frac{1}{6}$, $\frac{1}{3}$ & $3$, $6$ \\
\hline
$2$ & $\frac{7}{18}$ & $9$ & $\frac{1}{6}$ & $3$ \\
\hline
$2$ & $\frac{5}{12}$ & $12$ & $\frac{1}{4}$ & $4$ \\
\hline
\end{tabular}
\end{table}

\begin{remark}
Firstly, there might be overlaps in the above list among type $A$, $B$ and $C$, because type $B_{n}$ and type $C_{n}$ are
dual to each other, and a type $B_{n}$ arrangement can be obtained from a reduction of a higher rank $A_{m}$ arrangement
to a subtorus.

Secondly, however, besides the above list there actually are more ball quotient structures
due to some \emph{hidden symmetry}. For instance
when $k=\frac{1}{6}$ for type $E_{8}$, there is also a ball quotient structure, which can be seen from a point of
view of moduli space of rational elliptic surfaces with certain markings,
see \cite{Looijenga-2008} and \cite{Couwenberg-Heckman-Looijenga-CM} for details.
This hidden symmetry does not come from the root system perspective. Nevertheless, a description for this hidden symmetry in
$E_{8}$ root system still remains unclear.
\end{remark}

\subsection{A modular interpretation}\label{subsec:modular-interpretation}

From last section, we see that under suitable
Schwarz conditions, \index{Schwarz condition}
we can find a ball quotient structure for our space $W\backslash H^{\circ}$ as follows
\[ Pev:W\backslash H^{\circ}\rightarrow \Gamma\backslash\mathbb{B} \]
with $\Gamma$ a discrete subgroup of $\mathrm{Aut}(\mathbb{B})$ with finite covolume.
In particular, we even wish to find a modular interpretation for these
(potential) ball quotients. Namely, does there exist such a commutative diagram
\[
\begin{CD}
W\backslash H^{\circ}   @>Pev >>     \Gamma\backslash\mathbb{B} \\
@VVV                                     @VVV                         \\
\mathcal{M}               @> Per  >>    \Gamma'\backslash\mathbb{B}
\end{CD}
\]
that $\mathcal{M}$ is a suitable moduli space with $Per$ a suitable period map.
For $R$ of type $A_{n}$, the answer is already given by the theory of
Deligne-Mostow. In fact, the classical root system is just a special case of
the theory of Deligne-Mostow with all the weights being equal. We also encounter
the moduli space of del Pezzo surfaces when we look at the type $E_{n}$. However,
for the other root systems, we barely have any idea about them for the moment. We shall investigate the $A_{n}$ case here
since this example as well as type $E_{n}$ cases strongly motivated current research presented in this paper.

\begin{example}\label{exm:ball-quotient}
For the root system $R$ of type $A_{n}$, we impose a condition $k'=0$ for simplicity.
Let be given $n+3$ pairwise distinct
points $z_{0},\cdots,z_{n+2}$ on the projective line $\mathbb{P}^{1}$ and $n+3$
associated rational numbers $\mu_{0},\cdots,\mu_{n+2}\in (0,1)$ with $\sum\mu_{i}=2$.
Fix $z_{0}=0$ and $z_{n+2}=\infty$, if we denote the simply connected torus by $H'$,
then $H'^{\circ}$ can be defined as
\begin{multline*}
H'^{\circ}=\{(z_{1},\cdots,z_{n+1})\in(\mathbb{C}^{\times})^{n+1}\mid
z_{1}\cdots z_{n+1}=1,z_{i}\neq z_{j} \; \\
\text{for each pair of distinct} (i,j)\}.
\end{multline*}
And then the adjoint torus
\[ H^{\circ}=C_{n+1}\backslash H'^{\circ} \]
can be defined with $C_{n+1}=P^{\vee}/Q^{\vee}$
the cyclic group of order $n+1$. Let $\mathcal{M}_{0,n+3}$ denote the moduli space of
genus $0$ curve with $n+3$ marked points. Write $\mu_{i}=m_{i}/m$ with $m$ being their
smallest denominator, consider the algebraic curve $C(z)$ defined by the
affine equation:
\[ C(z): y^{m}=\prod(\zeta-z_{i})^{m_{i}}.  \]
Then the periods of the cyclic cover of $\mathbb{C}$
\[\int_{z_{i}}^{z_{i+1}}\frac{d\zeta}{y}   \]
are just solutions of the Lauricella $F_{D}$ hypergeometric equations.
If we take $\mu_{i}=k$ for $i=1,\cdots,n+1$ and the remaining
$\mu_{0}=\mu_{n+2}=(2-(n+1)k)/2$ so
that it becomes our special hypergeometric system associated with
the root system $A_{n}$. Let $\mathfrak{S}_{\mu}$
denote the subgroup of the symmetric group $\mathfrak{S}_{n+3}$ fixing
$\mu=(\mu_{0},\cdots,\mu_{n+2})$.
The half integrality condition \index{Half integrality condition}
from the theory of
Deligne-Mostow is given as follows:
\begin{align*}
\mu_{i}+\mu_{j}<1\Rightarrow (1-\mu_{i}-\mu_{j}) \in
\begin{cases}
&1/\mathbb{N} \quad \text{if} \quad \mu_{i}\neq \mu_{j}\\
&2/\mathbb{N} \quad \text{if} \quad \mu_{i}=\mu_{j}\\
\end{cases}
\; \text{for all} \; i\neq j.
\end{align*}
This happens to coincide with the Schwarz conditions for the special hypergeometric system with type
$A_{n}$ along the toric strata, along the mirrors and near the identity
element:
\begin{align*}
(n-1)k/2=(1-\mu_{0}-\mu_{1})\in 1/\mathbb{N}\\
(1-2k)/2=(1-\mu_{1}-\mu_{n+1})/2\in 1/\mathbb{N}\\
((n+1)k-1)/2=(1-\mu_{0}-\mu_{n+2})/2 \in 1/\mathbb{N}.
\end{align*}
If these conditions are satisfied, then we have a commutative diagram
\[
\begin{CD}
W\backslash H^{\circ}   @>Pev >>     \Gamma_{n}\backslash\mathbb{B}^{n} \\
@VVV                                     @VVV                         \\
\mathfrak{S}_{\mu}\backslash \mathcal{M}_{0,n+3}  @> Per  >>    \Gamma'_{n}\backslash\mathbb{B}^{n}
\end{CD}
\]
with left vertical arrow a covering map
and top arrow being an isomorphism
onto a Heegner divisor complement.
\end{example}

\end{document}